\documentclass[a4paper,11pt]{amsart}

\usepackage{amssymb, amsthm, amsmath}
\usepackage{fancybox}
\usepackage{color}
\usepackage{version}    
\usepackage{bm}
\usepackage{graphicx}

\newtheorem{definition}{\bf Definition}[section]
\newtheorem{theorem}[definition]{\bf Theorem}
\newtheorem{lemma}[definition]{\bf Lemma}

\newtheorem{proposition}[definition]{\bf Proposition}
\newtheorem{corollary}[definition]{\bf Corollary}
\newtheorem{remark}[definition]{\bf Remark}

\newcommand{\vep}{\varepsilon}
\renewcommand{\d}{\text{\rm d}}
\newcommand{\e}{\mathrm{e}}

\newcommand{\R}{\mathbb R}

\newcommand{\lam}{\lambda}

\DeclareMathOperator*{\esssup}{ess\,sup}

\newcommand{\iI}{I_{[0,+\infty)}}
\newcommand{\iII}{I_{[u_0(x),+\infty)}}
\newcommand{\iIo}{I_{[\uo,+\infty)}}

\newcommand{\iIIal}{I_{[\alpha,+\infty)}}
\newcommand{\uo}{\alpha}
\newcommand{\uol}{\alpha_\mu}
\newcommand{\uon}{\alpha_n}

\makeatletter

\@addtoreset{equation}{section}

\begin{document}

\title[Traveling waves for irreversible Allen-Cahn equations]{Traveling wave dynamics for Allen-Cahn equations with strong irreversibility}
\author{Goro Akagi}
\address{Goro Akagi: Mathematical Institute, Tohoku University, 6-3 Aoba, Aramaki, Aoba-ku, Sendai 980-8578 Japan}
\email{goro.akagi@tohoku.ac.jp}
\author{Christian Kuehn}
\address{Christian Kuehn: Technical University of Munich, Department of Mathematics, Boltzmannstra\ss e 3, D-85748 Garching bei M\"unchen, Germany.}
\email{ckuehn@ma.tum.de}
\author{Ken-Ichi Nakamura}
\address{Ken-Ichi Nakamura: Institute of Science and Engineering, Kanazawa University, Kakuma-machi, Kanazawa-shi, Ishikawa 920-1192 Japan}
\email{k-nakamura@se.kanazawa-u.ac.jp}
\date{\today}


\begin{abstract} 
Constrained gradient flows are studied in fracture mechanics to describe \emph{strongly irreversible} (or \emph{unidirectional}) evolution of cracks. The present paper is devoted to a study on the long-time behavior of non-compact orbits of such constrained gradient flows. More precisely, traveling wave dynamics for a one-dimensional fully nonlinear Allen-Cahn type equation involving the positive-part function is considered. Main results of the paper consist of a construction of a one-parameter family of \emph{degenerate} traveling wave solutions (even identified when coinciding up to translation) and exponential stability of such traveling wave solutions with some basin of attraction, although they are unstable in a usual sense.
\end{abstract}


\maketitle


\section{Introduction}\label{S:I}
 
The Allen-Cahn~\cite{Allen+Cahn:1979} equation (or Ginzburg-Landau equation, or Nagumo~\cite{Nagumo} equation) is the $L^2$ gradient flow of the Ginzburg-Landau free energy functional and often used in phase-field models for describing various separation processes of two different phases, e.g., in binary alloys. More precisely, evolution of an order-parameter $u(x,t)$ is described by the $L^2$ gradient flow of a free energy, and then, $u(x,t)$ eventually gets close to two different values $a_\pm$, which correspond to two different phases and may minimize (a part of) the free energy, on most of the domain as $t \to +\infty$. The traveling wave dynamics for the classical Allen-Cahn equation is well studied; see e.g.~\cite{Aronson+Weinberger:1974,XChen,Kuehn19,Sandstede2002,Volpert+etal:1994} and references therein. Yet, even minor variations lead to very challenging traveling wave problems. These variations have emerged in many different research areas recently including e.g.~fractional waves~\cite{AcKu,Chmaj:2013} and stochastic waves~\cite{HamsterHupkes,KuehnSPDEwaves}. In this work, we study a variation of the Allen-Cahn equation arising from its motivation as a phase field model in fracture mechanics. The idea of phase-field models can already be found in the study of brittle fracture, which was initiated by Griffith in 1920s and then also developed from mathematical points of view by many authors. In classical Griffith's criteria, crack \emph{irreversibly} extends, provided that the release rate of the elastic energy exceeds some critical value due to a (virtual) development of the crack. In a variational model proposed by Francfort and Marigo~\cite{FraMar98}, an energy functional $\mathcal F$ is defined for crack (sets) and displacement fields as the sum of the elastic energy of material and the surface energy of crack (see also~\cite{Francfort}). The Francfort-Marigo energy $\mathcal F$ is similar to the so-called Mumford-Shah energy, which has been used in image segmentation and can be regularized with the aid of an order-parameter (see~\cite{AmTo90,AmTo92}). In the same spirit, introducing an order-parameter $u(x,t)$ which may describe the degree of damage (e.g., $u=1$ for completely damaged parts and $u=0$ for sound parts), one can also regularize the Francfort-Marigo energy $\mathcal F$ as
$$
\mathcal F_\vep(u,v) = \underbrace{\frac 1 2 \int_\Omega (1-u)^2 |\nabla v|^2 \, \d x}_{\text{elastic energy}} - \int_\Omega g v \, \d x + \underbrace{\int_\Omega \left( \frac \vep 2 |\nabla u|^2 + \frac 1 {2\vep} u^2 \right) \, \d x}_{\text{surface energy of crack}}
$$
where $\vep > 0$ is a regularization parameter, $v = v(x,t)$ denotes the displacement field and $g = g(x,t)$ is an external force (see Giacomini~\cite{Giacom05} and references therein). 
Then the evolution of the order-parameter is formulated as a gradient flow of the regularized energy $\mathcal F_\vep$. However, due to the irreversible nature of fracture phenomena, one may consider a \emph{constrained} gradient flow, instead of standard ones, such as
\begin{equation}\label{mod-gf}
\partial_t u = \big( - \partial_u \mathcal F_\vep(u,v) \big)_+ 
\end{equation}
whose solutions comply with the constraint $\partial_t u \geq 0$, i.e., the evolution of damage is strongly irreversible (or unidirectional). 
On the other hand, such a stabilization may conflict with the constraint, and then, the constraint prevents the stabilization. We may expect that the balance between such two different effects becomes obvious in the long-time behavior of solutions, and therefore, in this paper, we shall make an attempt to investigate long-time behaviors of such constrained gradient flows. To this end, we shall particularly study a \emph{constrained Allen-Cahn equation}, which is simpler than the brittle fracture model but still possesses a gradient flow structure as well as the irreversibility constraint.

Let us consider the Cauchy problem for the one-dimensional Allen-Cahn equation with the positive-part function,
\begin{equation}\label{irAC}
 \partial_t u = \big( \partial_x^2 u - f(u) \big)_+ \ \mbox{ in } \R \times (0,+\infty),
  \quad u|_{t = 0} = u_0 \ \mbox{ in } \R,
\end{equation}
where $(\,\cdot\,)_+ = \max \{\cdot,0\} \geq 0$ and $f$ is a function in $\R$ satisfying a typical assumption for \emph{bi-stable} nonlinearity,
\begin{equation}\label{f-1}
\begin{cases}
f(a_\pm) = f(a_0) = 0,\quad f'(a_\pm) > 0,\\ 
f > 0 \ \mbox{ in } (a_-,a_0),\quad f < 0 \ \mbox{ in } (a_0,a_+)
\end{cases}
\end{equation}
for some $-\infty < a_- < a_0 < a_+ < +\infty$ and
\begin{equation}\label{f-2}
f \in C^2(\R), 
\quad \hat f \geq 0, \quad f' \geq - \lam
\end{equation}
for some primitive function $\hat f$ of $f$ and a constant $\lambda > 0$, that is, any primitive function of $f$ is bounded from below on $\R$ and there exists a monotone function $\beta : \R \to \R$ such that
\begin{equation}\label{f-lam-convex}
f(u) = \beta(u) - \lambda u.
\end{equation}
A typical example of $f$ complying with \eqref{f-1} and \eqref{f-2} is $f(u) = u^3 - u$ (and then, $a_\pm = \pm 1$, $a_0 = 0$, $\beta(u)=u^3$ and $\lam =1$).

The Cauchy-Dirichlet problem for \eqref{irAC} posed on any smooth bounded domain (in $\R^N$) has been first studied in~\cite{ae1}, where fundamental issues (e.g., well-posedness and regularity), characteristic features of the equation such as \emph{partial} smoothing  effect  and \emph{partial} energy-dissipation, and long-time behaviors of solutions are discussed. As in~\cite{ae1}, the fully nonlinear PDE of \eqref{irAC} (not in a divergence form) is equivalently rewritten in the following \emph{doubly-nonlinear} form:
\begin{equation}\label{irAC-dn}
 \partial_t u + \eta = \partial_x^2 u - f(u), \quad \eta \in \partial \iI(\partial_t u),
\end{equation}
where $\partial \iI$ denotes the subdifferential of the indicator function supported over the half-line $[0,+\infty)$ (see \eqref{subdiff} below for definition). Indeed, \eqref{irAC-dn} can be derived by applying the inverse mapping $\alpha(\cdot) : \R \to 2^{\R}$ of the positive-part function $(\cdot)_+$ given by
$$
\alpha(s) = s + \partial \iI(s) \quad \mbox{ for } \ s \geq 0
$$
to both sides of \eqref{irAC}. Equation \eqref{irAC-dn} is now in a divergence form and hence better suited for applying energy methods than the fully nonlinear form \eqref{irAC}. Moreover, \eqref{irAC-dn} can be regarded as a \emph{constrained} gradient flow of the free energy
$$
\mathcal F(u) := \frac 1 2 \int_\R |\partial_x u|^2 \, \d x + \int_\R \hat f(u) \, \d x
$$
(under a certain energy framework) and $\eta$ plays a role of a Lagrange multiplier to modify the pure gradient flow of the energy $\mathcal F$ to satisfy the constraint condition $\partial_t u \geq 0$. Furthermore, by comparison between \eqref{irAC} and \eqref{irAC-dn}, one can also find the relation,
$$
\eta = -\big( \partial_x^2 u - f(u) \big)_-,
$$
where $(s)_- := \max \{ - s, 0\} \geq 0$ for $s \in \R$. Another finding of the paper~\cite{ae1} is a  non-trivial  reformulation of \eqref{irAC} as the following obstacle problem:
\begin{equation}\label{pde-obs0}
 \partial_t u + \eta = \partial_x^2 u - f(u), \quad \eta \in \partial \iII(u),
\end{equation}
where $\iII$ is the indicator function supported over the interval $[u_0(x),+\infty)$, and the inclusion above is also equivalent (e.g., in the $L^2$ setting) to the relation,
$$
\min \left\{ u - u_0 , \partial_t u - \partial_x^2 u + f(u) \right\} = 0.
$$
It is noteworthy that the obstacle function above coincides with the initial datum, and a similar obstacle problem is also studied as a model of American option evaluation (see~\cite{LauSal09,CafFig13}). In~\cite{ae1}, these equivalent forms of the equation indeed play crucial roles for analysis, and moreover, they will do so in the present paper as well.

 In~\cite{ae1}, convergence of each solution to a single equilibrium is also proved for (pre)compact orbits of solutions to \eqref{irAC}, and moreover, in~\cite{ae2}, stability analysis of equilibria is carried out. On the other hand, asymptotic analysis on non-compact orbits of solutions to \eqref{irAC} has never been studied.  In the present paper, we restrict ourselves to \emph{traveling wave solutions} for \eqref{irAC}, that is,
\begin{equation}\label{u-phi}
u(x,t) = \phi(x - x_0 - ct)
\end{equation}
with $x_0 \in \R$, a velocity constant $c$ and a monotone profile function $\phi$. Traveling wave dynamics has been a topic of great interest in the field of reaction-diffusion equations such as the Allen-Cahn equation,
\begin{equation}\label{cac}
\partial_t u - \partial_x^2 u + f(u) = 0 \ \mbox{ in } \R \times \R_+.
\end{equation}
Substituting \eqref{u-phi} to the equation above, one can derive
 \begin{equation}\label{phi-cac}
  - \phi'' + f(\phi) = c\phi' \ \mbox{ in } \R,
 \end{equation}
which is often called a \emph{profile equation}. Traveling waves were studied in celebrated papers~\cite{Fisher} and~\cite{KPP} for the so-called Fisher-KPP equation whose reaction term is classified as a \emph{mono-stable} reaction, and then, a bi-stable reaction such as \eqref{cac} was also studied in~\cite{Kanel}. Thereafter there have been made many contributions devoted to this field. In particular, existence, uniqueness (up to translation) as well as exponential stability of traveling wave solutions are proved for bi-stable reactions in~\cite{FM}. The existence and uniqueness are proved through the use of phase-plane analysis, and the proof of the exponential stability relies on a classical parabolic regularity theory, a sub- and supersolution method and the theory of Dynamical Systems (e.g., $\omega$-limit set and linearization technique). In~\cite{XChen}, the stability result is extended to nonlocal reaction diffusion equations by developing a sub- and supersolution method and by excluding the other ingredients of the proof in~\cite{FM}. Beyond phase plane methods and sub-/super-solutions, several other methods to construct traveling waves and to prove their stability exist. We refer the interested reader to~\cite{Aronson+Weinberger:1974,Kuehn19,Sandstede2002,Volpert+etal:1994} and references therein.

In order to construct a traveling wave solution for \eqref{irAC}, one may simply substitute $u(x,t) = \phi(x-ct)$ to \eqref{irAC}, and then, the following profile equation is derived:
\begin{equation}\label{phi2}
 - c \phi' = \big( \phi'' - f(\phi) \big)_+ \ \mbox{ in } \R,
\end{equation}
which is however severely nonlinear. Here, instead of \eqref{phi2}, we introduce the following auxiliary obstacle problem as an alternative profile equation:
\begin{equation}\label{phi}
 - c \phi' - \phi'' + \partial \iIo(\phi) + f(\phi) \ni 0 \ \mbox{ in } \R
\end{equation}
for some $\alpha \in (a_-,a_0)$ (cf.~\eqref{alpha-hyp} below). Here $\partial \iIo$ denotes the subdifferential operator (with domain $D(\partial \iIo) = [\uo,+\infty)$) defined by
\begin{equation}\label{subdiff}
\partial \iIo(s) = \begin{cases}
		    \{0\} &\mbox{ if } s > \uo,\\
		    (-\infty,0] &\mbox{ if } s = \uo,\\
		    \emptyset &\mbox{ if } s < \uo
		   \end{cases}
\end{equation}
of the indicator function $\iIo$ supported over $[\uo,+\infty)$, that is,
$$
\iIo(s) = \begin{cases}
0 &\mbox{ if } s \geq \uo,\\
+\infty &\mbox{ otherwise.}
\end{cases}
$$
We further impose the asymptotic conditions,
\begin{equation}\label{phi-bc}
 \lim_{\xi \to - \infty} \phi(\xi) = \uo, \quad  \lim_{\xi \to +\infty} \phi(\xi) = a_+, \quad \phi' \geq 0 \ \mbox{ in } \R
\end{equation}
and
\begin{equation}\label{phi-bc2}
 \lim_{\xi \to \pm \infty} \phi'(\xi) = 0,
\end{equation}
on the profile $\phi$. Indeed, by \eqref{f-1}, all $u \equiv \alpha \in (a_-,a_0)$ are constant steady-states of \eqref{irAC}, even though they are not critical points of the free energy. Hence we shall construct an entire solution to \eqref{irAC} (in the form \eqref{u-phi}) which is a heteroclinic orbit connecting $u(\cdot,t) \equiv \alpha \in (a_-,a_0)$ at $t=-\infty$ and $u(\cdot,t) \equiv a_+$ at $t=+\infty$. Roughly speaking, $u_0 \equiv \alpha$ can be regarded as an ``initial datum at $t = -\infty$'', and hence, due to the reformulation of \eqref{irAC} in terms of the obstacle problem, the entire solution $u(x,t)$ may also solve \eqref{pde-obs0} with the obstacle $u_0 \equiv \alpha$ for $x, t \in \R$. Then \eqref{phi} can be derived by substituting \eqref{u-phi} to \eqref{pde-obs0}.

Our result reads,
\begin{theorem}[Existence of traveling wave solutions for \eqref{irAC}]\label{T:TW}
Assume that \eqref{f-1} with $f \in C^1(\R)$ holds. Let $\uo$ be a constant satisfying
\begin{equation}\label{alpha-hyp}
 \alpha \in (a_-,a_0), \quad \int^{a_+}_\alpha f(z) \, \d z < 0.
\end{equation}
Then the strongly irreversible Allen-Cahn equation \eqref{irAC} admits a traveling wave solution $u(x,t) = \phi(x-ct)$ for some $c \in \R$. More precisely, there exists $c$ such that \eqref{phi2} possesses a solution $\phi$ satisfying \eqref{phi-bc} and \eqref{phi-bc2}. Moreover, the following conditions hold true\/{\rm :}
 \begin{enumerate}
  \item $\phi \in W^{2,\infty}(\R)$, $\phi' \in L^2(\R)$, $\phi$ solves \eqref{phi} in $\R$, 
  \item $\uo \leq \phi < a_+$ and $0 \leq \phi' \leq C_1$ in $\R$, $-C_2 \leq c < 0$,
  \item if $\phi(s) > \uo$ at $s \in \R$, then $\phi'(s) > 0$
 \end{enumerate}
 for some constants $C_1,C_2 > 0$ depending only on $\uo$ and $f$. Furthermore, there exists $s_* \in \R$ such that
 \begin{equation}\label{phi-deg}
 \phi(s) = \uo \ \mbox{ on } (-\infty,s_*], \quad \phi(s) > \uo \ \mbox{ in } (s_*,+\infty).
 \end{equation}
In particular, $\phi \in C^2((s_*,+\infty))$, $\phi' > 0$ in $(s_*,+\infty)$ and $\phi$ solves \eqref{phi-cac} in $(s_*,+\infty)$.
\end{theorem}

We here emphasize that a conflict between stabilization and the irreversible constraint become apparent; indeed, the constrained Allen-Cahn equation \eqref{irAC} admits a one-parameter family of \emph{degenerate} traveling wave solutions (even identified when coinciding up to translation), although the traveling wave solution for the classical one \eqref{cac} can be uniquely (up to translation) determined (under \eqref{f-1}) and it is strictly monotone. Furthermore, the velocity of traveling waves for \eqref{irAC} can be negative, even though the potential $\hat{f}$ is balanced, i.e., $\hat f(a_+) = \hat f(a_-)$, under which the velocity of the traveling wave for \eqref{cac} must be zero.

\begin{remark}
{\rm
\begin{enumerate}
 \item As for the case $\uo = a_-$, any solution $\phi$ of \eqref{phi} with $\uo=a_-$ solves \eqref{phi-cac} in $\R$, i.e., it is also the traveling wave (profile) for \eqref{cac}. Then $\phi$ is strictly monotone. 
 \item We emphasize that $\phi \not\in C^2(\R)$ under the assumption \eqref{alpha-hyp}. Indeed, one can observe that $\phi$ loses $C^2$ regularity at $s_*$ as follows:
$$
\lim_{s \nearrow s_*}\phi''(s) = 0 \ \mbox{ and } \ \lim_{s \searrow s_*}\phi''(s) = f(\alpha) > 0.
$$
On the other hand, $\phi$ is smooth elsewhere.
\end{enumerate}
}
\end{remark}

One can further verify that
\begin{proposition}[Unique determination of velocity and profile]\label{P:uni-det}
For each $\alpha$ satisfying the assumptions of Theorem \ref{T:TW},
\begin{enumerate}
 \item the velocity $c_\alpha$ is uniquely determined\/{\rm ;}
 \item the profile function $\phi_\alpha$ is uniquely {\rm (}up to translation{\rm )} determined.
\end{enumerate}
 Moreover, $c_\alpha$ is strictly decreasing in $\alpha \in (a_-,a_0)$, and furthermore, the map $\alpha \mapsto (c_\alpha,\phi_\alpha)$ is continuous from $(a_-,a_0)$ to $\R\times C(\R)$. 
\end{proposition}
\noindent
We refer the reader to Proposition \ref{P:tw-unique}  and Corollaries \ref{C:c-order} and \ref{C:conti-cphi} in the next section  for more details. In what follows, we denote by $c_\alpha$ and $\phi_\alpha$ the velocity and unique profile function for each $\alpha$, and moreover, we  normalize  that $s_* = 0$ in \eqref{phi-deg}, that is,
\begin{equation}\label{phi-base}
\phi_\alpha \equiv \alpha \ \mbox{ in } (-\infty, 0] \ \mbox{ and } \ 
\phi_\alpha > \alpha \ \mbox{ in } (0,+\infty),
\end{equation}
by translation.

We are now moving on to a discussion on traveling wave dynamics of more general solutions to \eqref{irAC}. We first exhibit well-posedness for \eqref{irAC} in a certain  setting suitable for considering  traveling wave solutions.  We shall not use \eqref{f-1} for a while, but always need \eqref{f-2} with $f(0)=0$, which together with \eqref{f-lam-convex} implies $\beta(0)=0$. The condition $f(0)=0$ will be employed to construct $L^2$ solutions $u_n$ of \eqref{irAC} for compactly supported data $u_{0,n}$ (see \S \ref{Ss:loc-ex} below), otherwise $f(u_n)$ never lies on $L^2(\R)$. On the other hand, it is not necessary to assume $f(0)=0$ explicitly, provided that $f$ has a zero (see \eqref{f-1}). Indeed, one can assume that the zero of $f$ coincides with the origin without any loss of generality by translation in $u$.  

The following theorem ensures existence of \emph{$L^2_{\rm loc}$ solutions} for \eqref{irAC} (see Definition \ref{D:sol} below for more details on $L^2_{\rm loc}$ solutions). 
\begin{theorem}[Existence of $L^2_{\rm loc}$ solutions]\label{T:loc-L2-ex}
Assume \eqref{f-2}  with $f(0)=0$.  Let $T > 0$ and assume that
\begin{equation}
u_0 \in H^2_{\rm loc}(\R) \cap L^\infty(\R)\label{u0-loc-H2} 
\end{equation}
{\rm  (see Notation below){\rm }} and there exist constants $\gamma_\pm \in \R$ such that 
\begin{equation}
\gamma_- \leq \inf_{x\in\R} u_0(x) \leq \sup_{x\in\R} u_0(x) \leq \gamma_+ \ \mbox{ and } \ \pm f(\gamma_\pm) \geq 0.\label{mp}
\end{equation}
Then \eqref{irAC} admits a bounded $L^2_{\rm loc}$ solution $u = u(x,t)$ on $[0,T]$. Furthermore, $u=u(x,t)$ also solves \eqref{pde-obs0}, whose solution is uniquely determined by the initial datum, in $Q_T := \R \times (0,T)$. Moreover, for any non-negative $\rho \in C^\infty_c(\R)$, it holds that
\begin{equation}\label{eta-est}
\esssup_{t \geq 0} \int_\R |\eta(\cdot,t)|^2 \rho \, \d x
\leq \int_\R \big|\big(\partial_x^2 u_0-f(u_0)\big)_-\big|^2 \rho \, \d x.
\end{equation}
\end{theorem}
To prove the theorem above, we shall construct an $L^2_{\rm loc}$ solution of \eqref{irAC-dn}. Due to the non-smooth and nonlinear nature of the equation, we shall work on a frame of (generalized) gradient flows on Hilbert spaces (see, e.g.,~\cite{HB1}) instead of classical parabolic theories (e.g., Schauder theory) used for \eqref{cac}. On the other hand, such a frame is based on Hilbert spaces (or reflexive Banach spaces), although we are concerned with solutions with values in $L^2_{\rm loc}$, which is of course not even a normed space. Here we shall first construct an $L^2$ solution (as approximate solutions for compactly supported data) with the aid of a generalized gradient flow (in $L^2$) theory developed in~\cite{ae1}, and then, we shall extend it to an existence result of $L^2_{\rm loc}$ solutions by developing local estimates.

As we have seen in Theorem \ref{T:TW}, traveling wave solutions $U(x,t) = \phi_\alpha(x-x_0-c_\alpha t)$ lose $C^2$-regularity only at the free boundary point $x = x_0+c_\alpha t$ separating $\R$ into two half-lines, i.e., a half-line on which $U(x,t)\equiv \alpha$ and the other on which $U(x,t)$ is strictly increasing. On the other hand, $U(x,t)$ is continuously (partially) differentiable in time  everywhere. And so, there arises a question whether or not such a mechanism of losing classical regularity is universal for more general solutions $u(x,t)$ to \eqref{irAC}. The next theorem will guarantee $C^1$ regularity both in space and time of $u(x,t)$ on $\R \times \R_+$, and this fact will enable us to determine the  singularity  of $u(x,t)$ at the free boundary point 
\begin{equation}\label{r(t)}
r(t) := \sup\{r \in \R \colon u(x,t) = \alpha \mbox{ for all } x \leq r\}, 
\end{equation}
 provided that $r(0)$ is finite. More precisely, 
$u(x,t)$ loses $C^2$ regularity at $x = r(t)$ in a similar fashion to traveling waves, i.e.,
$$
\lim_{x \to r(t)-0}\partial_x^2 u(x,t) = 0 \ \mbox{ and } \ \lim_{x \to r(t)+0} \partial_x^2 u(x,t) = f(\alpha) > 0,
$$
for any $t > 0$ (see \S \ref{Ss:appl-c1} for more details). Furthermore, this fact will be also employed to prove Theorem \ref{T:Main} below.

\begin{theorem}[$C^1$ regularity]\label{T:C1}
Assume \eqref{f-2}  with $f(0)=0$,  \eqref{u0-loc-H2} and \eqref{mp}. Then \eqref{irAC} admits a unique bounded $L^2_{\rm loc}$-solution $u = u(x,t)$ which also solves \eqref{pde-obs0} and satisfies
 \begin{gather*}
u \in W^{1,\infty}(0,T;L^2(-R,R)) \cap W^{1,2}(0,T;H^1(-R,R)),\\
  \sqrt t \partial_t u \in L^\infty(0,T;H^1(-R,R)), \quad \sqrt t \partial_t^2 u \in L^2(0,T;L^2(-R,R))
 \end{gather*}
 for any $T>0$ and $R > 0$. In particular, $u$ belongs to $C^1(\R \times (0,+\infty))$.
\end{theorem}

Here it is noteworthy that the solution has classical regularity with respect to the time-variable. On the other hand, it is not of class $C^2$ in the space-variable (for instance, the traveling wave solution has $W^{2,\infty}$ regularity in both variables but no $C^2$ regularity). Such a feature exhibits a clear contrast to degenerate and singular parabolic equations as well as classical Allen-Cahn equations.

The rest of this section will be devoted to discussing stability of traveling wave solutions constructed in Theorem \ref{T:TW}. Here, central questions are under which conditions for initial data the corresponding $L^2_{\rm loc}$ solution of \eqref{irAC} converges to a traveling wave and what the convergence rate is.  We first notice the following
\begin{proposition}[Instability of traveling waves in $L^\infty(\R)$]\label{P:inst-tw}
For each $\alpha \in (a_-,a_0)$, the traveling wave solution $\phi_\alpha(x-c_\alpha t)$ of \eqref{irAC} is unstable in $L^\infty(\R)$.
\end{proposition}
This proposition will be proved at the end of Section \ref{S:ob} and suggests that one needs to impose a more restrictive assumption on initial data than the classical Allen-Cahn (see~\cite{FM,XChen}) to prove convergence (see also Remark \ref{R:non-deg} below).

To give an answer to the questions, we assume instead of \eqref{alpha-hyp} that 
\begin{equation}\label{alpha-hyp-2}
 \alpha \in (a_-,a_0), \quad \int^{a_+}_{a_-} f(z) \, \d z \leq 0,
\end{equation}
which implies \eqref{alpha-hyp} by $f > 0$ on $(a_-,a_0)$. Here we remark that, under  the second condition of  \eqref{alpha-hyp-2}, the velocity of traveling wave solutions for  the classical Allen-Cahn equation  \eqref{cac} is  also  non-positive. Moreover, let us introduce the following assumptions for initial data $u_0$:
\begin{description}
 \item[${\bf (H)}_\alpha$] it holds that
\begin{equation}\label{H1}
\begin{cases}
u_0 \in H^2_{\rm loc}(\R), \quad \displaystyle \liminf_{x\to+\infty}u_0(x) > a_0, \\
a_- \leq \displaystyle \inf_{x \in \R} u_0(x) \leq \sup_{x \in \R} u_0(x) < a_+, 
\end{cases}
\end{equation}
and moreover, there exists $\xi_1 \in \R$ such that
\begin{alignat}{3}
 u_0(x) &= \alpha \ &&\mbox{ for all } x \leq \xi_1.\label{H3}
\end{alignat}
\end{description}
 Here \eqref{H3} implies that $r(0)$ is finite (see \eqref{r(t)}), i.e., the free boundary point is well defined at $t = 0$. 
Assumptions \eqref{f-1} and ${\bf (H)}_\alpha$ imply \eqref{mp}. Moreover, thanks to a comparison principle (see Lemma \ref{L:cp-obs+} in \S \ref{S:ExCP}), one can modify the function $f(\cdot)$ outside $[a_-,a_+]$ in an arbitrary way. Hence the conditions $\hat f \geq 0$ and $f' \geq -\lam$ (see \eqref{f-2}) can be always assumed under the $C^2$ regularity of $f$ without any loss of generality. 
See also Remark \ref{R:H-alpha} below for further remarks on the assumption above.

Our main result is then stated as follows:
\begin{theorem}[Exponential convergence to degenerate traveling waves]\label{T:Main}
In addition to \eqref{f-1}, \eqref{f-2} and \eqref{alpha-hyp-2}, assume that ${\bf (H)}_\alpha$ and $f'(\alpha) > 0$ hold. Let $u = u(x,t)$ be an $L^2_{\rm loc}$ solution to \eqref{irAC} in $\R \times \R_+$. Then there exist constants $K, \kappa > 0$ and $x_0 \in \R$ such that
$$
\|u(\cdot,t)-\phi_\alpha(\cdot-c_\alpha t-x_0)\|_{L^\infty(\R)} \leq K \e^{-\kappa t} \quad \mbox{ for all } \ t \geq 0.
$$
Moreover, let $r(t)$ be defined by \eqref{r(t)}. Then $r(t) - c_\alpha t$ converges to $x_0$ at the rate of $O(\e^{-\frac{\kappa t}2})$ as $t \to +\infty$. 
\end{theorem}
 In the theorem stated above, the condition $f'(\alpha)>0$ is assumed in addition to \eqref{alpha-hyp-2}. Under \eqref{f-1}, one can always ensure that
$$
A := \left\{ \alpha \in (a_-,a_0) \colon f'(\alpha) > 0 \right\} \neq \emptyset.
$$
Indeed, 
by $f'(a_-) > 0$ along with the continuity of $f'$, the set $A$ involves any $\alpha \geq a_-$ close enough to $a_-$.

%
%

\noindent 
{\bf Structure of the paper.} \ Section \ref{S:ob} is concerned with traveling waves for \eqref{irAC}. We shall construct a one-parameter family of traveling wave solutions by focusing on the alternative profile equation \eqref{phi} and also discuss uniqueness (see Proposition \ref{P:tw-unique}) and various properties of traveling waves, e.g., continuous dependence of the velocity and profile function on the parameter (see Corollary \ref{C:conti-cphi}), which implies the instability of traveling waves (see Proposition \ref{P:inst-tw}). Moreover, a couple of comparison functions are constructed in terms of profile functions and velocity constants (see \S \ref{Ss:comfunc}) and they will play crucial roles throughout the paper. In Section \ref{S:ExCP}, we shall establish an existence result for $L^2_{\rm loc}$ solutions to \eqref{irAC} (see \S \ref{Ss:loc-ex}). To treat traveling wave solutions for the fully nonlinear equation \eqref{irAC}, one may need to work in an $L^2_{\rm loc}$ setting instead of a standard one, say in $L^2$. To this end, we shall also develop a comparison principle for $L^2_{\rm loc}$ solutions to the obstacle problem \eqref{pde-obs0} (see \S \ref{Ss:cp}). Section \ref{S:OL} is devoted to give an outline of a proof for Theorem \ref{T:Main}, which consists of four phases. The first two phases will be discussed in detail there and enable us to reduce the problem into a parabolic equation with a constant obstacle after a certain time. On the other hand, it requires more technical lemma to discuss the other two phases, and hence, the details will be given in the following sections. More precisely, in Section \ref{S:reg}, the $C^1$ regularity (in space and time) of solutions to \eqref{irAC} will be proved (see Theorem \ref{T:C1}). Sections \ref{S:Q-conv} and \ref{S:exp-conv} are devoted to quasi-convergence of shifted solutions for the constant obstacle problem to a profile function and exponential convergence, respectively. 

\noindent
{\bf Notation.} We write $\R_+ := (0,+\infty)$. Moreover, $H^2_{\rm loc}(\R)$ is the linear space consisting of functions $u \in L^2_{\rm loc}(\R)$ whose first and second derivatives (in the distributional sense) lie on $L^2_{\rm loc}(\R)$.
 Here and henceforth, we use the same notation $\iI$ for the indicator function over the half-line $[0,\infty)$ as well as for that defined on the function space $L^2(I)$ for any open interval $I\subseteq\R$ over the closed convex set $K := \{u \in L^2(I) \colon u \geq 0 \ \mbox{ a.e.~in } I\}$, namely,
$$
\iI(u) = \begin{cases}
	  0 &\mbox{ if } \ u \in K,\\
	  \infty &\mbox{ otherwise }
	 \end{cases}
	 \quad \mbox{ for } \ u \in L^2(I),
	 $$
	 if no confusion may arise. Moreover, let $\partial \iI$ also denote the subdifferential operator (precisely, $\partial_{\mathbb R} \iI$) in $\mathbb R$ as well as that (precisely, $\partial_{L^2(I)} \iI$) in $L^2(I)$ defined by
	 $$
	 \partial_{L^2(I)} \iI(u) = \left\{ \eta \in L^2(I) \colon (\eta, u - v) \geq 0 \ \mbox{ for all } \ v \in K \right\} \quad \mbox{ for } \ u \in K.
	 $$
Here, we note that these two notions of subdifferentials are equivalent each other in the following sense: for $u, \eta \in L^2(I)$,
	 $$
	 \eta \in \partial_{L^2(I)} \iI(u) \quad \mbox{ if and only if } \quad \eta(\cdot) \in \partial_{\mathbb R} \iI(u(\cdot)) \ \mbox{ a.e.~in } I
	 $$
	 (see, e.g.,~\cite{HB1,HB3}). Furthermore, $\iI^*$ denotes the \emph{convex conjugate} functional (from $L^2(I)$ to $[0,+\infty]$) of $\iI$ (see, e.g.,~\cite{B}), i.e., 
$$
\iI^*(\eta) = \sup_{w \in L^2(I)} [(\eta,w)_{L^2(I)} - \iI(w)] \quad \mbox{ for } \ \eta \in L^2(I),
$$
which can be also identified (similarly to $\iI$) with the conjugate function (from $\R$ to $[0,+\infty]$) given by
$$
\iI^*(r) = \sup_{z \in \R} (rz - \iI(r)) \quad \mbox{ for } \ r \in \R.
$$
Finally, we recall that
$$
\eta \in \partial \iI(u) \quad \mbox{ if and only if } \quad u \in \partial \iI^*(\eta),
$$
and then, the \emph{Fenchel-Moreau identity}
$$
\iI^*(\eta) = (\eta,u)_{L^2(I)} - \iI(u),
$$
holds and so does a pointwise identity (see, e.g.,~\cite{B} for more details).

\section{Traveling wave solutions via an obstacle problem}\label{S:ob}

In this section, we shall construct traveling wave solutions for \eqref{irAC} through the auxiliary obstacle problem \eqref{phi}. More precisely, we shall prove the existence of $\phi$ satisfying \eqref{phi}, \eqref{phi-bc} and \eqref{phi-bc2}.

Equation \eqref{phi} can be regarded as a profile equation for traveling wave solutions to the Allen-Cahn equation with a \emph{non-smooth} potential,
$$
V(u) := \iIo(u) + W(u), \quad W(u) := \hat f(u),
$$
where $\hat f$ is a primitive function of $f$. Hence the situation seems to be close to a standard one (cf.~\cite[pp.175--180]{Evans}). However, it is still beyond the scope of classical results due to the \emph{non-smoothness} of the potential (see, e.g.,~\cite{FM, XChen}). Let us first rewrite \eqref{phi} as a system of first-order ordinary differential inclusions,
\begin{align}
 \phi' &= \psi,\label{phi-1}\\
 \psi' &\in - c \psi - \gamma(\phi),\label{phi-2}
\end{align}
where $\gamma(\cdot)$ is a multi-valued function given by
$$
\gamma(\phi) = - \partial \iIo(\phi) - f(\phi).
$$
We note that $\gamma(\phi) = -f(\phi)$, provided that $\phi > \alpha$. Not surprisingly, one can prove:
\begin{lemma}\label{L:c}
Assume \eqref{alpha-hyp} and let $\phi \in W^{2,\infty}(\R)$ be a solution to \eqref{phi}, \eqref{phi-bc} and \eqref{phi-bc2}. Then $\phi'$ belongs to $L^2(\R)$, and moreover, it holds that
 \begin{equation}\label{c}
 c = \dfrac{W(a_+)-W(\uo)}{\int_{\R} \phi'(\sigma)^2 \,\d \sigma} < 0.
 \end{equation}
\end{lemma}

\begin{proof}
 Test \eqref{phi} by $\phi'$ to see that
 \begin{equation}\label{ei-W}
 \dfrac 1 2 \dfrac{\d}{\d s} \phi'(s)^2 + c \phi'(s)^2 - \dfrac{\d}{\d s} W(\phi(s)) = 0.
 \end{equation}
Here we used the chain-rule, $\phi' \eta = (\d/\d s) \iIo(\phi)$ for any $\eta \in \partial \iIo(\phi)$ (see, e.g.,~\cite{HB1}), and the fact that $\iIo(\phi) \equiv 0$ by definition. Integrating both sides over $(s,t)$ for $-\infty < s < t < +\infty$, we have
 $$
 \dfrac 1 2 \phi'(t)^2 -  \dfrac 1 2 \phi'(s)^2
 + c \int^t_s \phi'(\sigma)^2 \, \d \sigma
 = W(\phi(t)) - W(\phi(s)).
 $$
Passing to the limit as $t \to +\infty$ and $s \to -\infty$ and using \eqref{phi-bc}, \eqref{phi-bc2} and \eqref{alpha-hyp} (i.e., $W(a_+) < W(\alpha)$), one obtains
$$
c \int_{\R} \phi'(\sigma)^2 \, \d \sigma = W(a_+) - W(\uo) \in (-\infty,0),
$$
which implies $c < 0$, $\phi' \not\equiv 0$ and $\phi' \in L^2(\R)$. Thus \eqref{c} follows.
\end{proof}

\subsection{Construction of traveling wave solutions}\label{Ss:tw}

In this subsection, we shall prove Theorem \ref{T:TW}. In order to construct a function $\phi$ satisfying \eqref{phi}, \eqref{phi-bc} and \eqref{phi-bc2}, we assume \eqref{alpha-hyp}, which implies $W(\alpha) > W(a_+)$, and introduce a ``two-step approximation'' for the multi-valued function $\gamma$ (see Fig.~\ref{F:2} below). The first approximation is based on the celebrated \emph{Yosida approximation} and given by
$$
\gamma_\mu(\phi) := - (\partial \iIo)_\mu(\phi) - f(\phi), \quad \mu > 0,
$$
which is single-valued, continuous and smooth in $\R \setminus \{\uo\}$ (but not of class $C^1$ at $\uo$). Here we also note that $(\partial \iIo)_\mu$ entails a gradient structure,
$$
(\partial \iIo)_\mu = \partial (\iIo)_\mu,
$$
where $(\iIo)_\mu$ is the \emph{Moreau-Yosida regularization} of $\iIo$, that is,
$$
(\iIo)_\mu (s) = \begin{cases}
		      0 &\mbox{ if } s \geq \uo,\\
		      \frac{|s-\uo|^2}{2\mu} &\mbox{ if } s < \uo.
		     \end{cases}
$$
In particular, we observe $\gamma_\mu = -f$ on $(\uo,+\infty)$. Moreover, $\gamma_\mu$ admits a zero $\uol \in (a_-,\uo)$ (for $\mu > 0$ small enough) satisfying
$$
\uo - C \mu < \uol < \uo \quad \mbox{ for } \ \mu > 0
$$
for some constant $C > 0$. 

Now, let $\mu > 0$ be fixed. The second one is a smooth (at least of class $C^1$) approximation $g_n$ (for each $n \in \mathbb N$) of $\gamma_\mu$ satisfying
\begin{enumerate}
 \item $g_n(z) = \gamma_\mu(z) = -f(z)$ for all $z \geq \uo$,
 \item $g_n$ is increasing in $n$ and $-f \leq g_n \leq \gamma_\mu$ on $(-\infty,\uo]$,
 \item $g_n \to \gamma_\mu$ locally uniformly on $\R$ as $n \to +\infty$,
 \item for each $n \in \mathbb N$, $g_n$ has a zero $\uon$ such that $\uol - \frac C n < \uon \leq \uol$ for some $C > 0$; hence $g_n(\uon) = g_n(a_0) = g_n(a_+) = 0$,
 \item $g_n < 0$ on $(\uon,a_0)$, $g_n > 0$ on $(a_0,a_+)$,
 \item $g_n'(\uon) < 0$, $g_n'(a_+) < 0$.
\end{enumerate}
\begin{figure}[t]
\includegraphics[scale=0.37]{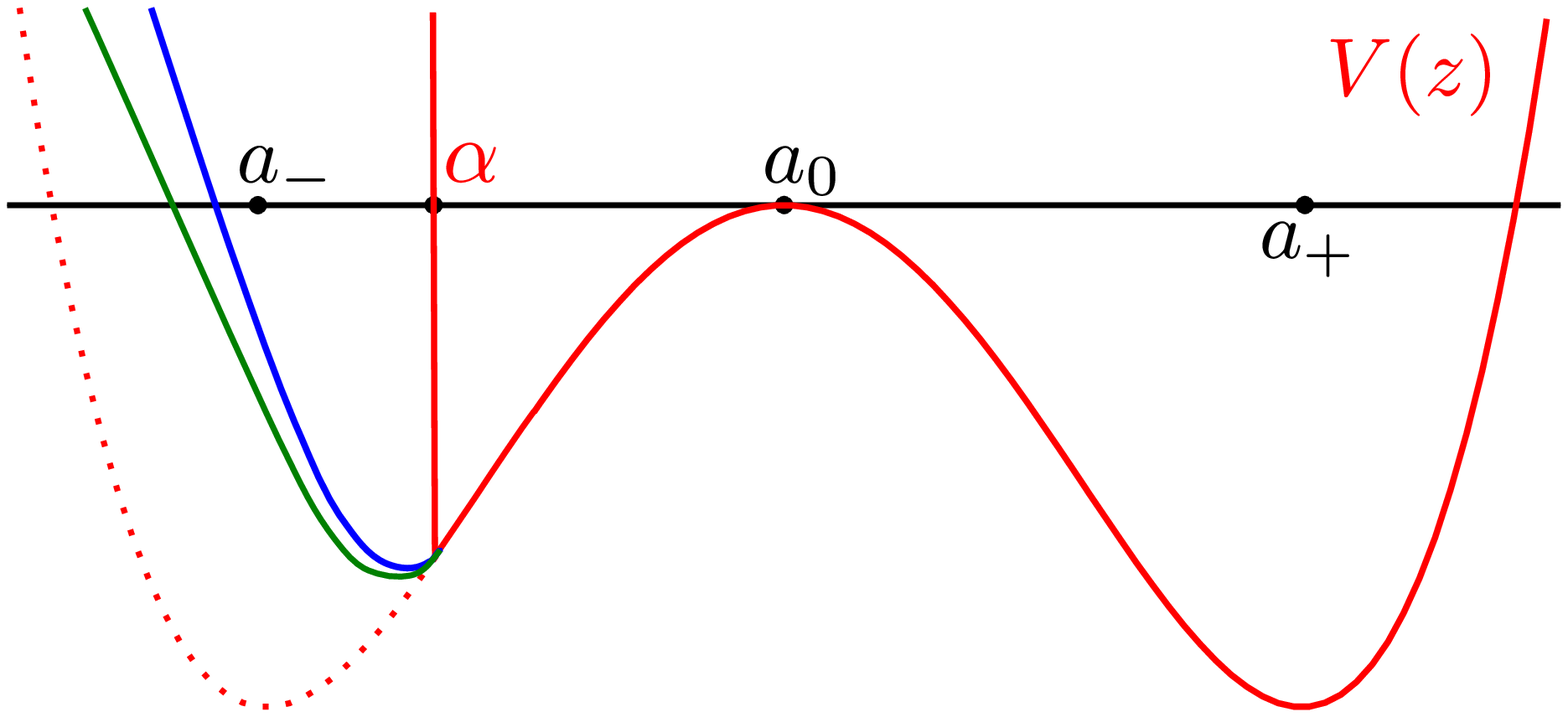}
\includegraphics[scale=0.37]{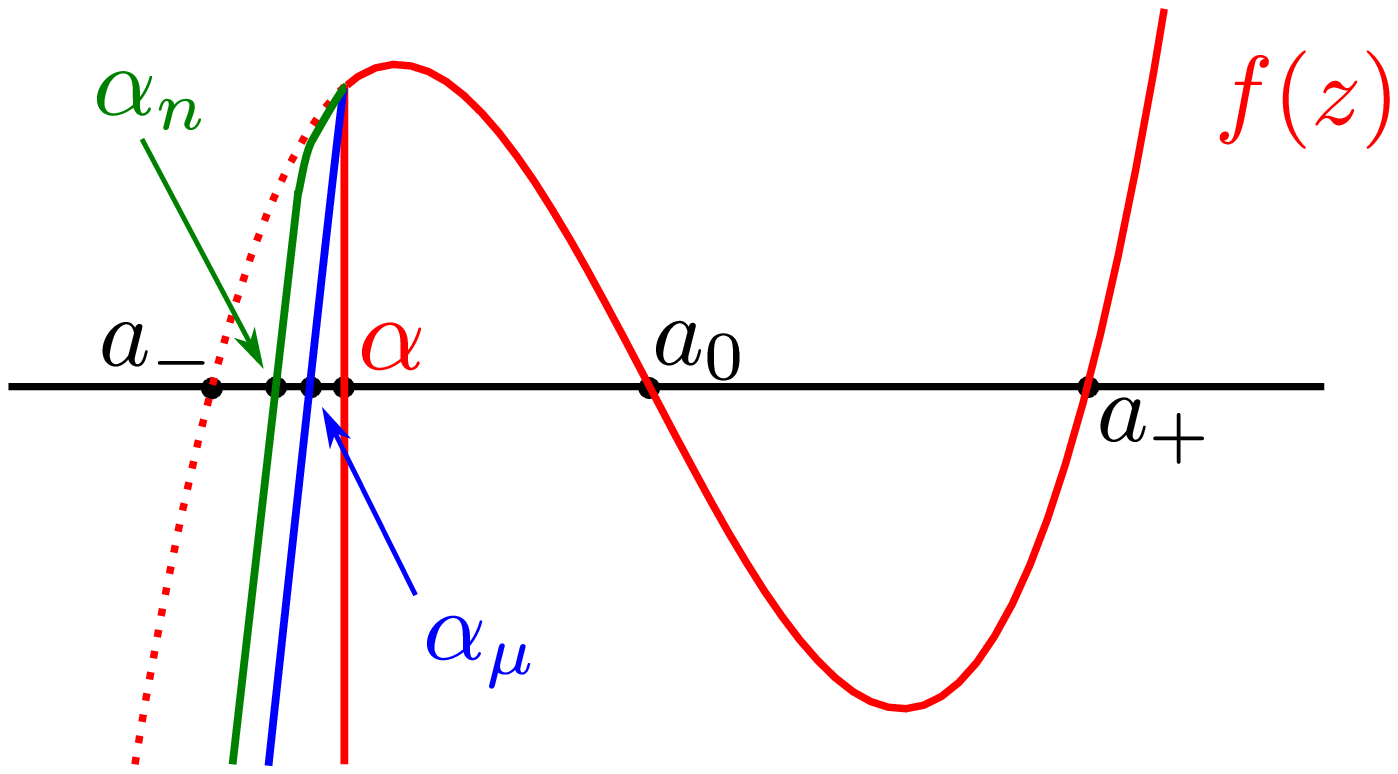}
\caption{ {\sc Cut-offed double-well potential and regularizations.} Here a balanced double-well potential $W(z)$ is considered, that is, $W(a_+)=W(a_-)$. Left: The red curve is the graph of the cut-offed potential $V(z) = W(z) + \iIIal(z)$. The dotted curve is the cut-offed part of the graph for $W(z)$. The blue one is a part of the graph for the Moreau-Yosida regularized potential $W_\mu(z)$, i.e., a primitive function of $-\gamma_\mu(z)$, and $\uo_\mu$ is the local minimizer of $W_\mu(z)$. Moreover, the green one is that for the smoothly approximated potential $W_n(z)$, i.e., a primitive function of $-g_n(z)$, and $\uo_n$ is the local minimizer of $W_n(z)$. It is noteworthy that the cut-offed potential and their regularizations are no longer balanced although the original one is balanced. Intuitively speaking, such an imbalanced situation induces a driving force of traveling waves for \eqref{irAC}. Right: The graphs of $-\gamma(z) = f(z) + \partial \iIIal(z)$ (red), $-\gamma_\mu(z)$ (blue) and $-g_n(z)$ (green).}\label{F:2}
\end{figure}
Here, to be precise, one should write $g_{\mu,n}$ and $\alpha_{\mu,n}$ instead of $g_n$ and $\uon$, respectively. However, we omit the subscript $\mu$, unless any confusion may arise. Then (i) implies
\begin{equation}\label{a}
\int^{a_+}_{\uon} g_n(z) \, \d z
= - \int^{a_+}_{\uo} f(z) \, \d z + \int^{\uo}_{\uon} g_n(z) \, \d z
> 0
\end{equation}
for $n \in \mathbb N$ large enough and $\mu > 0$ small enough, since $g_n = \gamma_\mu = -f$ on $(\uo,+\infty)$, $\int^{a_+}_{\uo} f(z) \, \d z < 0$ by \eqref{alpha-hyp} and
\begin{align*}
 \left| \int^{\uo}_{\uon} g_n(\zeta) \, \d \zeta \right|
 \leq (\uo - \uon) \sup_{s \in [a_-,a_0]}f(s) \leq \left(C \mu + \dfrac C n\right) \sup_{s \in [a_-,a_0]}f(s).
\end{align*}
Therefore by~\cite{FM,XChen} (and~\cite[\S 4.2, c]{Evans}), for each $n \in \mathbb N$ one can construct a heteroclinic orbit $(\phi_n,\psi_n)$ and a negative constant $c_n$ satisfying \eqref{phi-1} and \eqref{phi-2} with $\gamma$ and $c$ replaced by $g_n$ and $c_n$, respectively, as well as
$$
\psi_n > 0, \quad
\lim_{s \to +\infty} (\phi_n,\psi_n) = (a_+,0), \quad
\lim_{s \to -\infty} (\phi_n,\psi_n) = (\uon,0).
$$
By translation, one can assume without any loss of generality that
\begin{equation}\label{zero}
\phi_n(0) = a_0 \quad \mbox{ for any } \ n \in \mathbb{N}.
\end{equation}
Recalling \eqref{ei-W} with $W$ replaced by the corresponding potential $W_n = - \hat g_n$ and integrating both sides over $(s,0)$, we have
$$
 \dfrac 1 2 \phi_n'(0)^2 -  \dfrac 1 2 \phi_n'(s)^2
 + c_n \int^0_s \phi_n'(\sigma)^2 \, \d \sigma
 = W_n(\phi_n(0)) - W_n(\phi_n(s)).
 $$
 Pass to the limit as $s \to -\infty$. Then it yields that
 \begin{align*}
\dfrac 1 2 \phi_n'(0)^2 + c_n \int^0_{-\infty} \phi_n'(\sigma)^2 \, \d \sigma
&= W_n(a_0) - W_n(\uo_n)\\
&\geq W_n(a_0) - W_n(\uo) = W(a_0) - W(\uo),
 \end{align*}
 which together with $c_n < 0$ implies
 \begin{equation}\label{zero-deri}
  \dfrac 1 2 \phi_n'(0)^2 \geq W(a_0)-W(\uo) > 0 \quad \mbox{ for any } \ n \in \mathbb{N}.
 \end{equation}
Here we also used the fact that $W(a_0)-W(\uo) = \int^{a_0}_\uo f(z) \, \d z > 0$ by $f > 0$ in $(a_-,a_0)$ and $\alpha \in (a_-,a_0)$ (see \eqref{f-1} and \eqref{alpha-hyp}).

We move on to the limiting procedure of $(\phi_n,\psi_n)$ as $n \to +\infty$. To this end, we first establish a priori estimates. One finds immediately that
\begin{equation}\label{phi-n-bdd-01}
\uo - \frac C n - C \mu \leq \uon < \phi_n < a_+, \quad 0 < \psi_n.
\end{equation}
We further derive an upper bound for $\psi_n$. By simple calculation, the energy
$$
E_n(v,w) := \dfrac{w^2}2 + \int^v_{\uon} g_n(z) \, \d z
$$
turns out to satisfy
$$
\dfrac{\d}{\d s} E_n(\phi_n(s),\psi_n(s)) = - c_n \psi_n(s)^2 \geq 0.
$$
Hence the function $s \mapsto E_n(\phi_n(s), \psi_n(s))$ is non-decreasing. Integrating both sides over $(-\infty,s)$, we have
$$
E_n(\phi_n(s),\psi_n(s)) = E_n(\uon,0) - c_n \int^{s}_{-\infty} \psi_n(\sigma)^2 \, \d \sigma = - c_n \int^{s}_{-\infty} \psi_n(\sigma)^2 \, \d \sigma.
$$
From the relation
$$
c_n = - \dfrac{\int^{a_+}_{\uon} g_n(z) \, \d z}{\int_{\R} \psi_n(\sigma)^2 \, \d \sigma} \stackrel{\eqref{a}}< 0.
$$
it follows that
$$
E_n(\phi_n(s),\psi_n(s)) \leq \int^{a_+}_{\uon} g_n(z) \, \d z
\leq - \int^{a_+}_{\uo} f(z) \, \d z,
$$
which implies
\begin{align*}
\dfrac 1 2 \psi_n(s)^2
&\leq - \int^{a_+}_{\uo} f(z) \, \d z- \int^{\phi_n(s)}_{\uon} g_n(z) \, \d z\\
&\leq - \int^{a_+}_{\uo} f(z) \, \d z - \int^{a_0}_{\uon} g_n(z) \, \d z\\
&\leq - \int^{a_+}_{\uo} f(z) \, \d z + \left(a_0 - \uo + \frac C n + C \mu \right) \sup_{s \in [a_-,a_0]} f(s).
\end{align*}
Therefore
\begin{equation}\label{phi-n-bdd-02}
0 < \psi_n(s) \leq C_1,
\end{equation}
where $C_1 := \sqrt 2 [-\int^{a_+}_{\uo} f(z) \, \d z + \left(a_0 - \uo + 1\right) \sup_{s \in [a_-,a_0]} f(s)]^{1/2}$, for all $n \gg 1$ and $0 < \mu \ll 1$.

We next derive a uniform estimate for $c_n$, that is,
\begin{equation}\label{c-bdd}
- C_2 \leq c_n < 0 \quad \mbox{ for all } \ n \in \mathbb N,
\end{equation}
where $C_2$ is independent of $n$ as well as $\mu$. 
We start with observing by \eqref{phi-1} and \eqref{phi-2} with $\gamma$ and $c$ replaced by $g_n$ and $c_n$, respectively, that
$$
c_n = \dfrac{-\psi_n'(s) - g_n(\phi_n(s))}{\psi_n(s)} \quad \mbox{ for } \ s \in \R.
$$
From the fact that $\psi_n(s) \to 0$ as $s \to \pm\infty$ and $\psi_n > 0$ in $\R$, there exists $s_n \in \R$ such that $\psi_n(s_n) = \max_{s \in \R} \psi_n(s) > 0$, and hence, $\psi_n'(s_n) = 0$. Thus
$$
c_n = \dfrac{- g_n(\phi_n(s_n))}{\psi_n(s_n)}
\geq - \dfrac{\|g_n\|_{L^\infty(\uon,a_+)}}{\psi_n(s_n)}.
$$
Due to the energy inequality,
$$
E_n(\phi_n(s),\psi_n(s)) \geq E_n(\uon,0) = 0,
$$
which implies
$$
\dfrac 1 2 \psi_n(s)^2 + \hat g_n(\phi_n(s)) \geq 0,
$$
where $\hat g_n$ is given by $\hat g_n(z) := \int^z_{\uon} g_n(\zeta) \, \d \zeta$. Thus we obtain
$$
\psi_n(s) \geq \sqrt{-2 \hat g_n(\phi_n(s))},
$$
provided that $\hat g_n(\phi_n(s)) \leq 0$. Furthermore,
$$
\hat g_n(z) = \int^z_{\uon} g_n(\zeta) \, \d \zeta < 0 \quad \mbox{ for all } \ z \in [\uon,a_0].
$$
Moreover, recall that $\phi_n(0) = a_0$. Thus we deduce that
$$
\psi_n(0) \geq \sqrt{-2 \hat g_n(\phi_n(0))} = \sqrt{-2 \hat g_n(a_0)}
\geq \sqrt{2 \hat f(a_0)} > 0,
$$
where $\hat f(s) := \int^s_\alpha f(z) \, \d z$. Here we used the fact by $f > 0$ in $(a_-,a_0)$ that
$$
\hat g_n(a_0) = \int^{a_0}_{\uon} g_n(z) \, \d z
\leq - \int^{a_0}_{\uo} f(z) \, \d z = - \hat f(a_0) < 0.
$$
Consequently,
$$
c_n \geq - \dfrac{\|g_n\|_{L^\infty(\uon,a_+)}}{\psi_n(s_n)}
\geq - \dfrac{\|g_n\|_{L^\infty(\uon,a_+)}}{\psi_n(0)}
\geq - \dfrac{\|f\|_{L^\infty(a_-,a_+)}}{\sqrt{2 \hat f(a_0)}}
> -\infty.
$$
Here, we also used the fact $\|g_n\|_{L^\infty(\uon,a_+)} \leq \|f\|_{L^\infty(a_-,a_+)}$. Setting $C_2 := \|f\|_{L^\infty(a_-,a_+)}/\sqrt{2 \hat f(a_0)} < +\infty$, we obtain \eqref{c-bdd}.

 From \eqref{phi-1} and \eqref{phi-2} with $\gamma = \gamma_n$ along with \eqref{phi-n-bdd-01} and \eqref{phi-n-bdd-02},  we further deduce that
$$
|\phi_n'(s)| \leq C, \quad |\psi_n'(s)| \leq C \quad \mbox{ for all } s \in \R \ \mbox{ and } \ n \in \mathbb N,
$$
where $C$ is independent of $n \to +\infty$. Therefore, up to a (not relabeled) subsequence,
\begin{align*}
 \phi_n \to \phi \quad &\mbox{ weakly star in } W^{1,\infty}(\R),\\
 &\mbox{ locally uniformly on } \R,\\
 \psi_n \to \psi \quad &\mbox{ weakly star in } W^{1,\infty}(\R),\\
  &\mbox{ locally uniformly on } \R
\end{align*}
 for some $\phi, \psi \in W^{1,\infty}(\R)$.  Furthermore, since $g_n \to \gamma_\mu$ locally uniformly on $\R$, we have
\begin{align*}
|g_n(\phi_n) - \gamma_\mu(\phi)| &\leq |g_n(\phi_n) - \gamma_\mu(\phi_n)| + |\gamma_\mu(\phi_n) - \gamma_\mu(\phi)|\\
 &\leq \sup_{z \in [a_-,a_+]}|g_n(z) - \gamma_\mu(z)| + |\gamma_\mu(\phi_n) - \gamma_\mu(\phi)|\\
 &\to 0 \ \mbox{ locally uniformly on } \R \ \mbox{ as } n \to +\infty.
\end{align*}
Thus
\begin{align*}
 g_n(\phi_n) \to \gamma_\mu(\phi) \quad &\mbox{ weakly star in } L^\infty(\R),\\
 &\mbox{ locally uniformly on } \R.
\end{align*}
Consequently, the limit (denoted by $(\phi_\mu,\psi_\mu)$) of $(\phi_n,\psi_n)$ as $n \to +\infty$ solves \eqref{phi-1} and \eqref{phi-2} with $\gamma$ and $c$ replaced by $\gamma_\mu$ and some constant $c_\mu$, and moreover, it follows  from \eqref{phi-1} and \eqref{phi-2} with $\gamma = \gamma_\mu$  that
$$
\phi_\mu \in C_b^2(\R), \quad \psi_\mu \in C_b^1(\R),
$$
where $C_b^k := C^k \cap W^{k,\infty}$.  One also finds that $\uon \to \uol$, up to a subsequence, as $n \to +\infty$, by $\uol - C/n < \uon \leq \uol$.  Furthermore, $(\phi_\mu,\psi_\mu)$ satisfies the same a priori estimates obtained so far:
\begin{equation}\label{est-l}
 \uo - C \mu \leq \uol \leq \phi_\mu \leq a_+, \quad 0 \leq \psi_\mu \leq C_1, \quad - C_2 \leq c_\mu < 0.
\end{equation}
We also remark that \eqref{zero} and \eqref{zero-deri} hold true with $\phi_n$ replaced by the limit $\phi_\mu$ (in particular, $\phi_\mu$ is not constant).

Now, we are ready to discuss convergence of $(\phi_\mu,\psi_\mu)$ as $\mu \to +0$. Indeed, all bounds in the preceding a priori estimates are  bounded as $\mu \to +0$, since we find that $\min_{z \in[a_-,a_+]}(-f(z)) \leq \gamma_\mu(\phi_\mu) \leq \max_{z\in[a_-,a_+]} (-f(z))$ by $\gamma_\mu(\alpha_\mu) = 0$ and \eqref{est-l},  
and therefore, one can obtain, up to a (not relabeled) subsequence, convergences of $(\phi_\mu, \psi_\mu)$ in a similar manner,
\begin{align*}
 \phi_\mu \to \phi \quad &\mbox{ weakly star in } W^{1,\infty}(\R),\\
 &\mbox{ locally uniformly on } \R,\\
  \psi_\mu \to \psi \quad &\mbox{ weakly star in } W^{1,\infty}(\R),\\
 &\mbox{ locally uniformly on } \R,\\
\gamma_\mu(\phi_\mu) \to \xi \quad &\mbox{ weakly star in } L^\infty(\R)
\end{align*}
for some $\phi \in W^{2,\infty}(\R)$, $\psi \in W^{1,\infty}(\R)$ and $\xi \in L^\infty(\R)$. Some difference from the previous step arises in the identification of the limit of $\gamma_\mu(\phi_\mu)$ as $\mu \to +0$. By using the graph-convergence property of Yosida approximation  (or demiclosedness of maximal monotone graphs, see, e.g.,~\cite{HB1}), one can identify the limit $\xi$ as
$$
\phi \geq \uo \quad \mbox{ and } \quad \xi \in \gamma(\phi).
$$
Furthermore, $c_\mu$ also converges to a constant $c$ by extracting a (not relabeled) subsequence. Consequently, $(\phi,\psi)$ solves \eqref{phi-1} and \eqref{phi-2} with $\gamma$ and $c$ (hence $\phi$ solves \eqref{phi}), and moreover, we find that
$$
\phi \in W^{2,\infty}(\R), \quad \psi \in W^{1,\infty}(\R), \quad
\uo \leq \phi \leq a_+, \quad 0 \leq \psi \leq C_1,
$$
and furthermore, \eqref{zero} and \eqref{zero-deri} are satisfied with $\phi_n$ replaced by $\phi$ (thus, $\phi$ is not constant). 

We next check \eqref{phi-bc}. Since $\phi$ is non-decreasing and non-constant, $\phi(s)$ is convergent as $s \to \pm \infty$ such that
$$
\alpha \leq \ell_- := \lim_{s \to -\infty} \phi(s) < a_0 < \ell_+ := \lim_{s \to +\infty} \phi(s) \leq a_+.
$$
The next lemma clarifies the behavior of $\phi(s)$ as $s \to +\infty$ (and then, $\ell_-$ will be determined). Moreover, it also exhibits a significant difference between \eqref{irAC} and \eqref{cac}.

\begin{lemma}
 There exists $s_* \in \R$ such that $\phi(s) = \uo$ for all $s \leq s_*$ {\rm (}and hence, $\ell_- = \uo${\rm )} and $\phi(s) > \alpha$ for all $s > s_*$.
\end{lemma}

\begin{proof}
 We prove this by contradiction. Suppose on the contrary that $\phi > \uo$ in $\R$. Then $\phi$ also solves \eqref{phi-cac} by the fact that $\partial \iIo(\phi) = \{0\}$. Moreover, from the fact that $\phi(0) = a_0$ and $\phi'(0) > 0$, one can take constants $\vep_0 < a_0$ and $s_0 < 0$ close to $a_0$ and $0$, respectively, such that
 $$
\uo < \phi < \vep_0 \ \mbox{ in } (-\infty,s_0).
 $$
Hence we observe that
 $$
 f(\phi(s)) \geq d_0 := \inf_{z \in [\uo,\vep_0]} f(z) > 0 \quad \mbox{ for all } \ s \in (-\infty,s_0),
 $$
because $f > 0$ in $[\uo,\vep_0] \subset (a_-,a_0)$. Thus we have
 $$
 \phi'' + c\phi' \geq d_0 \ \mbox{ in } (-\infty,s_0).
 $$
 Integrating this over $(\sigma,s)$ with $-\infty < \sigma < s < s_0$, we have
 $$
 \phi'(s) + c \phi(s) \geq \phi'(\sigma) + c \phi(\sigma) + d_0 (s - \sigma).
 $$
 Letting $\sigma \to -\infty$ and employing the fact that $d_0 > 0$, we infer that
 $$
 \phi'(s) + c \phi(s) \geq \phi'(\sigma) + c \phi(\sigma) + d_0 (s - \sigma)
 \geq c \ell_- + d_0 (s - \sigma) \to + \infty
 $$
 as $\sigma \to -\infty$.
 However, this is a contradiction to the boundedness of $\phi'=\psi$ and $\phi$ mentioned above. 
Therefore there exists $s_* \in \R$ such that $\phi = \alpha$ in $(-\infty,s_*]$ and $\phi > \alpha$ in $(s_*,+\infty)$ by monotonicity of $\phi$.
\end{proof}

Let us also check $\ell_+ = a_+$. Suppose on the contrary that $\ell_+ < a_+$. Then one can take numbers $\vep_1 > a_0$ and $s_1 > 0$ close to $a_0$ and $0$, respectively, such that
$$
\vep_1 < \phi \leq \ell_+ < a_+ \ \mbox{ in } (s_1, +\infty).
$$
Thus
$$
\phi'' + c \phi' = f(\phi) \leq d_1 := \sup_{z \in [\vep_1,\ell_+]} f(z) < 0,
$$
which implies
$$
\phi'(s) + c \phi(s) \leq \phi'(\sigma) + c\phi(\sigma) + d_1(s-\sigma)
$$
for $s_1 < \sigma < s < +\infty$. Then $\phi'(s) + c \phi(s) \to -\infty$ as $s \to +\infty$; however, it contradicts the boundedness of $\phi$ and $\phi'$. Therefore $\ell_+ = a_+$. Thus we have checked \eqref{phi-bc}.

Moreover, we shall prove
\begin{lemma}\label{L:phi'->0}
Let $\phi$ satisfy \eqref{phi} and \eqref{phi-bc} for some $\alpha \in (a_-,a_0)$. Then it holds that $\phi'(s) \to 0$ as $s \to +\infty$.  
\end{lemma}

\begin{proof}
Since $\phi(s) \to a_+$ as $s \to +\infty$ by \eqref{phi-bc}, there is a sequence $\sigma_n \to +\infty$ such that $\phi'(\sigma_n) \to 0$ (indeed, if not, $\phi(s)$ must be divergent). Now, let $s_n \to +\infty$ be an arbitrarily sequence. One can take a subsequence $(\tilde\sigma_n)$ of $(\sigma_n)$ such that $\tilde\sigma_n \leq s_n$ for each $n \in \mathbb N$. On the other hand, note by \eqref{phi-bc} that $\phi > a_0 > \alpha$ in $(s_1,+\infty)$ for some $s_1 > 0$. Hence \eqref{phi} yields
$$
\phi'' + c \phi' = f(\phi) \leq 0 \ \mbox{ in } (s_1, +\infty).
$$
Integrating both side over $(\tilde\sigma_n,s_n)$ for $n$ large enough, we deduce that
$$
\phi'(s_n) + c \phi(s_n) \leq \phi'(\tilde\sigma_n) + c \phi(\tilde\sigma_n) \to c a_+.
$$
Since $\phi(s_n) \to a_+$ and $\phi'$ is non-negative, we conclude that
 $$
 \phi'(s_n) \to 0,
 $$
which completes the proof.
\end{proof}

Combining the above with the fact that $\phi' \equiv 0$ on $(-\infty,s_*)$, we also obtain \eqref{phi-bc2}. Moreover, as in the classical Allen-Cahn equation, one can also check that
\begin{lemma}\label{L:phi<a+}
Let $\phi$ satisfy \eqref{phi} and \eqref{phi-bc} for some $\alpha \in (a_-,a_0)$. Then it holds that 
$$
\phi < a_+ \ \mbox{ in } \R.
$$
\end{lemma}

\begin{proof}
Indeed, if $\phi$ could attain $a_+ > \uo$ at some $s_0 \in \R$, then $\phi$ solves \eqref{phi-cac} in a neighbourhood of $s_0$, and hence, $\phi'(s_0) = 0$ (see \eqref{phi} and \eqref{phi-bc}). On the other hand, the constant function $\overline\phi \equiv a_+$ is a solution of \eqref{phi-cac}. Hence due to the uniqueness of solutions to ordinary differential equations (ODEs) with locally Lipschitz nonlinearity, we conclude that $\phi \equiv a_+$ on $\R$. However, this is a contradiction to \eqref{phi-bc}. 
\end{proof}

One may also prove the following

\begin{lemma}\label{L:phi'>0} Under \eqref{f-1}, let $\phi$ satisfy \eqref{phi} and \eqref{phi-bc} for some $\alpha \in (a_-,a_0)$. Then it holds that
$$
\phi'(s) > 0 \ \mbox{ if } \ \phi(s) > \uo.
$$
\end{lemma}

\begin{proof}
Suppose on the contrary that $\phi(s_0) > \uo$ and $\phi'(s_0)=0$ for some $s_0 \in \R$. By \eqref{phi-bc} and Lemma \ref{L:phi<a+}, we have already known that $a_- < \alpha \leq \phi < a_+$ in $\R$. In case $\phi(s_0) \neq a_0$, we observe by \eqref{f-1} and \eqref{phi} that
$$
\phi''(s_0)=f(\phi(s_0))\neq 0.
$$
Hence $\phi$ attains strict minimum or maximum at $s_0$; however, it is a contradiction to the monotonicity of $\phi$ by \eqref{phi-bc}. In case $\phi(s_0) = a_0$, the uniqueness theorem for ODEs ensures $\phi \equiv a_0$, which also yields a contradiction to \eqref{phi-bc}. Thus we obtain $\phi' > 0$ if $\phi(s) >\alpha$.
\end{proof}

Next, we prove that $\phi$ also solves \eqref{phi2}.

\begin{lemma}
The function $\phi$ obtained above solves \eqref{phi2}.  Hence $u(x,t)=\phi(x-ct)$ is a solution of \eqref{irAC} in $\R \times \R_+$.
\end{lemma}

 \begin{proof}
 Let $s \in \R$. In case $\phi(s) > \uo$, noting the fact that $\partial \iIo(\phi(s)) = \{0\}$, $c \leq 0$ and $\phi' \geq 0$, we have
 $$
 0 \leq -c \phi'(s) \stackrel{\eqref{phi}}= \phi''(s) - f(\phi(s))
 = \left( \phi''(s) - f(\phi(s)) \right)_+.
  $$
  In case $\phi(s) = \uo$, it holds that $\phi'(s) = 0$, and hence,
  $$
  0 = -c \phi'(s) \stackrel{\eqref{phi}}\geq \phi''(s) - f(\phi(s)),
  $$
  which implies $(\phi''(s) - f(\phi(s)) )_+ = 0$. Thus $\phi$ solves \eqref{phi2} in $\R$.
 \end{proof}

Moreover, $u(x,t) = \phi(x - ct)$ solves a parabolic obstacle problem,
\begin{equation*}
 \partial_t u - \partial_x^2 u + \partial \iIo(u) + f(u) \ni 0 \ \mbox{ in } \R \times \R_+, \quad u|_{t = 0} = u_0 \ \mbox{ in } \R
\end{equation*}
with $u_0(x) := \phi(x)$. Here we further claim that
$$
\partial \iIo(u) \subset \partial \iII(u).
$$
Indeed, in case $\phi(x-ct) > \phi(x) = u_0(x)$, since $\phi(x) \geq \uo$, we see that $\phi(x-ct) > \uo$, and hence, $\partial \iIo(\phi(x-ct)) = \{0\} = \partial \iII(\phi(x-ct))$. In case $\phi(x-ct) = \phi(x) = u_0(x)$, it follows that $\partial \iII(\phi(x-ct)) = (-\infty,0]$, and therefore, $\partial \iIo(\phi(x-ct)) \subset \partial \iII(\phi(x-ct))$. Thus $u(x,t) = \phi(x -ct)$ turns out to be a solution of
\begin{equation*}
\partial_t u - \partial_x^2 u + \partial \iII(u) + f(u) \ni 0 \ \mbox{ in } \R \times \R_+, \quad u|_{t = 0} = u_0 \ \mbox{ in } \R
\end{equation*}
with $u_0(x) = \phi(x)$. 


\begin{remark}[Phase plane analysis]
{\rm
 The traveling wave profile $\phi_{ac}$ for the classical Allen-Cahn equation \eqref{cac} can be constructed based on a phase plane analysis for the profile equation \eqref{phi-cac}. Indeed, the profile $\phi_{ac}$ is a heteroclinic orbit connecting $(a_-,0)$ and $(a_+,0)$ (see the blue curve on Fig.~\ref{F:pp}). On the other hand, for each $\alpha \in (a_-,a_0)$, the profile function $\phi_\alpha$ corresponds to a curve emanating from $(\alpha,0)$ and converging to $(a_+,0)$ (see the red curve on Fig.~\ref{F:pp}). Then $\phi_\alpha$ is not selected as a traveling wave profile for \eqref{cac}, since it is not an entire solution of \eqref{phi-cac}. In contrast to the present paper, one can also construct traveling waves for \eqref{irAC} and investigate their properties by performing a phase plane analysis on solutions for the Cauchy problem of the second order ODE \eqref{phi-cac} in $(0,+\infty)$.  
\begin{figure}[t]
\includegraphics[scale=0.27]{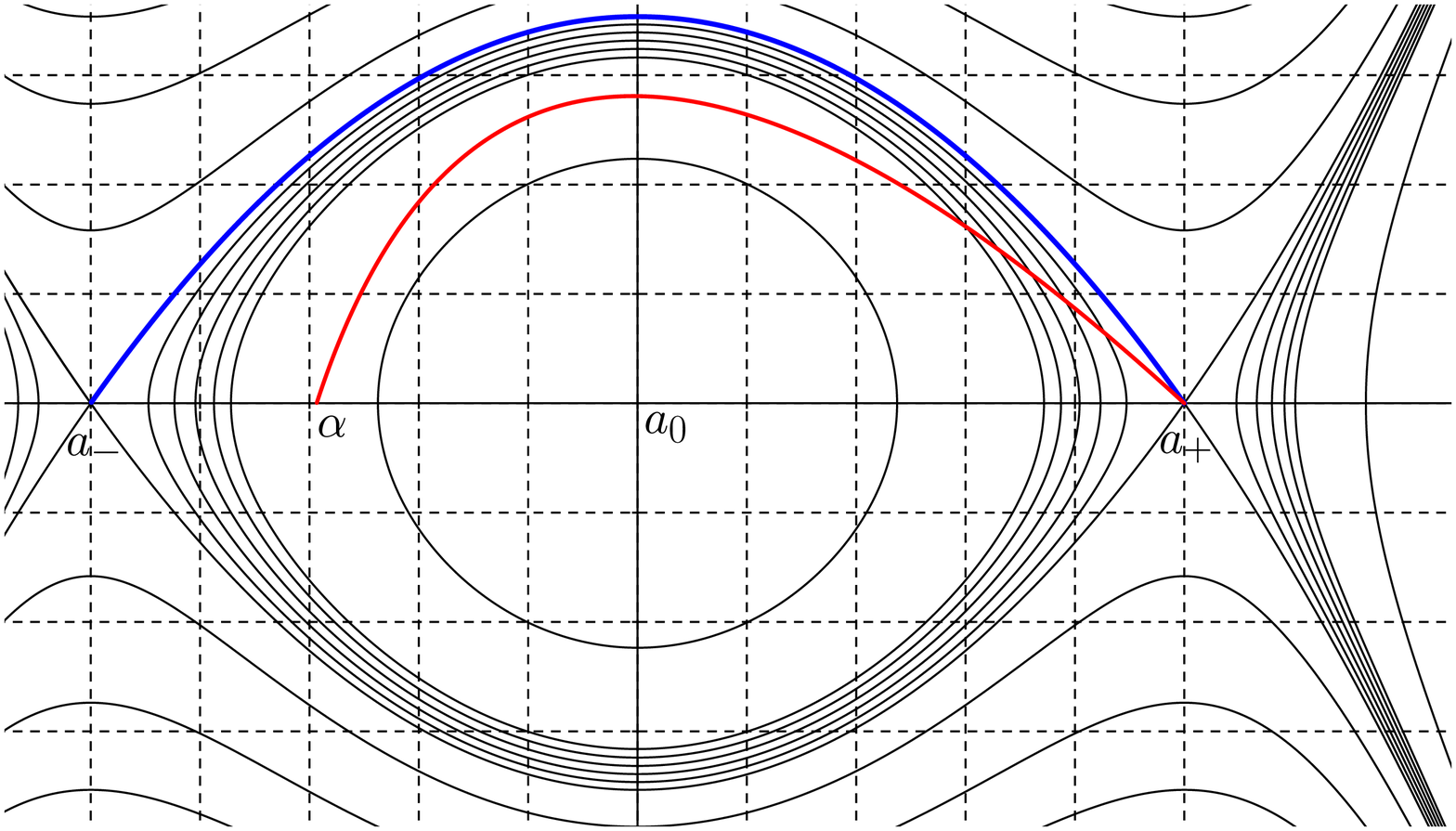} 
\caption{ {\sc Phase plane.} The blue and red curves correspond to traveling wave profiles for \eqref{cac} and \eqref{irAC}, respectively. Here the red one is not an entire solution for \eqref{phi-cac}, but it is for \eqref{phi2}.}\label{F:pp}
\end{figure}

}
\end{remark}

\subsection{Unique determination of velocity constant and profile function}\label{Ss:uni-tw}

As in the case of classical Allen-Cahn equations, for each $\alpha \in (a_-,a_0)$, one can prove the uniqueness of the velocity constant and the profile function for (degenerate) traveling wave solutions to \eqref{irAC} (cf.~Proposition \ref{P:uni-det}).

\begin{proposition}[Uniqueness of velocity constant and profile function]\label{P:tw-unique}
Under the same assumptions as in Theorem \ref{T:TW}, for each $\alpha \in (a_-, a_0)$, the velocity constant $c$ and the profile function $\phi$ satisfying \eqref{phi}, \eqref{phi-bc}, \eqref{phi-bc2} and \eqref{phi-deg} are uniquely determined for the profile function.
\end{proposition}

For the convenience of the reader, we give a proof, although it is similar to the classical case.

\begin{proof}
Let $\phi$ and $c$ be a profile function and a velocity constant of a traveling wave solution to \eqref{irAC} such that \eqref{phi}, \eqref{phi-bc}, \eqref{phi-bc2} and \eqref{phi-deg} are satisfied. In particular, $\phi$ solves
$$
- \phi'' + f(\phi) = c \phi' \ \mbox{ in } (s_*, +\infty),
$$
where $s_*$ is the number that appeared in \eqref{phi-deg}. Setting $\psi := \phi'$, one can rewrite the relation above as
$$
\psi' = f(\phi) - c \psi, \quad \phi' = \psi \ \mbox{ in } (s_*, +\infty).
$$
Let us regard $\psi$ as a function of $\phi$, i.e., $\psi = \psi(\phi)$ (indeed, this is possible, since $\phi$ is strictly increasing (i.e., $\psi > 0$) in $(s_*,+\infty)$ by Lemma \ref{L:phi'>0}). Then one can write
\begin{equation}\label{psi-phi}
\dfrac{\d \psi}{\d \phi}(\phi) = \frac{f(\phi)}{\psi(\phi)} - c \ \mbox{ for } \ \phi \in (\alpha, a_+). 
\end{equation}
Now, for each $j = 1,2$, let $(\phi_j, c_j)$ be a pair of profile function and velocity constant satisfying \eqref{phi}, \eqref{phi-bc}, \eqref{phi-bc2} and \eqref{phi-deg} and let $\psi_j = \phi_j'$. Moreover, write $\psi_j = \psi_j(\phi)$ as a function of $\phi \in (\alpha,a_+)$. Then by \eqref{psi-phi} with $\psi = \psi_j$ and $c = c_j$, one has
$$
\dfrac{\d}{\d \phi} \left[\psi_1(\phi) - \psi_2(\phi) \right]
= - (c_1-c_2) + \frac{f(\phi)}{\psi_1(\phi)\psi_2(\phi)} \left[\psi_2(\phi) - \psi_1(\phi) \right],
$$
whence it follows that
\begin{equation}\label{phi1-phi2}
\dfrac{\d}{\d \phi} \left[ (\psi_1- \psi_2) \exp \left( \int^\phi_{a_0} h(z) \, \d z \right) \right] = - (c_1-c_2) \exp \left( \int^\phi_{a_0} h(z) \, \d z \right), 
\end{equation}
where $h(z) := f(z)/(\psi_1(z) \psi_2(z))$. Note that by \eqref{f-1} and $\psi_j(\phi) > 0$, we have
$$
- \infty < \int^\phi_{a_0} h(z) \, \d z \leq 0 \quad \mbox{ for all } \ \phi \in (\alpha,a_+).
$$
Hence, if $c_1 \neq c_2$, then the function
$$
(\psi_1- \psi_2) \exp \left( \int^\phi_{a_0} h(z) \, \d z \right)
$$
is strictly monotone. On the other hand, it converges to $0$ as $\phi \to \alpha$ or $\phi \to a_+$ due to \eqref{phi-bc2}. This is a contradiction, and therefore, $c_1 = c_2$. Moreover, one can also derive $\psi_1 \equiv \psi_2$ in $(\alpha,a_+)$ by \eqref{phi1-phi2} along with $c_1 = c_2$ and by the fact that $\psi_j(\alpha) = 0$. Thus $\phi_1$ coincides with $\phi_2$.
\end{proof}

As a corollary, we have
\begin{corollary}[Strict decrease of the velocity in $\alpha$]\label{C:c-order}
Under the same assumptions as in Theorem \ref{T:TW}, the velocity $c_\alpha$ is strictly decreasing in $\alpha \in (a_-,a_0)$.
\end{corollary}

\begin{proof}
Let $a_- < \alpha_1 < \alpha_2 < a_0$. For $j=1,2$, let $\phi_j$ be the traveling wave profile for $\alpha = \alpha_j$ and set $\psi_j = \phi_j'$. Repeat the same argument as in the proof of Proposition \ref{P:tw-unique} with $\alpha = \alpha_2$. Here we observe that $\psi_1(\alpha_2)$ is still positive due to $\alpha_1 < \alpha_2$. Suppose on the contrary that
$$
c_1 \leq c_2.
$$
Then the right-hand side of \eqref{phi1-phi2} is non-negative for $\phi \in (\alpha_2,a_+)$, whence it follows that the function
$$
\phi \mapsto \left(\psi_1(\phi)-\psi_2(\phi)\right) \exp \left( \int^\phi_{a_0} h(z) \, \d z \right)
$$
is non-decreasing in $(\alpha_2,a_+)$. Recalling that $\psi_1(\alpha_2) > 0$ and $\psi_2(\alpha_2) = 0$, one can take $\phi_0 > \alpha_2$ close enough to $\alpha_2$ such that $\psi_1(\phi_0) - \psi_2(\phi_0) > (1/2) \psi_1(\phi_0) > 0$. Therefore
\begin{align*}
\left(\psi_1(\phi)-\psi_2(\phi)\right) \exp \left( \int^\phi_{a_0} h(z) \, \d z \right)
&\geq \frac 1 2 \psi_1(\phi_0) \exp \left( \int^{\phi_0}_{a_0} h(z) \, \d z \right) > 0
\end{align*}
for $\phi \in (\phi_0,a_+)$. On the other hand, by the passage to the limit as $\phi \to a_+$, the left-hand side converges to zero, and hence, it yields a contradiction.
\end{proof}

We can also verify that

\begin{corollary}[Continuity of velocity and profile]\label{C:conti-cphi}
Let $\alpha \in (a_-,a_0)$ and let $(\alpha_n)$ be a sequence in $(a_-,a_0)$ such that $\alpha_n \to \alpha$. Then 
$$
c_{\alpha_n} \to c_\alpha \quad \mbox{ and } \quad \phi_{\alpha_n} \to \phi_\alpha \ \mbox{ uniformly on } \R.
$$
\end{corollary}

\begin{proof}
By the monotonicity of $c_\alpha$ in $\alpha$, it follows that $c_{\alpha_n}$ converges to a limit $c \neq 0$ as $n \to +\infty$. Recalling Lemma \ref{L:c} along with the boundedness of $c_{\alpha_n}$, one derives the boundedness of $(\phi'_{\alpha_n})$ in $L^2(\R)$. Therefore by virtue of Ascoli-Arzela's theorem, we can prove, up to a (not relabeled) subsequence, that
\begin{alignat}{4}
\phi_{\alpha_n} &\to \phi \quad &&\mbox{ locally uniformly on } \R,\label{c:phian}\\
\phi_{\alpha_n}' &\to \phi' &&\mbox{ weakly in } L^2(\R)\nonumber
\end{alignat}
for some $\phi \in C(\R)$ satisfying $\phi' \in L^2(\R)$ and $\phi' \geq 0$ in $\R$ (and hence, up to a subsequence, $\phi_{\alpha_n}(s) \to \phi(s)$ for each $s \in \R$ by a diagonal argument). Since $\phi_{\alpha_n} \equiv \alpha_n$ on $(-\infty,s_*]$ and $\phi_{\alpha_n}$ solves \eqref{phi-cac} on $(s_*,+\infty)$ (see Theorem \ref{T:TW}), one can verify that $\phi \in H^2_{loc}(\R) \subset C^1(\R)$ and
$$
\phi \equiv \uo \ \mbox{ on } (-\infty,s_*], \quad -\phi'' + f(\phi) = c \phi' \ \mbox{ a.e.~in } (s_*,+\infty).
$$
In particular, we see that
$$
\phi \in C^2((s_*,+\infty)) \quad \mbox{ and } \quad \lim_{s \searrow s_*} \phi''(s) = f(\uo) > 0.
$$
Thus \eqref{phi-deg} follows, and moreover, we find that $\phi'' = f(\phi) - c \phi' > -c \phi'$, whenever $\phi$ lies in $(\uo,a_0)$. Hence $\phi$  attains $a_0$ at some $s_0 \in \R$ with $\phi'(s_0) > 0$. Since $\phi$ is non-decreasing in $\R$ and bounded from above by $a_+$, there exists a limit $\phi_\infty \in (a_0,a_+]$ such that $\phi(s) \to \phi_\infty$ as $s \to +\infty$. Suppose on the contrary that $\phi_\infty < a_+$. Then by $f(\phi_\infty) < 0$ and \eqref{phi-cac} on $(s_*,+\infty)$, we find that $c\phi(s)+\phi'(s)$ diverges to $-\infty$ as $s \to +\infty$. However, it is a contradiction. Thus we obtain $\phi_\infty = a_+$, and hence, we have \eqref{phi-bc}. Furthermore, \eqref{phi-bc2} also follows from Lemma \ref{L:phi'->0}. Accordingly, by Proposition \ref{P:tw-unique} we conclude that $c = c_\alpha$ and $\phi = \phi_\alpha$.

By virtue of the increase of $\phi_{\alpha_n}$ and $\phi$, we observe that
\begin{align*}
\sup_{s \geq R} |\phi(s)-\phi_{\alpha_n}(s)|
&\leq \sup_{s \geq R} \left( |\phi(s)-a_+|+|a_+-\phi_{\alpha_n}(s)| \right)\\
&\leq a_+ - \phi(R) + a_+ - \phi_{\alpha_n}(R)\\
&\leq 2(a_+-\phi(R)) + \phi(R)-\phi_{\alpha_n}(R).
\end{align*}
Hence, for any $\vep > 0$, one can choose $R_\vep > 0$ and $N_\vep \in \mathbb N$ large enough that
$$
\sup_{s \geq R_\vep} |\phi(s)-\phi_{\alpha_n}(s)| < \vep \quad \mbox{ for } \ n \geq N_\vep.
$$
Combining this fact with \eqref{c:phian}, we obtain
$$
\phi_{\alpha_n} \to \phi \quad \mbox{ uniformly on } \R,
$$
which completes the proof.
\end{proof}

Now, we are ready to prove Proposition \ref{P:inst-tw}.

\begin{proof}[Proof of Proposition \ref{P:inst-tw}]
Fix $\alpha \in (a_-,a_0)$. Then by Corollary \ref{C:conti-cphi}, for any $\vep > 0$ one can take $\alpha_\vep \neq \alpha$ such that
$$
\|\phi_\alpha - \phi_{\alpha_\vep}\|_{L^\infty(\R)} < \vep.
$$
Moreover, it also follows that
$$
\lim_{t \to +\infty}\|\phi_\alpha(\cdot-c_\alpha t) - \phi_{\alpha_\vep}(\cdot-c_{\alpha_\vep} t)\|_{L^\infty(\R)} \geq a_+-a_0 > 0
$$
by the fact $c_{\alpha_\vep} \neq c_{\alpha}$. 
\end{proof}

\subsection{Comparison functions}\label{Ss:comfunc}

In this subsection, based on degenerate traveling wave solutions constructed so far, we shall develop comparison functions for the classical Allen-Cahn equation,
\begin{equation}\label{ac-w}
\partial_t w = \partial_x^2 w - f(w)
\end{equation}
for later use. They will be finally used for proving exponential convergence of solutions for \eqref{irAC} to degenerate traveling wave solutions as $t \to +\infty$  under certain assumptions on initial data  in Section \ref{S:exp-conv}. In what follows, we shall always assume the normalization \eqref{phi-base}, i.e., $\phi(s) > 0$ if and only if $s > 0$.

Let $\alpha$ be such that \eqref{alpha-hyp} and $f'(\alpha) > 0$ are satisfied and set
\begin{equation}\label{wpm}
w^\pm(x,t) := \phi_\alpha(x - x_0 - c_\alpha t \pm \sigma \delta (1-\e^{-\beta t}) ) \pm \delta \e^{-\beta t}
\end{equation}
for $\delta,\sigma,\beta>0$ and $x_0 \in \R$.

\begin{lemma}\label{L:comfunc}
Assume that $f'(\alpha) > 0$. Then there exist positive constants $\beta_0, \delta_0$ depending only on $f$, $\alpha$ and $\phi_\alpha$ such that $w^+$ {\rm (}respectively, $w^-${\rm )} is a classical supersolution {\rm (}respectively, a subsolution{\rm )} of the classical Allen-Cahn equation \eqref{ac-w} in $\R \times \R_+$ {\rm (}respectively, in ${Q_-}:=\{(x,t) \in \R \times \R_+ \colon x > x_0 + c_\alpha t + \sigma\delta(1-\e^{-\beta t})\}${\rm )} for any $x_0 \in \R$, $\beta \in (0,\beta_0)$, $\delta \in (0,\delta_0)$ and $\sigma > \sigma_\beta$ and some $\sigma_\beta > 0$ depending only on $f$, $\alpha$, $\phi_\alpha$ and $\beta$. Moreover, $\sigma_\beta$ is decreasing in $\beta$ and divergent to $+\infty$ as $\beta \to +0$. 

Furthermore, $w^+$ and $w^-$ are a super- and a subsolution, respectively, of the obstacle problem
\begin{equation}\label{obs}
\min \{u-\alpha, 
\partial_t u - \partial_x^2 u + f(u)
\} 
= 0 \ \mbox{ in } \R \times (0,+\infty).
\end{equation}
\end{lemma}

\begin{proof}
By translation, one can assume $x_0 = 0$ without any loss of generality. In what follows, we write $z^\pm := x - c_\alpha t \pm \sigma \delta (1-\e^{-\beta t})$. By simple calculation,
$$
\partial_t w^\pm(x,t) = \phi_\alpha'(z^\pm) \left[ - c_\alpha \pm \sigma \delta \beta \e^{-\beta t}\right] \mp \delta \beta \e^{-\beta t}.
$$
Moreover,
\begin{align*}
\lefteqn{
- \partial_x^2 w^\pm(x,t) + f(w^\pm(x,t)) 
}\\
&= - \phi_\alpha''(z^\pm) + f(\phi_\alpha(z^\pm) \pm \delta \e^{-\beta t})\\
&= - \phi_\alpha''(z^\pm) + f(\phi_\alpha(z^\pm)) \pm \int^1_0 f'(\phi_\alpha(z^\pm) \pm \delta \e^{-\beta t} s) \delta \e^{-\beta t}\, \d s.
\end{align*}
Then we observe that
\begin{align*}
 \lefteqn{
 \partial_t w^\pm(x,t) - \partial_x^2 w^\pm(x,t) + f(w^\pm(x,t))
 }\\
&= - c_\alpha \phi_\alpha'(z^\pm) - \phi_\alpha''(z^\pm) + f(\phi_\alpha(z^\pm))\\
&\quad \pm \delta \e^{-\beta t} \left[ \sigma \beta \phi_\alpha'(z^\pm) - \beta
+ \int^1_0 f'(\phi_\alpha(z^\pm) \pm \delta \e^{-\beta t} s) \, \d s \right].
\end{align*}

We first prove the assertion for $w^+$. In what follows, we simply write $z$ instead of $z^+$. Note that $w^+ > \alpha$ by definition. Moreover, $- c_\alpha \phi_\alpha' - \phi_\alpha'' + f(\phi_\alpha) \geq 0$ in $\R$, and therefore, it suffices to check
\begin{equation}\label{cl1}
\sigma \beta \phi_\alpha'(z) - \beta
+ \int^1_0 f'(\phi_\alpha(z) + \delta \e^{-\beta t} s) \, \d s  \geq 0
\end{equation}
under certain choice of $\beta$, $\delta$ and $\sigma$. 

Let $\delta_0 = \delta_0(f,\alpha,\phi_\alpha) > 0$ and $r = r(f,\alpha,\phi_\alpha) > 0$ be small enough and let $R = R(f,\alpha,\phi_\alpha)> 0$ be large enough that
$$
\inf_{s \in [0,1]}f'(\underbrace{\phi_\alpha(y)+s \delta_0}_{\ \simeq \ \uo}) > \dfrac{f'(\alpha)}2 > 0 \quad \mbox{ for all } \ y \in (-\infty,r]
$$
and
$$
c_R := \inf_{s \in [0,1]} f'(\underbrace{\phi_\alpha(y)+s \delta_0}_{\ \simeq \ a_+}) > 0 \quad \mbox{ for all } \ y \in [R,+\infty).
$$
Indeed, the former one is possible, because $f'(\alpha) > 0$ by assumption, $\phi_\alpha = \alpha$ on $(-\infty,0]$ and $f'$ and $\phi_\alpha$ are continuous. Moreover, so is the latter one, since $f'(a_+) > 0$, $f'$ is continuous at $a_+$ and $\phi_\alpha(+\infty) = a_+$.

\underline{For $z \in (R,+\infty)$}, it follows that
\begin{align*}
\sigma \beta \phi_\alpha'(z) - \beta
+ \int^1_0 f'(\phi_\alpha(z) + \delta \e^{-\beta t} s) \, \d s
> - \beta + c_R > 0,
\end{align*}
whenever $0 < \delta < \delta_0$ and $0 < \beta < c_R$. \underline{For $z \in (-\infty,r)$}, one observes that
\begin{align*}
\sigma \beta \phi_\alpha'(z) - \beta
+ \int^1_0 f'(\phi_\alpha(z) + \delta \e^{-\beta t} s) \, \d s
> - \beta + \dfrac{f'(\alpha)}2 > 0,
\end{align*}
whenever $0 < \delta < \delta_0$ and $0 < \beta < \frac{f'(\alpha)}2$. We set
$$
\beta_0 := \min \left\{ c_R, \dfrac{f'(\alpha)}2 \right\} > 0.
$$

Set
$$
0 < m := \min_{z \in [r,R]} \phi_\alpha'(z) 
< +\infty.
$$
\underline{For $z \in [r,R]$}, we see that
\begin{align*}
\lefteqn{
\sigma \beta \phi_\alpha'(z) - \beta
+ \int^1_0 f'(\phi_\alpha(z) + \delta \e^{-\beta t} s) \, \d s
}\\
&> (\sigma m - 1)\beta - \|f'\|_{L^\infty(a_-,a_++\delta_0)} > 0,
\end{align*}
provided that
$$
\sigma > \sigma_\beta := m^{-1} \left( \dfrac{\|f'\|_{L^\infty(a_-,a_++\delta_0)}}\beta + 1 \right).
$$
Therefore $w^+$ turns out to be a supersolution of \eqref{ac-w} (and hence, of \eqref{obs}) in $\R \times \R_+$ for any $\delta \in (0,\delta_0)$, $\beta \in (0,\beta_0)$ and $\sigma > \sigma_\beta$. Moreover, we remark that the choice of $\delta_0$ and $\beta_0$ depend only on $f$, $\alpha$ and $\phi_\alpha$; on the other hand, $\sigma_\beta$ depends on $\beta$ as well as $f$, $\alpha$ and $\phi_\alpha$ (more precisely, $\sigma_\beta$ is decreasing in $\beta$ and divergent to $+\infty$ as $\beta \to +0$).

As for the assertion for $w^-$, we first note that $- c_\alpha \phi_\alpha' - \phi_\alpha'' + f(\phi_\alpha) = 0$ in $\R_+$. Hence as in the assertion for $w^+$, one can prove in a similar way that $w^-$ is a subsolution of the Allen-Cahn equation \eqref{ac-w}, whenever $z^- > 0$. Moreover, it follows immediately that $w^- < \alpha$ when $z^- \leq 0$. Therefore $w^-$ turns out to be a subsolution of \eqref{obs} on $\R \times \R_+$.
\end{proof}

\section{Existence and comparison principle of $L^2_{\rm loc}$ solutions}\label{S:ExCP}

Throughout this section, we assume \eqref{f-2}  with $f(0)=0$.  We shall construct $L^2_{\rm loc}$ solutions to \eqref{irAC} under \eqref{u0-loc-H2} and \eqref{mp}. One may be able to weaken the assumption \eqref{u0-loc-H2} (see~\cite{ae1}). However, we shall here restrict ourselves to the case \eqref{u0-loc-H2} for simplicity. Here and henceforth, we are concerned with \emph{$L^2_{\rm loc}$ solutions} to \eqref{irAC} defined by
\begin{definition}[$L^2_{\rm loc}$ solution]\label{D:sol}
Let $T > 0$ be fixed. A measurable function $u : \R \times (0,T) \to \R$ is called an $L^2_{\rm loc}$ solution to \eqref{irAC} on $[0,T]$ if the following conditions are all satisfied\/{\rm :}
\begin{enumerate}
 \item[(i)] For any bounded interval $I$ over $\R$, it holds that
$$
u \in L^2(0,T;H^2(I)) \cap C_w([0,T];H^1(I)) \cap W^{1,2}(0,T;L^2(I))\/{\rm ;}
$$ 
 \item[(ii)] there exists $\eta \in L^2_{\rm loc}(\R \times [0,T])$ such that
\begin{equation}\label{eq:ac+}
\partial_t u + \eta - \partial_x^2 u + f(u) = 0, \quad \eta \in \partial \iI(\partial_t u) \ \mbox{ a.e.~in } \R \times (0,T)\/{\rm ;}
\end{equation}
 \item[(iii)] for any bounded interval $I$ over $\R$, $u(t) \to u_0$ strongly in $L^2(I)$ as $t \to +0$.
\end{enumerate}
Moreover, a function $u : \R \times (0,T) \to \R$ is called an \emph{$L^2_{\rm loc}$ subsolution} {\rm (}\emph{supersolution}, respectively{\rm )} to \eqref{irAC} on $[0,T]$, if all the conditions above hold with the equation in \eqref{eq:ac+} replaced by the inequality $\partial_t u + \eta - \partial_x^2 u + f(u) \leq 0$ {\rm (}respectively, $\partial_t u + \eta - \partial_x^2 u + f(u) \geq 0${\rm )}.
\end{definition}

One can also define the notion of $L^2_{\rm loc}$ solution (super- and subsolutions) for \eqref{pde-obs0} by replacing the inclusion $\eta \in \partial \iI(\partial_t u)$ of \eqref{eq:ac+} with $\eta \in \partial \iII(u)$.

\subsection{Comparison principle for $L^2_{\rm loc}$ solutions to obstacle problems}\label{Ss:cp}

Let us start with proving a comparison principle for $L^2_{\rm loc}$ solutions to obstacle problems.
\begin{lemma}[Comparison principle for $L^2_{\rm loc}$ solutions to obstacle problems\label{L:cp-obs+}]
 Assume that \eqref{f-2} is satisfied  with $f(0)=0$.  Let $I$ be a {\rm (}possibly unbounded{\rm )} interval in $\R$ and let $\phi, \psi : I \to \R$ be measurable functions. Let $u$ and $v$ be an $L^2_{\rm loc}$ sub- and an $L^2_{\rm loc}$ supersolution for \eqref{pde-obs0} with the obstacle function $u_0$ replaced by $\phi$ and $\psi$, respectively, in $Q_T = I \times (0,T)$ for some $T > 0$ such that $u-v$ is bounded from above over $Q_T$. Suppose that $u \leq v$ a.e.~on the parabolic boundary $\partial_p Q_T = (I \times \{0\}) \cup (\partial I \times [0,T))$ and $\phi \leq \psi$ a.e.~in $I$. Then it holds that
 $$
 u \leq v \ \mbox{ a.e.~in } Q_T.
 $$
\end{lemma}

\begin{proof}
 Set $I_r := (-r,r) \cap I$ for $r > 0$. Let $R > 0$ and let $\zeta_R : \R \to [0,+\infty)$ be a smooth cut-off function satisfying
$$
\zeta_R \equiv 1 \mbox{ on } (-R,R), \quad \zeta_R \equiv 0 \mbox{ on } \R \setminus (-2R,2R), \quad \|\zeta_R'\|_{L^\infty(\R)} \leq \dfrac{2}R.
$$
Subtracting inequalities and setting $w := u - v$, we see that
$$
 \partial_t w  - \partial_x^2 w + \beta(u) - \beta(v) \leq \lam w + \nu - \mu \ \mbox{ in } Q_T,
$$
where $\mu$ and $\nu$  belong to  $\partial I_{[\phi(x),+\infty)}(u)$ and $\partial I_{[\psi(x),+\infty)}(v)$, respectively, and $\beta(u) := f(u) + \lambda u$. Test both sides by $w_+ \zeta_R^2$,  where $w_+$ is the positive-part of $w$, i.e., $w_+ := \max\{w,0\}$. Then noting that $(\beta(u)-\beta(v))w_+ \geq 0$ by the monotonicity of $\beta$,  we have
\begin{align*}
\lefteqn{
\dfrac 1 2 \dfrac{\d}{\d t} \int_{I}w_+^2 \zeta_R^2 \, \d x 
+ \int_I |\partial_x w_+|^2 \zeta_R^2 \, \d x
+  2\int_I w_+ \partial_x w \zeta_R \zeta_R' \, \d x
}\\
& \leq \lam \int_{I} w_+^2 \zeta_R^2 \, \d x 
 + \int_I (\nu - \mu) w_+ \zeta_R^2 \, \d x \quad \mbox{ for a.e. } 0 < t < T.
\end{align*}
Here we observe by $\nu \leq 0$ and $\mu \in \partial I_{[\phi(x),+\infty)}(u)$ that
\begin{align*}
 \int_I (\nu - \mu)w_+ \zeta_R^2\, \d x
 &= \int_I \nu w_+  \zeta_R^2\, \d x - \int_I \mu w_+  \zeta_R^2\, \d x\\
 &\leq - \int_{\{u = \phi\} \cap \{u \geq v\}} \mu w  \zeta_R^2\, \d x.
 \end{align*}
Due to the fact that $v \geq \psi \geq \phi$ a.e.~in $I$,  either (i) or (ii)  holds: (i) the set $\{u = \phi\} \cap \{u \geq v\}$ has Lebesgue measure zero; (ii) $w = 0$ for a.e.~$x \in \{u = \phi\} \cap \{u \geq v\}$. Hence it follows that
 $$
  \int_{\{u = \phi\} \cap \{u \geq v\}} \mu w \zeta_R^2\, \d x
 = 0.
  $$
Combining all these facts, we deduce that
$$
 \int_I (\nu - \mu)w_+ \zeta_R^2 \, \d x \leq 0.
 $$
 Therefore multiplying both sides by $\e^{-2\lambda t}$, one has
\begin{align*}
\lefteqn{
\dfrac 1 2 \dfrac{\d}{\d t} \left( \e^{-2\lambda t} \int_{I}w_+^2 \zeta_R^2 \, \d x \right)
+ \e^{-2\lambda t} \int_{I} |\partial_x w_+|^2 \zeta_R^2 \, \d x
}\\
& \leq 2 \e^{-2\lambda t} \int_I |w_+||\partial_x w_+|\zeta_R|\zeta_R'| \, \d x\\
& \leq \e^{-2\lambda t} \frac{4} R \|w_+\|_{L^\infty(Q_T)} \left( \int_{I} |\partial_x w_+|^2 \zeta_R^2 \, \d x \right)^{1/2} \sqrt{4R}\\
& \leq \frac12 \e^{-2\lambda t} \int_{I} |\partial_x w_+|^2 \zeta_R^2 \, \d x + \dfrac{32}{R} \e^{-2\lambda t} \|w_+\|_{L^\infty(Q_T)}^2.
\end{align*}
Indeed,  by assumption,  $w = u-v$ is bounded from above over $Q_T$. Integrate both sides over $(0,t)$ and employ $w_+(\cdot,0) \equiv 0$ by assumption. Then one obtains
\begin{align}
\dfrac 1 2 \e^{-2\lambda t} \int_I w_+(\cdot,t)^2 \zeta_R^2 \, \d x + \frac12 \int^t_0 \e^{-2\lambda \tau} \left(\int_{I} |\partial_x w_+|^2 \zeta_R^2 \, \d x \right) \d \tau
\label{picopico}\\
\leq \dfrac{16}{\lambda R} \|w_+\|_{L^\infty(Q_T)}^2
\nonumber
\end{align}
for any $R > 0$ and $t > 0$. Thus passing to the limit as $R \to +\infty$, we conclude that
$$
\e^{-2\lambda t} \int_I w_+(\cdot,t)^2 \, \d x
+ \int^t_0 \e^{-2\lambda \tau} \left(\int_I |\partial_x w_+|^2 \, \d x \right) \d \tau = 0,
$$
whence follows that $w_+ \equiv 0$ (i.e., $u \leq v$) { a.e.~}in $Q_T$.
\end{proof}

\begin{remark}\label{R:equivalence}
{\rm
It is noteworthy that \eqref{pde-obs0} with $u_0 \equiv \alpha$ is equivalent to
\begin{equation}\label{obs-alpha}
\min\{u-\alpha,\partial_t u-\partial_x^2 u+f(u)\} = 0 \ \mbox{ in } Q_T. 
\end{equation}
More precisely, for functions $u$ satisfying (i) of Definition \ref{D:sol}, it holds that $u$ is a \emph{solution} of \eqref{pde-obs0} if and only if $u$ is a \emph{solution} of \eqref{obs-alpha}. Such an equivalence remains true for \emph{supersolutions}. Hence, Lemma \ref{L:cp-obs+} is  valid for comparing  an $L^2_{\rm loc}$ solution $u(x,t)$ and the supersolution $w^+(x,t)$ of \eqref{obs-alpha} (when $f'(\alpha) > 0$). On the other hand,  the equivalence  is no longer true for \emph{subsolutions}. Indeed, subsolutions to \eqref{pde-obs0} with $u_0 \equiv \alpha$ are still constrained to be greater than or equal to $\alpha$. However,  subsolutions of \eqref{obs-alpha} can violate the constraint.  

In latter sections, we shall often compare solutions of \eqref{obs-alpha} with the sub- and supersolutions $w^\pm$ defined by \eqref{wpm}. However, Lemma \ref{L:cp-obs+} is not valid for the subsolution $w^-$. Instead, the following lemma is useful to compare solutions of \eqref{irAC} or \eqref{pde-obs0}, which are supersolutions of \eqref{cac}, with $w^-$, which is a subsolution of \eqref{cac} on a positive half-line and is constant on the complementary negative half-line.
}
\end{remark}

\begin{lemma}[Comparison principle for the Allen-Cahn equation]\label{L:cp-sub}
Assume \eqref{f-2}  with $f(0)=0$.  Let $u^- = u^-(x,t)$ and $u^+ = u^+(x,t)$ be an $L^2_{\rm loc}$ sub- and supersolution to the classical Allen-Cahn equation \eqref{cac} on $Q^-_T := \{(x,t) \in \R \times (0,T) \colon x > d(t) \}$ for some $d(\cdot) \in C^1((0,T)) \cap C([0,T])$ for some $T > 0$ such that $u^--u^+$ is bounded from above in $Q^-_T$ and
$$
u^-(d(t),t) \leq u^+(d(t),t) \ \mbox{ for all } \ t \in [0,T].
$$
If $u^-(x,0) \leq u^+(x,0)$ for a.e.~$x \geq d(0)$, then $u^-(x,t)\leq u^+(x,t)$ for all~$x \geq d(t)$ and all $t \in [0,T]$.
\end{lemma}

The lemma above can be proved in a more or less standard way (see \S \ref{Apdx:S:cp} in the appendix). We emphasize that it is necessary for us to establish a comparison principle for $L^2_{\rm loc}$ solutions to \eqref{cac}, since solutions of \eqref{irAC} or \eqref{pde-obs0} are not always of class $C^2$ in space (see Theorem \ref{T:TW}).

\begin{remark}[Applications of comparison theorems]\label{R:comparison}
{\rm
\begin{enumerate}
 \item[(i)] Let $u$ be an $L^2_{\rm loc}$ solution to \eqref{pde-obs0}. Then $u(x,t)$ and $\phi_\alpha(x-c_\alpha t - x_0)$ are comparable, provided that they are ordered over $\R$ at $t = 0$. Indeed, one can check that $\phi_\alpha(x-c_\alpha t - x_0)$ is an $L^2_{\rm loc}$ solution to \eqref{pde-obs0} with $u_0 = \phi_\alpha(x-x_0)$, and therefore, Lemma \ref{L:cp-obs+} is applicable.
\item[(ii)] One can apply Lemma \ref{L:cp-sub} to an $L^2_{\rm loc}$ solution $u(x,t)$ and the subsolution $w^-(x,t)$ (defined by \eqref{wpm}) for \eqref{obs-alpha} in $Q := \R \times \R_+$ when $f'(\alpha)>0$. Indeed, $u(x,t)$ is a supersolution to \eqref{cac} on $Q$. Moreover, $w^-(x,t)$ is a subsolution of \eqref{cac} in $Q_- := \{(x,t) \in \R \times \R_+ \colon x > d(t) \}$ with $d(t) = c_\alpha t + x_0 + \sigma \delta(1-\e^{-\beta t})$. Hence thanks to Lemma \ref{L:cp-sub}, $w^-(x,t) \leq u(x,t)$ on $Q_-$,  provided that $w^-(\cdot,0) \leq u(\cdot,0)$ a.e.~in $\R$,  and moreover, the inequality also holds true on $Q \setminus Q_-$, where $w^-(x,t) \leq \alpha \leq u(x,t)$. Thus we obtain $w^- \leq u$ on $Q$.
\end{enumerate}
}
\end{remark}

\subsection{Proof of Theorem \ref{T:loc-L2-ex}}\label{Ss:loc-ex}

We are now in position to prove Theorem \ref{T:loc-L2-ex}. One can construct a sequence $(u_{0,n})$ in $C^\infty_c(\R)$ by setting $u_{0,n} := (\rho_n *  u_0) \zeta_n$ with a mollifier $\rho_n$ and a smooth cut-off $\zeta_n : \R \to [0,1]$ such that
$$
u_{0,n} \to u_0 \quad \mbox{ strongly in } H^2(I)
$$
for any bounded interval $I$ over $\R$. Then we observe that
\begin{equation}\label{u0n-snd}
M_- := \min\{\gamma_-,0\} \leq u_{0,n} \leq M_+ := \max\{\gamma_+,0\} \ \mbox{ a.e.~in } \R.
\end{equation}
As in~\cite{ae1}, we can prove:
\begin{enumerate}
 \item[(i)] there exists an $L^2$ solution $u_n$ to \eqref{irAC} with $u_0$ replaced by $u_{0,n}$ on $[0,T]$, i.e., 
$$
u_n \in L^2(0,T;H^2(\R)) \cap C_w([0,T];H^1(\R)) \cap W^{1,2}(0,T;L^2(\R)),
$$
 \item[(ii)] for any non-negative $\rho \in C^\infty_c(\R)$, it holds that
\begin{equation}\label{etan-bdd}
\esssup_{t \in (0,T)} \int_\R |\eta_n(\cdot,t)|^2 \rho \, \d x
\leq \int_\R | ( \partial_x^2 u_{0,n} - f(u_{0,n}) )_- |^2 \rho \, \d x,
\end{equation}
where $\eta_n$ is a section of $\partial \iI(\partial_t u_n)$ as in \eqref{eq:ac+},
 \item[(iii)] $u_n$ also solves \eqref{pde-obs0} with $u_0$ replaced by $u_{0,n}$ in $\R \times (0,T)$
\end{enumerate}
(see Appendix \S \ref{S:A0}). Furthermore, by \eqref{f-2}  with $f(0)=0$  and \eqref{mp}, $u^+ \equiv M_+$ and $u^-\equiv M_-$ are a super- and a subsolution,  respectively,  of \eqref{pde-obs0} with $u_0 := u(\cdot,0)$. Hence by maximum principle for (still possibly unbounded) $L^2$ solutions of \eqref{pde-obs0}, we can deduce that
\begin{equation}\label{loc:u-bdd}
M_- \leq u_n \leq M_+ \ \mbox{ a.e.~in } \R \times (0,T)
\end{equation}
(see Appendix \S \ref{S:A1}).

From now on, we shall establish local energy estimates. We omit the subscript $n$ unless any confusion may arise.  Set $I_r := (-r,r)$ for $r > 0$ and recall  $\zeta_R$ defined in the proof of Lemma \ref{L:cp-obs+}. Test  \eqref{irAC-dn}  by $u \zeta_R^2$. Then it follows that
\begin{align*}
\lefteqn{
 \dfrac 1 2 \dfrac \d {\d t} \int_\R |u|^2 \zeta_R^2 \, \d x
 + \int_\R |\partial_x u|^2 \zeta_R^2 \, \d x + \int_\R f(u)u \zeta_R^2 \, \d x
}\\
&= -  2 \int_\R \partial_x u \zeta_R \zeta_R' u \, \d x - \int_\R \eta u \zeta_R^2 \, \d x\\
&\leq \frac12 \int_\R |\partial_x u|^2 \zeta_R^2 \, \d x + \dfrac{32}{R}  \|u\|_{L^\infty(Q_T)}^2 \\
&\quad + \dfrac 1 2 \int_\R |u|^2 \zeta_R^2 \, \d x 
+ \dfrac 1 2 \int_\R |(\partial_x^2 u_{0} - f(u_{0}))_-|^2 \zeta_R^2 \, \d x.
\end{align*}
 Here we also used \eqref{etan-bdd}.  By assumption \eqref{f-2}  with $f(0)=0$,  we have
$$
f(u)u \geq - \lambda u^2.
$$
Set $\lambda' := \lambda + 1/2$.  It follows that
\begin{align*}
\lefteqn{
 \dfrac 1 2 \dfrac \d {\d t} \left( \e^{-2\lambda' t} \int_\R |u|^2 \zeta_R^2 \, \d x \right)
 + \dfrac12 \e^{-2\lambda' t} \int_\R |\partial_x u|^2 \zeta_R^2 \, \d x
}\\
&\leq \frac{32}{R} \e^{-2\lambda' t} \|u\|_{L^\infty(Q_T)}^2
+ \dfrac 1 2 \e^{-2\lambda' t} \int_\R |(\partial_x^2 u_{0} - f(u_{0}))_-|^2 \zeta_R^2 \, \d x.
\end{align*}
Integrating both sides over $(0,t)$, one derives from \eqref{loc:u-bdd} that
\begin{align}
 \lefteqn{
 \frac 12 \e^{-2\lambda' t} \|u(t)\|_{L^2(I_R)}^2 + \frac 12 \int^t_0 \e^{-2\lambda'\tau} \|\partial_x u\|_{L^2(I_R)}^2 \, \d \tau
}\label{loc:ux-L2}\\
&\leq \dfrac 1 2 \|u_0\|_{L^2(I_{2R})}^2
+ \dfrac{16}{\lam'R} \left(\max \{|M_+|,|M_-|\}\right)^2\nonumber\\
&\quad + \dfrac 1 {4\lambda'} \|(\partial_x^2 u_{0} - f(u_{0}))_-\|_{L^2(I_{2R})}^2
\nonumber
\end{align}
for any $t \in [0,T]$ and $R > 0$.

Moreover, test  \eqref{irAC-dn}  by $\partial_t u \zeta_R^2$. Since the product between $\partial_t u$ and $\eta$ vanishes, one finds that
\begin{align*}
\lefteqn{
 \int_\R |\partial_t u|^2 \zeta_R^2 \, \d x
 + \dfrac{\d}{\d t} \int_\R \left[ \frac 1 2 |\partial_x u|^2 + \hat f(u) \right] \zeta_R^2 \, \d x
}
\\
& = - 2\int_\R \partial_x u \zeta_R\zeta'_R \partial_t u \, \d x\\
& \leq \dfrac{4}R \|\partial_t u \zeta_R\|_{L^2(\R)} \|\partial_x u\|_{L^2(I_{2R})}\\
& \leq \frac12 \int_\R |\partial_t u|^2 \zeta_R^2 \, \d x + \dfrac{8}{R^2} \|\partial_x u\|_{L^2(I_{2R})}^2,
\end{align*}
where $\hat f$ is the primitive function of $f$ in \eqref{f-2} such that $\hat f \geq 0$. Since $\zeta_R \equiv 1$ on $I_R$, by integration in time, we infer that
\begin{align}
 \lefteqn{
 \frac12\int^t_0 \|\partial_t u\|_{L^2(I_R)}^2 \, \d \tau
 + \frac 1 2 \|\partial_x u(t)\|_{L^2(I_R)}^2 + \int^R_{-R} \hat f(u(x,t)) \, \d x 
}\label{loc:ut-L2}\\
&\leq \dfrac12\|\partial_xu_0\|_{L^2(I_{2R})}^2+ \int^{2R}_{-2R} \hat f(u_0(x)) \, \d x
+ \dfrac{8}{R^2} \int^t_0 \|\partial_x u\|_{L^2(I_{2R})}^2 \, \d \tau
\nonumber
\end{align}
for any $t \in [0,T]$ and $R > 0$.

Consequently, for any $R > 0$, thanks to \eqref{loc:ux-L2} and \eqref{loc:ut-L2}, there exists a constant $C_{T,R} > 0$  (also depending on $T$)  such that
\begin{align}\label{AA}
 \int^T_0 \|\partial_t u_n\|_{L^2(I_R)}^2 \, \d t
 + \sup_{t \in (0,T)} \|\partial_x u_n(t)\|_{L^2(I_R)}^2 \leq C_{ T,R}
\end{align}
for all $n \in \mathbb{N}$. Moreover, \eqref{etan-bdd} asserts that
$$
 \esssup_{t \in (0,T)}  \|\eta_n(t)\|_{L^2(I_R)} \leq C_R.
$$
Furthermore, by  \eqref{irAC-dn}  and a maximal regularity estimate, it follows that
$$
\int^T_0 \|\partial_x^2 u_n\|_{L^2(I_R)}^2 \, \d t
+ \int^T_0 \|\beta(u_n)\|_{L^2(I_R)}^2 \, \d t \leq C_{ T,R},
$$
where $\beta$ is a monotone function such that $f(u_n) = \beta(u_n) - \lambda u_n$, for any $n \in \mathbb N$.  By these estimates along with \eqref{loc:u-bdd},  there exists $u \in L^\infty(\R \times (0,T))$ such that, up to a (not relabeled) subsequence,
\begin{alignat*}{3}
u_n &\to u \quad && \mbox{ weakly star in } L^\infty(\R \times (0,T))
\end{alignat*}
and, for any $R > 0$,
\begin{alignat*}{3}
u_n &\to u \quad &&\mbox{ weakly in } W^{1,2}(0,T;L^2(I_R)) \cap L^2(0,T;H^2(I_R)),\\
& &&\mbox{ weakly star in } L^\infty(0,T;H^1(I_R)),\\
& &&\mbox{ strongly in } C(\overline{I_R} \times [0,T]) \cap L^2(0,T;C^1(\overline{I_R})),\\
\eta_n &\to \eta \quad &&\mbox{ weakly star in } L^\infty(0,T;L^2(I_R)),\\
f(u_n) &\to f(u) \quad &&\mbox{ strongly in } C(\overline{I_R} \times [0,T]),
\end{alignat*}
where $I_R := (-R,R)$. In particular, we remark that $\partial_t u \geq 0$ a.e.~in $Q_T$, and moreover, \eqref{eta-est} follows.

One finds that, for a.e.~$t \in (0,T)$, 
\begin{align*}
 \int_\R \partial_x u \partial_x \varphi \, \d x
 &= - \int_\R \left( \partial_t u + f(u) + \eta \right) \varphi \, \d x
\quad \mbox{ for any } \ \varphi \in C^\infty_c(\R),
\end{align*}
which implies $\partial_x u \in H^1_{\rm loc}(\R)$ (i.e., $u \in H^2_{\rm loc}$) by $\partial_t u + f(u) + \eta \in L^2_{\rm loc}(\R)$. Finally, we claim that
$$
\eta \in \partial \iI(\partial_t u) \ \mbox{ a.e.~in } Q_T. 
$$
To see this, for any $\varphi \in C^\infty_c(\R)$ and $v \in L^2_{\rm loc}(Q_T)$ satisfying $\varphi \geq 0$ and $v \geq 0$, we have
\begin{align*}
\int^T_0 \int_\R (v - \partial_t u_n) \eta_n \varphi \, \d x \d t \leq 0
\end{align*}
from the relation $\varphi \eta_n \in \partial \iI(\partial_t u_n)$. Passing to the limit as $n \to +\infty$, we observe that
\begin{align*}
0 &\geq \liminf_{n \to +\infty} \int^T_0 \int_\R (v - \partial_t u_n) \eta_n \varphi \, \d x \d t
\\
&\geq \int^T_0 \int_\R v \eta \varphi \, \d x \d t
- \limsup_{n \to +\infty} \int^T_0 \int_\R \partial_t u_n \eta_n \varphi \, \d x \d t.
\end{align*}
Here,  the multiplication of \eqref{irAC-dn} and $\partial_t u_n \varphi$ yields
\begin{align*}
\lefteqn{
 \int^T_0 \int_\R \partial_t u_n \eta_n \varphi \, \d x \d t
}\\
&= - \int^T_0 \int_\R |\partial_t u_n|^2 \varphi \, \d x \d t
 - \dfrac 1 2 \int_\R |\partial_x u_n(\cdot,T)|^2 \varphi \, \d x\\
&\quad - \int_\R \hat f(u_n(\cdot,T)) \varphi \, \d x
+ \dfrac 1 2 \int_\R |\partial_x u_{0,n}|^2 \varphi \, \d x
+ \int_\R \hat f(u_{0,n}) \varphi \, \d x\\
&\quad - \int^T_0 \int_\R \partial_t u_n \partial_x u_n \varphi' \, \d x \d t.
\end{align*}
Passing to the limit as $n \to +\infty$, we obtain
\begin{align*}
\lefteqn{
 \limsup_{n \to +\infty}\int^T_0 \int_\R \partial_t u_n \eta_n \varphi \, \d x \d t
}\\
&\leq - \int^T_0 \int_\R |\partial_t u|^2 \varphi \, \d x \d t
 - \dfrac 1 2 \int_\R |\partial_x u(\cdot,T)|^2 \varphi \, \d x\\
&\quad - \int_\R \hat f(u(\cdot,T)) \varphi \, \d x
+ \dfrac 1 2 \int_\R |\partial_x u_0|^2 \varphi \, \d x
+ \int_\R \hat f(u_0) \varphi \, \d x\\
&\quad - \int^T_0 \int_\R \partial_t u \partial_x u \varphi' \, \d x \d t\\
&= - \int^T_0 \int_\R  \partial_t u \left( \partial_t u - \partial_x^2 u + f(u) \right) \varphi \, \d x \d t
= \int^T_0 \int_\R \partial_t u \eta \varphi \, \d x \d t.
\end{align*}
Thus we obtain
$$
\int^T_0 \int_\R (v - \partial_t u) \eta \varphi \, \d x \d t \leq 0
$$
for any $v \in L^2_{\rm loc}(Q_T)$ satisfying $v \geq 0$ a.e.~in $Q_T$. Therefore $\eta\varphi$ belongs to the set $\partial \iI(\partial_t u)$ a.e.~in $\mathrm{supp} \,\varphi \times (0,T)$. Hence $\eta \in \partial \iI(\partial_t u)$ a.e.~in $Q_T$ by the arbitrariness of $\varphi \geq 0$. 

Finally, recalling that $\eta_n$ lies on $\partial I_{[\,\cdot\, \geq u_{0,n}]}(u_n) = \partial \iI(u_n-u_{0,n})$, we observe by the demiclosedness of maximal monotone graphs on $L^2(I_R)$ that $\eta \in \partial \iI(u-u_0) = \partial \iII(u)$ for a.e.~$(x,t) \in Q_T$.
This completes the proof. \qed

\section{Stability of traveling waves -- Outline --}\label{S:OL}

This section is devoted to an outline of the proof for Theorem \ref{T:Main}. Throughout this section, in addition to \eqref{f-1} and \eqref{f-2}, we always assume that $f'(\alpha)>0$, ${\rm (H)}_\alpha$ and \eqref{alpha-hyp-2} are all fulfilled. Before going into the detail, we first give a remark on assumptions for initial data. 

\begin{remark}[Hypotheses]\label{R:H-alpha}
{\rm We first give a remark on assumptions for initial data. As for \eqref{H1}, the second condition and the first inequality of the third condition are also assumed in the studies on classical Allen-Cahn equations. The condition $u_0 \in H^2_{\rm loc}(\R)$ can be relaxed as in~\cite{ae1}, where a \emph{partial smoothing effect} is discussed for \eqref{irAC}. The last inequality of the third condition of \eqref{H1} and \eqref{H3} enable us to avoid some technical difficulties; indeed, if they are not assumed, one may not be able to reduce \eqref{irAC} into a simple obstacle problem such as \eqref{pde-alpha} below, where the obstacle function is constant, and the coincidence set $[u=u_0]$ may also include a neighbourhood of $x=+\infty$.
}
\end{remark}

\subsection{Initial phase}

One readily observes by ${\bf (H)}_\alpha$ that there exists $\xi_0 \in \R$ such that $u_0(x) \leq \phi_\alpha(x-\xi_0)$ for $x \in \R$, and therefore, Lemma \ref{L:cp-obs+} (see also Remark \ref{R:comparison}) yields
\begin{equation}\label{ini-ph:order}
u_0(x) \leq u(x,t) \leq \phi_\alpha(x-c_\alpha t - \xi_0) \quad \mbox{ for all } \ x \in \R, \ t \geq 0.
\end{equation}
Hence one can set
$$
r(t) := \sup \left\{ r \in \R \colon u(x,t) = \alpha \ \mbox{ for all } \ x \leq r \right\}
$$
for $t \geq 0$. Then $r(t) \in [\xi_0 + c_\alpha t ,\xi_1]$ and it is non-increasing. Furthermore, it follows that
$$
u(\cdot,t) \equiv \alpha \ \mbox{ on } (-\infty,r(t)] \ \mbox{ for all } \ t \geq 0.
$$

We claim that 
\begin{lemma}[The infimum of $u(\cdot,t)$ is eventually greater than $a_-$]\label{Cl:Ph1}
There is a number $t_1 > 0$ large enough so that
 \begin{equation}\label{sol-above}
\inf_{x \in \R} u(x,t) \geq \inf_{x \in \R} u(x,t_1) \in (a_-,\alpha] \ \mbox{ for all } \ t \geq t_1.
 \end{equation}
\end{lemma}

\begin{proof}
We find by ${\bf (H)}_\alpha$ that
 \begin{equation}\label{rg1:u0}
 u(x,t) \geq u_0(x) = \alpha \in (a_-,a_0)
 \ \mbox{ for any } \ x \leq \xi_1 \ \mbox{ and } \ t \geq 0.
 \end{equation}
Let $u_{ac}$ be the solution of the classical Allen-Cahn equation \eqref{cac} with the same initial data $u_0$. Noting that
$$
\partial_t u = \big( \partial_x^2 u - f(u) \big)_+ \geq \partial_x^2 u - f(u) \ \mbox{ a.e.~in } \R \times \R_+
$$
(i.e., $u$ is a supersolution to \eqref{cac}) and employing a (classical) comparison principle, one has
 \begin{equation*}
 u_{ac}(x,t) \leq u(x,t) \ \mbox{ for all } \ (x,t) \in \R \times \R^+.  
 \end{equation*}
 Then thanks to~\cite{FM}, it follows by ${\bf (H)}_\alpha$ that
 $$
 \|u_{ac}(\cdot,t) - \phi_{ac}(\cdot -c t)\|_{L^\infty(\R)} \leq K_0 \e^{-\kappa_0 t} \ \mbox{ for all } \ t \geq 0
 $$
 for some constants $K_0, \kappa_0 > 0$, a velocity constant $c \in \R$ and a strictly increasing profile function $\phi_{ac}(x)$ satisfying
 $$
 - \phi''_{ac} + f(\phi_{ac}) = c \phi'_{ac} \ \mbox{ in } \R \ \  \mbox{ and } \ \lim_{x \to \pm\infty} \phi_{ac}(x) = a_\pm.
 $$
 Hence
 \begin{equation}\label{est-bel-fund}
  u(x,t) \geq u_{ac}(x,t) \geq \phi_{ac}(x-ct) - K_0 \e^{-\kappa_0 t} \ \mbox{ for all } \ t \geq 0.
 \end{equation}
 We can choose $t_1 > 0$ so large that
$$
K_0 \e^{-\kappa_0 t_1} < \frac{a_+ - a_0}2,
$$
 and moreover, there exists $\bar\xi_2 \in \R$ such that
$$
\phi_{ac}(\xi) > \frac{a_+ + a_0}2 \ \mbox{ for } \ \xi \geq \bar\xi_2. 
$$
Thereby
 $$
 u(x,t_1) \geq \phi_{ac}(x-ct_1) - K_0 \e^{-\kappa_0 t_1} > a_0
 $$
 for all $x \geq \xi_2 := \bar\xi_2+ct_1$. Thus combining this fact with \eqref{rg1:u0} and using the non-decrease of $u(x,t)$ in time, we obtain
 $$
 u(x,t) \geq \min \left\{ \alpha , \bar m \right\} > a_-
 \ \mbox{ for all } \ x \in \R \setminus (\xi_1, \xi_2) \ \mbox{ and } \ t \geq t_1,
 $$
where $\bar m := \phi_{ac}(\bar\xi_2)-K_0\e^{-\kappa_0t_1} > a_-$.

On the other hand, $v_{ac} := u_{ac} - a_-$ solves
$$
\partial_t v_{ac} = \partial_x^2 v_{ac} - f(u_{ac}) + f(a_-)
\geq \partial_x^2 v_{ac} - M v_{ac} \ \mbox{ in } \R \times \R_+,
$$
where $M := \sup_{s \in [a_- , a_+]} |f'(s)|$, since $u_{ac}$ lies on $[a_-,a_+]$. Let $w$ be the solution to 
$$
\partial_t w = \partial_x^2 w - M w \ \mbox{ in } \R \times \R_+, \quad w|_{t = 0} = u_0 - a_- \geq 0\ \mbox{ in } \R.
$$
Then $w$ satisfies
\begin{align*}
 w(x,t) &= \frac {\e^{-Mt}} {\sqrt{4\pi t}} \int_{\R} \e^{-\frac{(x-y)^2}{4t}} \left(u_0(y)-a_-\right) \, \d y\\
&\geq \frac {\e^{-Mt}} {\sqrt{4\pi t}} \left(\alpha - a_- \right) \int^{\xi_1}_{-\infty} \e^{-\frac{(x-y)^2}{4t}} \, \d y > 0
\end{align*}
for all $t > 0$ and $x \in \R$. Hence, by comparison principle,
$$
u_{ac}(x,t) \geq a_- + w(x,t) \ \mbox{ for all } \ t > 0 \ \mbox{ and } \ x \in \R,
$$
which in particular implies
$$
\inf_{x \in [\xi_1,\xi_2]} u(x,t) \geq \inf_{x \in [\xi_1,\xi_2]} u_{ac}(x,t) > a_- \ \mbox{ for } \ t > 0.
$$
Combining all these facts obtained so far with the non-decrease of $u(x,t)$ in time and \eqref{ini-ph:order}, we obtain \eqref{sol-above}.
\end{proof}

\subsection{Second phase}

In the second phase, based on Lemma \ref{Cl:Ph1}, we shall prove:

\begin{lemma}[The coincidence set is eventually a negative half-line]\label{Cl:Ph2}
There exists $t_2 \geq t_1$ such that
\begin{enumerate}
 \item[(i)] $u(x,t) \geq \alpha$ for $x \in \R$ and $t \geq t_2$,
 \item[(ii)] $u(x,t) > \max\{\sup_{x \in \R} u_0(x), a_+ - \frac{\delta_0}2\}$ for $x \geq \xi_1$ and $t \geq t_2$,
\end{enumerate}
where $\delta_0 > 0$ is  the  constant appeared in Lemma \ref{L:comfunc}. Hence, it holds that
\begin{equation}\label{u=u0-equi}
 u(x,t) = u_0(x) \quad \mbox{ if and only if } \quad u(x,t) = \alpha
\end{equation}
for all $x \in \R$ and $t \geq t_2$. Moreover, the coincidence set $[u(\cdot,t)=u_0(\cdot)]$ can be written as
\begin{equation}\label{coin-set}
[u(\cdot,t)=\alpha] = (-\infty,r(t)] \quad \mbox{ for } \ t \geq t_2.
\end{equation}
\end{lemma}

\begin{proof}
By \eqref{H1} and \eqref{sol-above}, one can take $\vep_0 \in (0,\delta_0/2)$ small enough that $u_0(x) \leq a_+ - \vep_0$ for all $x \in \R$ and $a_+ - \vep_0 > a_0$. Then there are $\tau_1 \geq 0$ and $\bar\xi_3 \in \R$ such that
$$
K_0 \e^{-\kappa_0 \tau_1} < \frac{\vep_0} 3, \quad
\phi_{ac}(\xi) > a_+ - \frac{\vep_0} 3 \ \mbox{ for } \ \xi \geq \bar\xi_3,
$$
which implies
\begin{align*}
u(x,\tau_1) \geq \phi_{ac}(x-c\tau_1) - K_0 \e^{-\kappa_0 \tau_1} 
\geq a_+ - \frac 2 3 \vep_0
\end{align*}
for all $x \geq \bar\xi_3 + c\tau_1$ and $t \geq \tau_1$. Here one can take $\tau_1 > t_1$ without any loss of generality. In particular, by the definition of $\vep_0$, we obtain $u(x,\tau_1) \geq a_+ - (2/3) \vep_0 > u_0(x)$ for all $x \geq \xi_3 := \bar\xi_3 + c\tau_1$, and therefore, the non-decrease of $u(x,t)$ in time yields that $u(x,t) \geq a_+ - (2/3) \vep_0 >  u_0(x)$ for $x \geq \xi_3$ and $t \geq \tau_1$. 

Let $m \in (a_-,a_0)$ be a constant close enough to $a_-$ such that $f'(m) > 0$ (see~\eqref{f-1}) and $m < \inf_{x \in \R} u(x,t_1) \in (a_-,\alpha]$ (see \eqref{sol-above}). By Theorem \ref{T:TW}, it then follows that \eqref{irAC} possesses a traveling wave $\phi_m$ with a negative velocity $c_m$ from the inequality,
$$
\int^{a_+}_m f(z) \, \d z \stackrel{\eqref{f-1}}< \int^{a_+}_{a_-} f(z) \, \d z \stackrel{\eqref{alpha-hyp-2}}\leq 0.
$$
Set $\delta = \vep_0 \in [0,\delta_0)$. Then due to Lemma \ref{L:comfunc} along with $f'(m) > 0$ the function $U(x,t) := \phi_m(x-\xi_3-c_m t - \sigma (1-\e^{-\beta t})) - \delta \e^{-\beta t}$ is a subsolution to \eqref{ac-w} in ${Q_-} := \{(x,t) \in \R \times \R_+ \colon x > \xi_3 + c_m t + \sigma (1-\e^{-\beta t})\}$, and moreover, $u(\cdot, \tau_1) \geq U(\cdot,0)$ in $\R$, since $u(\cdot,\tau_1) \geq a_+ - (2/3)\vep_0$ on $(\xi_3,+\infty)$ and $u(\cdot,\tau_1) \geq m$ in $\R$. Hence thanks to Lemma \ref{L:cp-sub} and $u \geq m$ in $\R \times (t_1,+\infty)$, we obtain
$$
u(x,\tau_1+t) \geq U(x,t) \quad \mbox{ for all } \ x \in \R \ \mbox{ and } \ t \geq 0.
$$
By $c_m < 0$, one can take $\tau_2 \geq 0$ such that
$$
U(x,t) > a_+ - \vep_0 \quad \mbox{ for all } \ x \geq \xi_1 \ \mbox{ and } \ t \geq \tau_2,
$$
which in particular yields 
$$
u(x,\tau_1+t) \geq U(x,t) > \max \left\{ \sup_{x \in \R}u_0(x), a_+ - \tfrac{\delta_0}2 \right\} \ \mbox{ for all } \ x \geq \xi_1 \ \mbox{ and } \ t \geq \tau_2.
$$
Noting that $u_0 \equiv \alpha$ on $(-\infty,\xi_1]$ and setting $t_2 := \tau_1 + \tau_2 > t_1$, we have obtained (i) and (ii). Hence, one can verify \eqref{u=u0-equi} immediately.

It still remains to prove \eqref{coin-set}. By definition of $r(t)$, we claim that
$$
u(x,t) > \alpha \ \mbox{ for all } x > r(t) \ \mbox{ and } \ t \geq t_2.
$$
Suppose on the contrary that there exist $t_* \geq t_2$ and $x_1 \in \R$ such that $x_1 > r(t_*)$, $u(x_1,t_*) = \alpha$ and $u(\cdot,t_*) > \alpha$ in $I := (r(t_*),x_1)$. Due to the preceding argument with $c_m < 0$, one can take $x_1 < \xi_1$. Let us consider the Cauchy-Dirichlet problem in $I \times (0,t_*)$ for \eqref{irAC} with the Dirichlet boundary condition $u(r(t_*),t) = u(x_1,t) = \alpha$. Since $u(\cdot,0) \equiv \alpha$ on $I$ by ${\rm (H)}_\alpha$, the unique solution of the Cauchy-Dirichlet problem is constant, i.e., $u \equiv \alpha$. However, it is a contradiction, since the solution $u(x,t)$ of \eqref{irAC} in $\R \times \R_+$ must coincide with that of the Cauchy-Dirichlet problem in $I \times (0,t_*)$ and $u$ is now supposed to be greater than $\alpha$ in $I$ at $t = t_*$. Therefore we conclude that the level set $[u(\cdot,t)=\alpha]$ coincides with the half-line $(-\infty,r(t)]$ for $t \geq t_2$.
\end{proof}

Now, we note by \eqref{u=u0-equi} that \eqref{pde-obs0} is equivalent to the constant obstacle problem,
\begin{equation}\label{pde-alpha}
 \partial_t u - \partial_x^2 u + f(u) + \eta = 0, \quad \eta \in \partial \iIIal(u)
\end{equation}
in $\R \times (t_2 , +\infty)$. Furthermore, $u$ solves the classical Allen-Cahn equation on $(r(t),+\infty)$ for any $t \geq t_2$. Indeed, we observe that $\eta(x,t) = 0$ for a.e.~$x > r(t)$ and all $t \geq t_2$.

\subsection{Third phase}\label{Ss:ph3}

This subsection is devoted to proving quasi-convergence of each solution for \eqref{irAC} to a traveling wave. More precisely, we shall prove:
\begin{lemma}[Quasi-convergence to a traveling wave]\label{L:quasi-conv}
There exist a sequence $t_n \to +\infty$ and $h_0 \in \R$ such that
$$
u(\cdot,t_n) - \phi_\alpha( \cdot -c_\alpha t_n -h_0) \to 0 \ \mbox{ uniformly on } \R.
 $$
Hence for any $\delta > 0$ there exists $n_\delta \in \mathbb N$ {\rm (}large enough{\rm )} such that
\begin{equation}\label{quasi-conv}
\phi_\alpha(x-c_\alpha t_{n_\delta}-h_0) - \delta \leq u(x,t_{n_\delta}) \leq \phi_\alpha(x-c_\alpha t_{n_\delta}-h_0) + \delta
\end{equation}
for all $x \in \R$.
\end{lemma}

 To prove this lemma, setting
$$
v(y,t) := u(y+c_\alpha t,t)-a_+ \ \mbox{ for } \ y,t \geq 0,
$$
we shall reduce the problem \eqref{pde-alpha} to a \emph{weighted} gradient system,
$$
\e^{c_\alpha y}\partial_t v - \partial_y (\e^{c_\alpha y} \partial_y v) + \e^{c_\alpha y} f(v+a_+) + \partial \iIIal(v+a_+) \ni 0,
$$
which enables us to establish global (in space) energy estimates for $v(\cdot,t)$ with exponentially decaying weights by $c_\alpha < 0$. Using uniform (in time) estimates for $v(\cdot,t)$, we shall verify the precompactness of the orbit $\{v(\cdot,t)+a_+ \colon t > 0\}$ locally in space, and moreover, we finally complete the proof by identifying accumulation points of the orbit as the (degenerate) traveling profile $\phi_\alpha$ with certain shifts and by improving the convergence to be uniform in space with the aid of the subsolution $w^-$. See \S \ref{S:Q-conv} for the details.

\begin{remark}[Convergence along the whole sequence]
{\rm
We are already in position to prove \emph{convergence} of $u(\cdot,t) - \phi_\alpha(\cdot - c_\alpha t - h_0)$ without any rate of convergence as $t \to +\infty$. Indeed, employing the sub- and supersolution $w^\pm$ (see Lemma \ref{L:comfunc}), one can verify that the limit of $v(\cdot,t)$ is unique (see~\cite[Lemmas 4.1 and 4.2]{FM}). 
}
\end{remark}

\subsection{Final phase}

In this phase, modifying a sub- and supersolution method developed in~\cite{XChen}, we shall verify
\begin{lemma}[Exponential convergence to traveling waves]\label{L:exp-stbl}
There exist constants $K, \kappa > 0$ and $x_0 \in \R$ such that
$$
\|u(\cdot,t)-\phi_\alpha(\cdot-c_\alpha t-x_0)\|_{L^\infty(\R)} \leq K \e^{-\kappa t} \quad \mbox{ for all } \ t \geq 0.
$$
Moreover, $r(t) - c_\alpha t$ converges to $x_0$ at the rate of $O(\e^{-\frac{\kappa t}2})$ as $t \to +\infty$. 
\end{lemma}

 The proof of the lemma mentioned above consists of three steps. First, we shall set up ``enclosing'' lemma (see Lemma \ref{L:enclo} below), which will be used to enclose the behavior of solutions and whose proof basically relies on~\cite{XChen} but turns out to be much more delicate due to the degeneracy of the traveling wave profile and sub- and supersolutions $w^\pm$. Indeed, there arises an additional requirement for the initial free boundary point $r_0$, i.e.,
$$
u(r_0,0) = \alpha, \quad  u(x,0) > \alpha \ \mbox{ for all } x > r_0.
$$
Here one may also find a reason why we need to consider the third phase before applying the enclosing lemma (cf.~Chen~\cite{XChen} used only sub- and supersolution methods for the classical (but possibly nonlocal) Allen-Cahn equation). Next, we shall incrementally apply the enclosing lemma to prove the \emph{exponential convergence of the solution} for \eqref{pde-alpha} to a traveling wave solution for \eqref{irAC} as in~\cite{XChen}. Here we also need to pay careful attention to the verification of the requirement for $r_0$ at each step. Finally, we shall prove the \emph{exponential convergence of the free boundary point} $r(t)$ of $u(\cdot,t)$ to that of the traveling wave. It is a novel aspect of the issue. Indeed, no such free boundary point appears in the study of the classical Allen-Cahn equation. Moreover, the $C^1$ regularity of $u$ will be also used to apply Taylor's theorem here (see also \S \ref{Ss:appl-c1} below). We refer the reader to \S \ref{S:exp-conv} for the details.

\begin{remark}[Stefan problem for the Allen-Cahn equation]\label{R:StAC}
{\rm
Stefan problems for reaction-diffusion equations posed on the half-line $(r(t),+\infty)$ are also studied in view of mathematical biology in~\cite{DuLin}, where motion equations of the free boundary $r(t)$ are prescribed of the form
$$
\dfrac{\d r}{\d t}(t) = - \mu \partial_x u(r(t),t) \quad \mbox{ for } \ t > 0
$$
for a positive constant $\mu$ and a solution $u(x,t)$. Hence the derivative $\partial_x u$ never vanishes on the free boundary. We also refer the reader to~\cite{KaYa11,DuGu11,DuGu12,DuLoZh15,DuMaZh15,KaMa15,KaMa18,KaYa18} and references therein. On the other hand, in the present issue, each solution $u(x,t)$ solves the equation \eqref{irAC} on the whole line $\R$ and the motion equation of the free boundary $r(t)$ defined by \eqref{r(t)} is not prescribed (cf.~\S \ref{Sss:moeq} below). Moreover, $\partial_x u(x,t)$ always vanishes at $x = r(t)$.
}
\end{remark}

\begin{remark}[Non-degenerate data]\label{R:non-deg}
{\rm
Exponential convergence never holds for initial data slowly converging to $\alpha$ at $x = -\infty$. Indeed, suppose that $u_0$ satisfies
$$
\lim_{x \to -\infty} u_0(x) = \alpha, \quad u_0' \geq 0 \ \mbox{ in } (-\infty,\bar x)
$$
for some $\bar x \in \R$. Then $u(\cdot,t) \geq u_0 \geq \alpha$ in $(-\infty,\bar x)$. It follows that
\begin{align*}
 \|u(\cdot,t)-\phi_\alpha(\cdot-c_\alpha t-x_0)\|_{L^\infty(\R)}
 &\geq \sup_{x \leq c_\alpha t + x_0} \left( u(x,t)-\uo \right)\\
 &\geq \sup_{x \leq c_\alpha t + x_0} \left( u_0(x) - \uo \right)
 = u_0(c_\alpha t + x_0) - \uo
\end{align*}
for $t$ large enough. Hence the decay rate of the difference $u(\cdot,t)-\phi_\alpha(\cdot-c_\alpha t - x_0)$ is bounded from below by that of $u_0(c_\alpha t + x_0)-\alpha$.
} 
\end{remark}

\section{$C^1$ regularity of solutions}\label{S:reg}

In this section, we shall give a proof for Theorem \ref{T:C1}. Throughout this section, we assume (only) \eqref{f-2} with $f(0)=0$, \eqref{u0-loc-H2} and \eqref{mp}.

\subsection{Formal argument}\label{Ss:he}

Before moving on to a proof, let us exhibit a formal argument to derive $C^1$ regularity of $u$, in particular, continuity of $\partial_t u$. In what follows, we often write $L^2$ and $H^1$ instead of $L^2(\R)$ and $H^1(\R)$, respectively. (Formally), differentiate \eqref{irAC} in time. Then
$$
\partial_t^2 u + \partial_t \eta + (1-\partial_x^2) \partial_t u + \beta'(u)  \partial_t u = \bar\lam \partial_t u
$$
with ${\bar\lam} := 1 + \lam$. Test it by $\partial_t^2 u$. It then follows that
\begin{align*}
 \lefteqn{
 \|\partial_t^2 u\|_{L^2}^2 + \underbrace{\left( \partial_t \eta, \partial_t^2 u\right)_{L^2}}_{\geq 0} + \dfrac 1 2 \dfrac \d {\d t} \|\partial_t u\|_{H^1}^2
 }\\
 &=  \dfrac {\bar\lam} 2 \dfrac{\d}{\d t} \|\partial_t u\|_{L^2}^2 - \int_\R \beta'(u) \partial_t u \partial_t^2 u \, \d x \\
 &\leq  \dfrac {\bar\lam} 2 \dfrac{\d}{\d t} \|\partial_t u\|_{L^2}^2 + \dfrac 1  2 \|\partial_t^2 u\|_{L^2}^2 + \frac 1 2 \|\beta'(u)\|_{L^\infty}^2 \|\partial_t u\|_{L^2}^2.
\end{align*}
Hence multiplying both sides by $t$ and integrating it over $(0,t)$, we observe that
\begin{align*}
 \lefteqn{
 \dfrac 1 2 \int^t_0 \tau\|\partial_t^2 u\|_{L^2}^2 \, \d \tau + \dfrac t 2 \|\partial_t u(t)\|_{H^1}^2
 }\\
 &\leq \dfrac {\bar\lam} 2 t\|\partial_t u(t)\|_{L^2}^2 + \frac 1 2 \|\beta'(u)\|_{L^\infty(Q)}^2 \int^t_0 \tau \|\partial_\tau u\|_{L^2}^2 \, \d \tau\\
 &\quad + \dfrac 1 2 \int^t_0 \|\partial_\tau u\|_{H^1}^2 \, \d \tau,
\end{align*}
where the right-hand side can be proved to be bounded for $t \in [0,T]$ (see \S 4 and \S 5 of~\cite{ae1}). Therefore $\partial_t u$ belongs to $W^{1,2}(\delta,T;L^2(I)) \cap L^\infty(\delta,T;H^1(I))$, which is compactly embedded in $C^\gamma([\delta,T];C^\gamma(\overline I))$ (see~\cite{Simon}), for any $\delta > 0$, bounded interval $I \subset \R$ and any $\gamma \in [0,1/2)$. However, the argument above is just formal and not rigorous. Indeed, $\eta$ is discontinuous on the interfacial curve over the $(x,t)$-plane.\footnote{For instance, $\eta$ corresponds to the traveling wave solution $u(x,t) = \phi_\alpha(x-c_\alpha t)$ is given by $\eta(x,t) = -f(\alpha)\chi_{(-\infty,c_\alpha t]}(x)$, where $\chi_{(-\infty,c_\alpha t]}$ stands for the characteristics function supported over $(-\infty,c_\alpha t]$.} Therefore some regularization is needed to justify the argument.

\subsection{Proof of Theorem \ref{T:C1} for compactly supported data}\label{Ss:scs}

We first restrict ourselves to the case $u_0 \in C^\infty_c(\R)$. Let us recall an approximation of \eqref{irAC} used in~\cite{ae1} and~\cite{Arai, Barbu75} (with appropriate modifications). Define a functional $\psi : L^2(\R) \to [0,+\infty]$ by
 \begin{equation}\label{varphi}
 \psi(w) := \begin{cases}
	     \frac 1 2 \|w\|_{H^1(\R)}^2 + \int_\R \hat\beta(w(x)) \, \d x &\mbox{ if } \ w \in H^1(\R), \ \hat \beta(w(\cdot)) \in L^1(\R),\\
	     +\infty &\mbox{ otherwise},
	    \end{cases}
 \end{equation}
where $\hat \beta$ is a primitive function of $\beta$ such that $\inf_{s \in \R} \hat \beta(s) = \hat \beta(0) = 0$.
 Then $\partial \psi(w)$ coincides with $w -\partial_x^2 w + \beta(w)$ in $L^2(\R)$ and $D(\partial \psi) = \{ w \in H^2(\R) \colon \beta(w) \in L^2(\R)\}$. Moreover, for each $\mu \in (0,1)$, there exists a unique strong solution $u_\mu \in C^{1,1}([0,T];L^2(\R))$ of the following approximate problem:
\begin{align}\label{aprx}
 \partial_t u_\mu + \eta_\mu + \partial \psi_\mu(u_\mu) &= {\bar\lam} u_\mu, \quad  \eta_\mu \in \partial \iI(\partial_t u_\mu), \quad 0 < t < T, \\
 u_\mu(0) &= u_0 \in C^\infty_c(\R) \subset D(\partial \psi),\label{ic-aprx}
\end{align}
 where $\partial \psi_\mu$ is the subdifferential operator of the \emph{Moreau-Yosida regularization} $\psi_\mu$ of $\psi$ (equivalently, the \emph{Yosida approximation} of $\partial \psi$), such that $\eta_\mu \in C^{0,1}([0,T];L^2(\R))$  (see, e.g.,~\cite{HB1}). Indeed, it is solvable in $L^2(\R)$ globally in time, since $\eqref{aprx}$ is rewritten as an evolution equation with Lipschitz nonlinearity (see \S 5 of~\cite{ae1}). Here, one can write
  \begin{equation}\label{Jlam}
  \partial \psi_\mu(w) = \partial \psi(J_\mu w) = (1- \partial_x^2) (J_\mu w) + \beta(J_\mu w) \quad \mbox{ for all } \ w \in L^2(\R), 
  \end{equation}
  where $J_\mu : L^2(\R) \to D(\partial \psi)$ stands for the \emph{resolvent} of $\partial \psi$, i.e., $J_\mu := (I + \mu \partial \psi)^{-1}$. Moreover, \eqref{aprx} is rewritten as
\begin{equation}\label{ap+}
 \partial_t u_\mu = \left( {\bar\lam} u_\mu - \partial \psi_\mu (u_\mu) \right)_+, \quad 0 < t < T.
\end{equation}

  As in~\cite{ae1}, testing \eqref{aprx} by $\partial_t u_\mu$ and using the relation $\eta_\mu \partial_t u_\mu \equiv 0$, we have
  $$
  \|\partial_t u_\mu\|_{L^2}^2 + \dfrac \d {\d t} \psi_\mu(u_\mu) = \dfrac {\bar\lam} 2 \dfrac \d {\d t} \|u_\mu\|_{L^2}^2.
  $$
  Integrating both sides over $(0,t)$, we find that
  \begin{align}\label{e0}
  \int^t_0 \|\partial_\tau u_\mu\|_{L^2}^2 \, \d \tau + \psi_\mu(u_\mu(t)) - \dfrac {\bar\lam} 2 \|u_\mu(t)\|_{L^2}^2 \leq E(u_0),
\end{align}
where $E(w) := \psi(w) - \frac {\bar\lam} 2 \|w\|_{L^2}^2$. Noting that
  $$
  \dfrac 1 2 \dfrac \d {\d t} \|u_\mu\|_{L^2}^2 = \left( u_\mu, \partial_t u_\mu \right)_{L^2}
  \leq \dfrac \vep 2 \|\partial_t u_\mu\|_{L^2}^2 + \dfrac 1 {2\vep} \|u_\mu\|_{L^2}^2
  $$
  for any $\vep > 0$ and putting $\vep = 2/({\bar\lam}+1)$, we infer that
  $$
  \dfrac 1 2 \dfrac \d {\d t} \|u_\mu\|_{L^2}^2 + \dfrac \d {\d t} \psi_\mu(u_\mu) \leq \dfrac 1 4 ({\bar\lam} + 1)^2 \|u_\mu\|_{L^2}^2.
  $$  
  Thus by Gronwall's inequality,
  $$
  \|u_\mu(t)\|_{L^2}^2 \leq \left[\|u_0\|_{L^2}^2+2\psi(u_0)\right] \e^{\frac 1 2({\bar\lam}+1)^2 t} \quad \mbox{ for } \ t \geq 0.
  $$
  Hence, by \eqref{e0},
   \begin{align}\label{e:psi}
    \lefteqn{
    \int^t_0 \|\partial_\tau u_\mu\|_{L^2}^2 \, \d \tau + \psi_\mu(u_\mu(t))
    }\\
    &\leq \dfrac {\bar\lam} 2 \left[\|u_0\|_{L^2}^2+2\psi(u_0)\right] \e^{\frac 1 2 ({\bar\lam}+1)^2 t} + E(u_0) \ \mbox{ for } \ t \geq 0.\nonumber
   \end{align}

  We next establish higher energy estimates for approximate solutions $u_\mu$. Differentiate \eqref{aprx} in $t$ (now, it is possible due to the regularization). We have
  \begin{equation}\label{Dut}
  \partial_t^2 u_\mu + \partial_t \eta_\mu + \dfrac{\d}{\d t} \partial \psi_\mu(u_\mu) = {\bar\lam} \partial_t u_\mu, \quad 0 < t < T.
  \end{equation}
  Multiply this by $\partial_t u_\mu$ to get
  \begin{align*}
   \dfrac 1 2 \dfrac \d {\d t} \|\partial_t u_\mu\|_{L^2}^2 + \underbrace{\left( \partial_t \eta_\mu, \partial_t u_\mu \right)_{L^2}}_{=0}
   + \left( \dfrac{\d}{\d t} \partial \psi_\mu(u_\mu) , \partial_t u_\mu \right)_{L^2}
   = {\bar\lam} \|\partial_t u_\mu\|_{L^2}^2.
  \end{align*}
  Here we used the fact that
  $$
  \left( \partial_t \eta_\mu, \partial_t u_\mu \right)_{L^2} = \dfrac \d {\d t} \iI^*(\eta_\mu) = 0
  $$
by $\partial_t u_\mu \in \partial \iI^*(\eta_\mu)$ and $\iI^*(\eta_\mu) = (\eta_\mu,\partial_t u_\mu)_{L^2} - \iI(\partial_t u_\mu) \equiv 0$ (see Notation). Moreover, noting that
  $$
  \partial_t v_\mu + \mu \left[(1-\partial_x^2)(\partial_t v_\mu) + \beta'(v_\mu) \partial_t v_\mu \right] = \partial_t u_\mu, 
  $$
  where $v_\mu := J_\mu u_\mu$ satisfies
  $v_\mu + \mu \left[ (1-\partial_x^2) v_\mu + \beta(v_\mu) \right] = u_\mu$ (see Lemma \ref{L:Jlam} below), we observe that
  \begin{align*}
   \lefteqn{
   \left( \dfrac{\d}{\d t} \partial \psi_\mu(u_\mu) , \partial_t u_\mu \right)_{L^2}
   }\\
   &= \left( (1-\partial_x^2)(\partial_t v_\mu), \partial_t v_\mu + \mu \left[(1-\partial_x^2)(\partial_t v_\mu) + \beta'(v_\mu) \partial_t v_\mu \right]\right)_{L^2}\\
   &\quad + \left( \beta'(v_\mu) \partial_t v_\mu, \partial_t u_\mu \right)_{L^2} =: I_1 + I_2.
  \end{align*}
  Then
  \begin{align*}
   I_1 &= \|\partial_t v_\mu\|_{H^1}^2 + \mu \left\|(1-\partial_x^2)(\partial_t v_\mu)\right\|_{L^2}^2
   + \mu \left((1-\partial_x^2)(\partial_t v_\mu), \beta'(v_\mu) \partial_t v_\mu \right)_{L^2}\\
   &\geq \|\partial_t v_\mu\|_{H^1}^2 + \mu \left\|(1-\partial_x^2)(\partial_t v_\mu)\right\|_{L^2}^2\\
   &\quad - \mu \|(1-\partial_x^2) (\partial_t v_\mu)\|_{L^2} \|\beta'(v_\mu)\|_{L^\infty} \|\partial_t v_\mu\|_{L^2}\\
   &\geq \|\partial_t v_\mu\|_{H^1}^2 + \dfrac \mu 2 \left\|(1-\partial_x^2)(\partial_t v_\mu)\right\|_{L^2}^2
   - \frac \mu 2 \|\beta'(v_\mu)\|_{L^\infty}^2 \|\partial_t v_\mu\|_{L^2}^2.
  \end{align*}
    Let us also employ the fact that
   \begin{align}
    \|v_\mu(t)\|_{L^\infty} &\leq C \|v_\mu(t)\|_{H^1}\label{e:v-inf}\\
    &\leq C \psi(v_\mu(t))^{1/2} \nonumber\\
    &\leq C \psi_\mu(u_\mu(t))^{1/2} \stackrel{\eqref{e:psi}}\leq C_T
    \quad \mbox{ for } \ t \in [0,T].\nonumber
   \end{align}
   Here and henceforth, $C_T$ denotes a constant independent of $\mu$ (but may depend on $T$) and may vary from line to line. From the non-expansivity of $J_\mu$ in $L^2$, one has $\|\partial_t v_\mu(t)\|_{L^2} \leq \|\partial_t u_\mu(t)\|_{L^2}$, and thus, we deduce that
  \begin{align*}
   I_1 &\geq \|\partial_t v_\mu\|_{H^1}^2 + \dfrac \mu 2 \left\|(1-\partial_x^2)(\partial_t v_\mu)\right\|_{L^2}^2
   - C_T \mu \|\partial_t u_\mu\|_{L^2}^2.
  \end{align*}
  On the other hand, it follows by \eqref{e:v-inf} that
  \begin{align*}
   I_2 &\geq - \|\beta'(v_\mu)\|_{L^\infty} \|\partial_t v_\mu\|_{L^2} \|\partial_t u_\mu\|_{L^2} \geq - C_T \|\partial_t u_\mu\|_{L^2}^2.
  \end{align*}
  Combining all these facts, we obtain
   \begin{align*}
    \dfrac 1 2 \dfrac \d {\d t} \|\partial_t u_\mu\|_{L^2}^2 + \|\partial_t v_\mu\|_{H^1}^2 + \dfrac \mu 2 \left\|(1-\partial_x^2)(\partial_t v_\mu)\right\|_{L^2}^2 \quad
    \\
    \leq \left[ {\bar\lam} + C_T (\mu+1) \right] \|\partial_t u_\mu\|_{L^2}^2.
   \end{align*}
Here we note that $\|\partial_t u_\mu(0)\|_{L^2}^2 = \|( {\bar\lam} u_0 - \partial \psi_\mu (u_0))_+\|_{L^2}^2$, which is bounded for $\mu$. Furthermore, by Gronwall's inequality, we obtain
\begin{align}
\sup_{t \in [0,T]} \|\partial_t u_\mu(t)\|_{L^2}^2 \leq C_T,
\label{he0}
\end{align}
and moreover,
\begin{align}
\int^T_0 \|\partial_t v_\mu\|_{H^1}^2 \, \d t + \dfrac \mu 2 \int^T_0 \left\|(1-\partial_x^2)(\partial_t v_\mu)\right\|_{L^2}^2 \, \d t \leq C_T. \label{he1}
\end{align}

Test \eqref{Dut} by $\partial_t^2 u_\mu \in L^\infty(0,T;L^2(\R))$. Then by the monotonicity of $\partial \iI$,
  \begin{align*}
   \|\partial_t^2 u_\mu\|_{L^2}^2 + \underbrace{\left(\partial_t \eta_\mu, \partial_t^2 u_\mu\right)_{L^2}}_{\geq 0}
   + \left( \dfrac{\d}{\d t} \partial \psi_\mu(u_\mu) 
   , \partial_t^2 u_\mu \right)_{L^2}
   = \dfrac {\bar\lam} 2 \dfrac{\d}{\d t} \|\partial_t u_\mu\|_{L^2}^2.
  \end{align*}
  We note that
   \begin{align*}
    \lefteqn{
   \left( \dfrac{\d}{\d t} \partial \psi_\mu(u_\mu), \partial_t^2 u_\mu \right)_{L^2}
    }\\
    &= \left( (1- \partial_x^2) (\partial_t v_\mu) , \partial_t^2 v_\mu + \mu \left[ (1-\partial_x^2)(\partial_t^2 v_\mu) + \beta''(v_\mu) (\partial_t v_\mu)^2 + \beta'(v_\mu) \partial_t^2 v_\mu \right] \right)_{L^2}\\
   &\quad + \left( \beta'(v_\mu) \partial_t v_\mu, \partial_t^2 u_\mu \right)_{L^2} =: J_1 + J_2.
   \end{align*}
   Here we used the fact that $v_\mu \in C^{1,1}([0,T];H^2(\R))$ satisfies
   \begin{equation}
    \partial_t^2 u_\mu = \partial_t^2 v_\mu + \mu \left[ (1-\partial_x^2) (\partial_t^2 v_\mu) + \beta''(v_\mu) (\partial_t v_\mu)^2 + \beta'(v_\mu) \partial_t^2 v_\mu \right]
   \end{equation}
   (see Lemma \ref{L:Jlam} below). Then, by simple calculation and \eqref{e:v-inf},
   \begin{align*}
    J_1 &= \dfrac 1 2 \dfrac \d {\d t} \|\partial_t v_\mu\|_{H^1}^2 + \dfrac \mu 2 \dfrac \d {\d t} \| (1 - \partial_x^2) (\partial_t v_\mu)\|_{L^2}^2 
\\
    &\quad + \mu \left( (1 - \partial_x^2) (\partial_t v_\mu) , \beta''(v_\mu) (\partial_t v_\mu)^2 \right)_{L^2}
    + \mu \left( (1 - \partial_x^2) (\partial_t v_\mu), \beta'(v_\mu) \partial_t^2 v_\mu \right)_{L^2}\\
    &\geq \dfrac 1 2 \dfrac \d {\d t} \|\partial_t v_\mu\|_{H^1}^2 + \dfrac \mu 2 \dfrac \d {\d t} \|(1-\partial_x^2) (\partial_t v_\mu)\|_{L^2}^2 \\
    &\quad - \mu \|(1-\partial_x^2)(\partial_t v_\mu)\|_{L^2} \|\beta''(v_\mu)\|_{L^\infty}\|\partial_t v_\mu\|_{L^4}^2\\
    &\quad - \mu \|(1-\partial_x^2)(\partial_t v_\mu)\|_{L^2} \|\beta'(v_\mu)\|_{L^\infty} \|\partial_t^2 v_\mu\|_{L^2}\\
    &\geq \dfrac 1 2 \dfrac \d {\d t} \|\partial_t v_\mu\|_{H^1}^2 + \dfrac \mu 2 \dfrac \d {\d t} \|(1-\partial_x^2) (\partial_t v_\mu)\|_{L^2}^2\\
    &\quad - C_T \mu \|(1-\partial_x^2) (\partial_t v_\mu)\|_{L^2} \|\partial_t v_\mu\|_{L^2} \|\partial_t v_\mu\|_{H^1}\\
    &\quad - C_T \mu \|(1-\partial_x^2) (\partial_t v_\mu)\|_{L^2} \|\partial_t^2 v_\mu\|_{L^2}.
   \end{align*}
   By Young's inequality,
   \begin{align*}
    J_1 &\geq \dfrac 1 2 \dfrac \d {\d t} \|\partial_t v_\mu\|_{H^1}^2 + \dfrac \mu 2 \dfrac \d {\d t} \|(1-\partial_x^2) (\partial_t v_\mu)\|_{L^2}^2\\
    &\quad - \mu \|(1-\partial_x^2) (\partial_t v_\mu)\|_{L^2}^2
    -  C_T \mu\left( \|\partial_t v_\mu\|_{H^1}^2 \|\partial_t u_\mu\|_{L^2}^2 + \|\partial_t^2 v_\mu\|_{L^2}^2 \right).
   \end{align*}
   Moreover, we observe that
   \begin{align*}
    J_2 \geq - \|\beta'(v_\mu)\|_{L^\infty} \|\partial_t v_\mu\|_{L^2} \|\partial_t^2 u_\mu \|_{L^2}
    \geq - C_T \|\partial_t u_\mu\|_{L^2}^2 - \dfrac 1 2 \|\partial_t^2 u_\mu \|_{L^2}^2.
\end{align*}
Combining all these facts, we deduce that
\begin{align}
    \lefteqn{
    \dfrac 1 2 \|\partial_t^2 u_\mu \|_{L^2}^2 + \dfrac 1 2 \dfrac \d {\d t} \|\partial_t v_\mu\|_{H^1}^2 + \dfrac \mu 2 \dfrac \d {\d t} \|(1-\partial_x^2) (\partial_t v_\mu)\|_{L^2}^2
    }\label{high-en-02}\\
    &\leq \dfrac {\bar\lam} 2 \dfrac \d {\d t} \|\partial_t u_\mu\|_{L^2}^2
    + \mu \|(1-\partial_x^2) (\partial_t v_\mu)\|_{L^2}^2 \nonumber\\
    &\quad + C_T \mu \left( \|\partial_t v_\mu\|_{H^1}^2 \|\partial_t u_\mu\|_{L^2}^2 + \|\partial_t^2 v_\mu\|_{L^2}^2 \right) + C_T \|\partial_t u_\mu\|_{L^2}^2.\nonumber
\end{align}
By Lemma \ref{L:Jlam} below, we find that
\begin{align*}
\|\partial_t^2 v_\mu\|_{L^2}^2 \leq C\|\partial_t^2 u_\mu\|_{L^2}^2 + C\mu \|\beta''(v_\mu)\|_{L^\infty}^2\|\partial_t v_\mu\|_{L^2}^2 \|\partial_t v_\mu\|_{H^1}^2.
\end{align*}
For any $\mu > 0$ small enough, one can obtain
   \begin{align}
    \lefteqn{
    \dfrac 1 4 \|\partial_t^2 u_\mu \|_{L^2}^2 + \dfrac 1 2 \dfrac \d {\d t} \|\partial_t v_\mu\|_{H^1}^2 + \dfrac \mu 2 \dfrac \d {\d t} \|(1-\partial_x^2) (\partial_t v_\mu)\|_{L^2}^2
    }\label{star1}\\
    &\leq \dfrac {\bar\lam} 2 \dfrac \d {\d t} \|\partial_t u_\mu\|_{L^2}^2
    + \mu \|(1-\partial_x^2)(\partial_t v_\mu)\|_{L^2}^2\nonumber\\
    &\quad  + C_T \mu \|\partial_t v_\mu\|_{H^1}^2 \|\partial_t u_\mu\|_{L^2}^2 + C_T \|\partial_t u_\mu\|_{L^2}^2.\nonumber
   \end{align}

\begin{remark}[Regularization and error terms]
{\rm Thanks to a good regularization, one can overcome the lack of differentiability of $\eta$, but still keep the convenient relation $\eta_\mu \partial_t u_\mu \equiv 0$, which may be lost by other regularizations. On the other hand, some singularities arise as the term $\mu \|(1-\partial_x^2) (\partial_t v_\mu)\|_{L^2}^2$ and the difference between $u_\mu$ and $v_\mu$; however, they can be controlled by using additional energy estimates as well as resolvent estimates (see Lemma \ref{L:Jlam} below).}
\end{remark}
   
Multiply both sides of \eqref{star1} by $t$. Then
   \begin{align*}
    \lefteqn{
    \dfrac t 4 \|\partial_t^2 u_\mu \|_{L^2}^2 + \dfrac 1 2 \dfrac \d {\d t} \left( t \|\partial_t v_\mu\|_{H^1}^2 \right) + \dfrac \mu 2 \dfrac \d {\d t} \left(t\|(1-\partial_x^2) (\partial_t v_\mu)\|_{L^2}^2\right) + \dfrac {\bar\lam} 2 \|\partial_t u_\mu\|_{L^2}^2
    }\\
    &\leq \dfrac {\bar\lam} 2 \dfrac \d {\d t} \left(t\|\partial_t u_\mu\|_{L^2}^2\right)
    + \mu t\|(1-\partial_x^2) (\partial_t v_\mu)\|_{L^2}^2 + C_T \mu t \|\partial_t v_\mu\|_{H^1}^2 \|\partial_t u_\mu\|_{L^2}^2\\
    &\quad + C_T t \|\partial_t u_\mu\|_{L^2}^2 + \dfrac 1 2 \|\partial_t v_\mu\|_{H^1}^2 + \dfrac \mu 2 \|(1-\partial_x^2) (\partial_t v_\mu)\|_{L^2}^2.
   \end{align*}
 Integrating both sides over $(0,t)$, we deduce by \eqref{he0} and \eqref{he1} that
   \begin{align*}
    \lefteqn{
    \int^t_0 \dfrac \tau 4 \|\partial_t^2 u_\mu \|_{L^2}^2 \, \d \tau + \dfrac t 2 \|\partial_t v_\mu(t)\|_{H^1}^2 + \dfrac \mu 2 t\|(1-\partial_x^2) (\partial_t v_\mu)(t)\|_{L^2}^2 }\\
    &\leq \dfrac {\bar\lam} 2 t\|\partial_t u_\mu(t)\|_{L^2}^2
    + \mu \int^T_0 \tau \|(1-\partial_x^2) (\partial_t v_\mu)\|_{L^2}^2 \, \d \tau \\
    &\quad + C_T \mu \left(\sup_{\tau \in [0,T]} \|\partial_t u_\mu(\tau)\|_{L^2}^2\right) T \int^T_0 \|\partial_t v_\mu\|_{H^1}^2 \, \d \tau\\
    &\quad + C_T \int^t_0 \tau \|\partial_t u_\mu\|_{L^2}^2 \, \d \tau
    + \dfrac 1 2 \int^T_0 \|\partial_t v_\mu\|_{H^1}^2 \, \d t\\
    &\quad + \dfrac \mu 2 \int^T_0 \left\|(1-\partial_x^2)(\partial_t v_\mu)\right\|_{L^2}^2 \, \d t
    \stackrel{\eqref{he0},\eqref{he1}}\leq C_T.
   \end{align*}
 Therefore we infer that
   \begin{alignat*}{4}
    \sqrt{t} \partial_t^2 u_\mu &\to \sqrt{t} \partial_t^2 u \quad &&\mbox{ weakly in } L^2(0,T;L^2(\R)),\\
    \sqrt{t} \partial_t v_\mu &\to \sqrt{t} \partial_t u \quad &&\mbox{ weakly star in } L^\infty(0,T;H^1(\R)),
   \end{alignat*}
   which implies $\partial_t u \in L^\infty(\delta,T;H^1(\R)) \cap W^{1,2}(\delta,T;L^2(\R))$ for any $\delta > 0$. Thus by embeddings (see~\cite{Simon}), we conclude that
   $$
   \partial_t u \in C^\gamma(\R \times (0,T]) \ \mbox{ for any } \gamma \in [0,1/2).
   $$
   On the other hand, as in~\cite[\S 3-5]{ae1} one can verify that
   $$
   u \in L^\infty(0,T;H^2(\R)) \cap W^{1,2}(0,T;H^1(\R)),
   $$
   which implies
   $$
   \partial_x u \in C^\gamma(\R\times[0,T]) \ \mbox{ for any } \gamma \in [0,1/2).
$$
Thus $u \in C^{1,\gamma}(\R \times (0,T])$ for $\gamma \in (0,1/2)$. 
This ensures the assertion of Theorem \ref{T:C1} for compactly supported data. 

To close this subsection, let us prove the following lemma:
   \begin{lemma}[Resolvent estimates]\label{L:Jlam}
    Let $u \in C^{1,1}([0,T];L^2(\R))$ be fixed. Then $v_\mu := J_\mu u(\cdot)$ belongs to $C^{1,1}([0,T];H^2(\R))$, and moreover, it holds that $\beta(v_\mu) \in C^{1,1}([0,T];L^2(\R))$ and
    \begin{align*}
     \partial_t v_\mu + \mu \left[ (1-\partial_x^2)(\partial_t v_\mu) + \beta'(v_\mu) \partial_t v_\mu \right] &= \partial_t u,\\
     \partial_t^2 v_\mu + \mu \left[ (1-\partial_x^2)(\partial_t^2 v_\mu) + \beta''(v_\mu) (\partial_t v_\mu)^2 + \beta'(v_\mu) \partial_t^2 v_\mu \right] &= \partial_t^2 u.
    \end{align*}
    Moreover, it holds that
    \begin{equation}
     \|\partial_t^2 v_\mu\|_{L^2}^2 \leq C\|\partial_t^2 u\|_{L^2}^2 + C\mu \|\beta''(v_\mu)\|_{L^\infty}^2\|\partial_t v_\mu\|_{L^2}^2 \|\partial_t v_\mu\|_{H^1}^2
    \end{equation}
    for some $C$ independent of $\mu > 0$ small enough.
    \end{lemma}

   \begin{proof}
    Since $J_\mu$ is non-expansive in $L^2(\R)$, it holds that $v_\mu \in C^{0,1}([0,T];L^2(\R))$ and $\|\partial_t v_\mu(t)\|_{L^2} \leq \|\partial_t u(t)\|_{L^2}$. By the definition of $J_\mu$, one has
   $$
   v_\mu(t) + \mu \partial \psi(v_\mu(t)) = u(t).
    $$
    Test it by $v_\mu(t)$ to find that $v_\mu \in L^\infty(0,T;H^1(\R)) \subset L^\infty(\R\times (0,T))$. Moreover, as both $u$ and $v_\mu$ are Lipschitz continuous on $[0,T]$ with values in $L^2(\R)$, so is $\partial \psi_\mu(u(\cdot))$. Moreover, we find that
   \begin{align*}
    \|\beta(v_\mu(\cdot,t))-\beta(v_\mu(\cdot,s))\|_{L^2}^2
    &\leq \sup_{|z|\leq C_T} |\beta'(z)|^2 \|v_\mu(t)-v_\mu(s)\|_{L^2}^2,
   \end{align*}
    where $C_T$ is a constant greater than $\|v_\mu\|_{L^\infty(Q_T)}$ for any $\mu > 0$ and which yields that $\beta(v_\mu) \in C^{0,1}([0,T];L^2(\R))$, and hence, so is $(1-\partial_x^2) v_\mu$. Thus $v_\mu$ belongs to $C^{0.1}([0,T];H^2(\R))$.

    Differentiating both sides in $t$, one has
   \begin{align}\label{vlam}
    \partial_t v_\mu + \mu \left[ (1-\partial_x^2)(\partial_t v_\mu) + \beta'(v_\mu) \partial_t v_\mu \right] = \partial_t u.
   \end{align}
    Multiplying both sides by $\partial_t v_\mu$, we infer that $\partial_t v_\mu \in L^\infty(0,T;H^1(\R)) \subset L^\infty(\R \times (0,T))$.
    By \eqref{vlam}, one finds that
    \begin{align*}
     \lefteqn{
   \partial_t v_\mu(t) - \partial_t v_\mu(s) + \mu (1-\partial_x^2) \left[ \partial_t v_\mu(t) - \partial_t v_\mu(s) \right]}\\
     &+ \mu \left[ \beta'(v_\mu(t)) \partial_t v_\mu(t) - \beta'(v_\mu(s)) \partial_t v_\mu(s) \right]
   =\partial_t u(t) - \partial_t u(s).
    \end{align*}
   Test it by the difference $\partial_t v_\mu(t) - \partial_t v_\mu(s)$. It then follows that
   \begin{align*}
    \lefteqn{
    \left\|\partial_t v_\mu(t) - \partial_t v_\mu(s)\right\|_{L^2}^2 + \mu\left\|\partial_t v_\mu(t) - \partial_t v_\mu(s)\right\|_{H^1}^2
    }\\
    &\quad  + \mu \int_\R \beta'(v_\mu(x,t)) \left| \partial_t v_\mu(x,t) - \partial_t v_\mu(x,s) \right|^2 \, \d x\\
    &\leq \left\| \partial_t u(t) - \partial_t u(s) \right\|_{L^2} \left\| \partial_t v_\mu(t) - \partial_t v_\mu(s) \right\|_{L^2}\\
    &\quad + \mu \|\beta''(v_\mu(t))\|_{L^\infty} \|v_\mu(t)-v_\mu(s)\|_{L^2} \|\partial_t v_\mu(s)\|_{L^\infty} \left\| \partial_t v_\mu(t) - \partial_t v_\mu(s) \right\|_{L^2},
   \end{align*}
   which leads us to obtain $\partial_t v_\mu \in C^{0,1}([0,T];H^1(\R))$. Hence $(1-\partial_x^2) (\partial_t v_\mu) + \beta'(v_\mu) \partial_t v_\mu = - \mu^{-1}( \partial_t v_\mu - \partial_t u ) \in C^{0,1}([0,T];L^2(\R))$. Furthermore, let us observe that
    \begin{align*}
     \lefteqn{
    \left\| \beta'(v_\mu(t)) \partial_t v_\mu(t) - \beta'(v_\mu(s)) \partial_t v_\mu(s) \right\|_{L^2}
     }\\
     &\leq \sup_{|z|\leq C_T} |\beta''(z)| \left\| v_\mu(t)-v_\mu(s)\right\|_{L^2} \|\partial_t v_\mu(t)\|_{L^\infty}\\
     &\quad + \sup_{|z| \leq C_T}|\beta'(z)| \left\| \partial_t v_\mu(t)-\partial_t v_\mu(s)\right\|_{L^2},
    \end{align*}
    whence follows $\beta'(v_\mu) \partial_t v_\mu \in C^{0,1}([0,T];L^2(\R))$. So is $(1-\partial_x^2) (\partial_t v_\mu)$, and therefore, $\partial_t v_\mu \in C^{0,1}([0,T];H^2(\R))$. Hence we can differentiate both sides of \eqref{vlam} in $t$. Then
\begin{equation}\label{vlam2}
\partial_t^2 v_\mu + \mu \left[ (1-\partial_x^2) (\partial_t^2 v_\mu) + \beta''(v_\mu) (\partial_t v_\mu)^2 + \beta'(v_\mu) \partial_t^2 v_\mu \right] = \partial_t^2 u.
\end{equation}
Multiply it by $\partial_t^2 v_\mu$. We have
\begin{align*}
\|\partial_t^2 v_\mu\|_{L^2}^2 + \mu \|\partial_t^2 v_\mu\|_{H^1}^2 + \mu \int_\R \beta''(v_\mu) (\partial_t v_\mu)^2 \partial_t^2 v_\mu \, \d x + \mu \int_\R \beta'(v_\mu) (\partial_t^2 v_\mu)^2 \, \d x
\\
\leq \frac 1 2 \|\partial_t^2 u\|_{L^2}^2 + \frac 1 2 \|\partial_t^2 v_\mu\|_{L^2}^2.
\end{align*}
Then
\begin{align*}
    \left|\int_\R \beta''(v_\mu) (\partial_t v_\mu)^2 \partial_t^2 v_\mu \, \d x\right|
    &\leq \|\beta''(v_\mu)\|_{L^\infty} \|\partial_t v_\mu\|_{L^4}^2 \|\partial_t^2 v_\mu\|_{L^2}\\
    &\leq C \|\beta''(v_\mu)\|_{L^\infty} \|\partial_t v_\mu\|_{H^1} \|\partial_t v_\mu\|_{L^2} \|\partial_t^2 v_\mu\|_{L^2}\\
    &\leq \|\partial_t^2 v_\mu\|_{L^2}^2 + C \|\beta''(v_\mu)\|_{L^\infty}^2\|\partial_t v_\mu\|_{L^2}^2 \|\partial_t v_\mu\|_{H^1}^2.
   \end{align*}
   Therefore
   $$
   \|\partial_t^2 v_\mu\|_{L^2}^2 \leq C \|\partial_t^2 u\|_{L^2}^2 + C\mu \|\beta''(v_\mu)\|_{L^\infty}^2\|\partial_t v_\mu\|_{L^2}^2 \|\partial_t v_\mu\|_{H^1}^2
    $$
    for $\mu > 0$ small enough. This completes the proof.
   \end{proof}

\subsection{Proof of Theorem \ref{T:C1} for general data}\label{Ss:C1gd}

We first derive a couple of energy inequalities with weights for later use. Let $u=u(x,t)$ be an $L^2$ solution to \eqref{irAC} and \eqref{pde-obs0} for $u_0 \in C^\infty_c(\R)$. By subtraction, we find that
\begin{align}
 \partial_t \left[u(t+h)-u(t)\right] + \eta(t+h)-\eta(t) - \partial_x^2 \left[u(t+h)-u(t)\right]\qquad
\label{subt-eq}\\
 + \beta(u(t+h))-\beta(u(t)) = \lam \left[u(t+h)-u(t)\right].\nonumber
\end{align}
Test it by $\partial_t u(t+h) \zeta^2$ with a non-negative smooth cut-off function $\zeta = \zeta(x)$. Then
\begin{align*}
\lefteqn{
\left( \partial_t \left[u(t+h)-u(t)\right], \partial_t u(t+h) \right)_{L^2_\zeta} + \left( \eta(t+h)-\eta(t), \partial_t u(t+h)\right)_{L^2_\zeta}
}\\
&\quad + \int_\R \partial_x \left[u(t+h)-u(t)\right] \partial_x \partial_t u(t+h) \zeta^2 \, \d x \\
&\quad + \left( \beta(u(t+h))-\beta(u(t)), \partial_t u(t+h) \right)_{L^2_\zeta}\\
&= \lam \left( u(t+h)-u(t) , \partial_t u(t+h) \right)_{L^2_\zeta}
- 2\int_\R \partial_x \left[u(t+h)-u(t)\right] \partial_t u(t+h) \zeta\zeta' \, \d x, 
\end{align*}
 where $(v,w)_{L^2_\zeta} := \int_\R v(x)w(x) \zeta(x)^2 \, \d x$ for $v,w \in L^2_{\rm loc}(\R)$.  Due to the relation $\partial_t u \in \partial \iI^*(\eta)$ (equivalently, $\eta \in \partial \iI(\partial_t u)$) and $\iI^*(\eta) = (\eta,\partial_t u)_{L^2}-\iI(\partial_tu)=0$, one has $\left( \eta(t+h)-\eta(t), \partial_t u(t+h)\right)_{L^2_\zeta} \geq 0$. Integrate both sides over $(s,t)$ for any $0 < s < t < T$. Then it follows that
\begin{align*}
\lefteqn{
\int^t_s \left( \partial_t \left[u(\tau+h)-u(\tau)\right], \partial_t u(\tau+h) \right)_{L^2_\zeta} \, \d \tau
}\\
&\quad  + \int^t_s \int_\R \partial_x \left[u(\tau+h)-u(\tau)\right] \partial_x \partial_t u(\tau+h) \zeta^2 \, \d x \d \tau\\
&\quad + \int^t_s \left( \beta(u(\tau+h))-\beta(u(\tau)), \partial_t u(\tau+h) \right)_{L^2_\zeta} \, \d \tau\\
&\leq \lam \int^t_s \left( u(\tau+h)-u(\tau), \partial_t u(\tau+h) \right)_{L^2_\zeta} \, \d \tau\\
&\quad - 2\int^t_s \int_\R \partial_x \left[u(\tau+h)-u(\tau)\right] \partial_t u(\tau+h) \zeta\zeta' \, \d x \d \tau.
\end{align*}
Furthermore, dividing both sides by $h>0$ and passing to the limit as $h \to +0$, one can obtain
\begin{align*}
\int^t_s \left( \partial_t^2 u , \partial_t u \right)_{L^2_\zeta} \, \d \tau + \int^t_s \|\partial_x \partial_t u\|^2_{L^2_\zeta} \, \d \tau
+ \int^t_s \int_\R \beta'(u) |\partial_t u|^2 \zeta^2 \, \d x \d \tau
\quad
\\
\leq \lam \int^t_s \|\partial_t u\|^2_{L^2_\zeta} \, \d \tau
- 2\int^t_s \int_\R \partial_x \partial_t u \partial_t u \zeta\zeta' \, \d x \d \tau, 
\nonumber
\end{align*}
 where $\|w\|_{L^2_\zeta} := (\int_\R |w(x)|^2\zeta(x)^2 \, \d x)^{1/2}$ for $w \in L^2_{\rm loc}(\R)$,  from the fact by Theorem \ref{T:C1} for $u_0 \in C^\infty_c(\R)$ that
\begin{alignat*}{3}
\frac{u(\cdot+h)-u(\cdot)}{h} &\to \partial_t u \quad &&\mbox{ strongly in } L^2(s,t;H^1(\R)),\\
\frac{\partial_t u(\cdot+h)-\partial_t u(\cdot)}{h} &\to \partial_t^2 u \quad &&\mbox{ strongly in } L^2(s,t;L^2(\R)),\\
\frac{\beta(u(\tau+h))-\beta(u(\tau))}{h} &\to \beta'(u)\partial_t u \quad &&\mbox{ strongly in } L^2(s,t;L^2(\R)),\\
\partial_t u(\cdot+h)&\to \partial_t u \quad &&\mbox{ strongly in } L^2(s,t;H^1(\R)).
\end{alignat*}
With the aid of Lebesgue's differentiation theorem, one can obtain the following:
\begin{align}\label{we-1}
\dfrac 1 2 \dfrac \d {\d t} \|\partial_t u\|_{L^2_\zeta}^2 + \|\partial_x \partial_t u\|^2_{L^2_\zeta} + \int_\R \beta'(u) |\partial_t u|^2 \zeta^2 \, \d x \quad
\\
\leq \lam \|\partial_t u\|^2_{L^2_\zeta} - 2 \int_\R \partial_x \partial_t u \partial_t u \zeta\zeta' \, \d x
\nonumber
\end{align}
for a.e.~$t \in (0,T)$.

We next test \eqref{subt-eq} by $\partial_t [u(t+h)-u(t)]\zeta^2$. Then
\begin{align*}
\lefteqn{
\left\| \partial_t \left[u(t+h)-u(t)\right] \right\|_{L^2_\zeta}^2 + \left( \eta(t+h)-\eta(t), \partial_t [u(t+h)-u(t)] \right)_{L^2_\zeta}
}\\
&\quad + \frac 1 2 \frac{\d}{\d t} \left\| \partial_x \left[u(t+h)-u(t)\right] \right\|_{L^2_\zeta}^2 + \left( \beta(u(t+h))-\beta(u(t)), \partial_t [u(t+h)-u(t)] \right)_{L^2_\zeta}\\
&= \dfrac{\lam} 2 \dfrac \d {\d t} \left\| u(t+h)-u(t) \right\|_{L^2_\zeta}^2 \\
&\quad - 2\int_\R \partial_x \left[u(t+h)-u(t)\right] \partial_t \left[u(t+h)-u(t)\right] \zeta\zeta' \, \d x. 
\end{align*}
By the monotonicity of $\partial \iI$, the second term of the left-hand side is non-negative. Integrating both sides over $(s,t)$ for $0 < s < t < T$, we have
\begin{align*}
\lefteqn{
\int^t_s \left\| \partial_t \left[u(\tau+h)-u(\tau)\right] \right\|_{L^2_\zeta}^2 \, \d \tau + \frac 1 2 \left\| \partial_x \left[u(t+h)-u(t)\right] \right\|_{L^2_\zeta}^2
}\\
&\quad - \frac{\lam}2 \left\| u(t+h)-u(t) \right\|_{L^2_\zeta}^2 \\
&\quad + \int^t_s \left( \beta(u(\tau+h))-\beta(u(\tau)), \partial_t [u(\tau+h)-u(\tau)] \right)_{L^2_\zeta} \, \d \tau\\
& \leq \frac 1 2 \left\| \partial_x \left[u(s+h)-u(s)\right] \right\|_{L^2_\zeta}^2 - \dfrac{\lam}2 \left\| u(s+h)-u(s) \right\|_{L^2_\zeta}^2 \\
&\quad - 2 \int^t_s \int_\R \partial_x \left[u(\tau+h)-u(\tau)\right] \partial_t \left[u(\tau+h)-u(\tau)\right] \zeta \zeta' \, \d x \d \tau. 
\end{align*}
Divide both sides by $h>0$ and pass to the limit as $h \to +0$. It then follows by Theorem \ref{T:C1} for $u_0 \in C^\infty_c(\R)$ that
\begin{align*}
\lefteqn{
\int^t_s \left\| \partial_t^2 u \right\|_{L^2_\zeta}^2 \, \d \tau + \frac 1 2 \left\| \partial_x \partial_t u \right(t)\|_{L^2_\zeta}^2 
- \frac{\lam}2 \left\| \partial_t u(t) \right\|_{L^2_\zeta}^2
+ \int^t_s \int_\R \beta'(u) \partial_t u \partial_t^2 u \zeta^2 \, \d x \d \tau
}\\
&\quad \leq \frac 1 2 \left\| \partial_x \partial_t u(s) \right\|_{L^2_\zeta}^2 - \dfrac{\lam}2 \left\| \partial_t u(s) \right\|_{L^2_\zeta}^2 - 2\int^t_s \int_\R \partial_x \partial_t u(\tau) \partial_t^2 u(\tau) \zeta\zeta' \, \d x \d \tau 
\end{align*}
for any $0 < s < t < T$. Therefore, thanks to Lebesgue's differentiation theorem, we obtain
\begin{align}\label{we-2}
\left\| \partial_t^2 u \right\|_{L^2_\zeta}^2 + \frac 1 2 \dfrac \d {\d t} \left\| \partial_x \partial_t u \right(t)\|_{L^2_\zeta}^2 
+ \int_\R \beta'(u) \partial_t u \partial_t^2 u \zeta^2 \, \d x
\quad \\
\leq \frac{\lam}2 \dfrac \d {\d t} \left\| \partial_t u(t) \right\|_{L^2_\zeta}^2 -2 \int_\R \partial_x \partial_t u \partial_t^2 u \zeta\zeta' \, \d x\nonumber
\end{align}
a.e.~in $(0,T)$.

We can then obtain
\begin{lemma}
Assume that $u_0 \in H^2_{\rm loc}(\R) \cap L^\infty(\R)$. Then the unique $L^2_{\rm loc}$ solution of \eqref{irAC} as well as \eqref{pde-obs0} belongs to $C^{1,\gamma}(\R \times \R_+)$ for $\gamma \in [0,1/2)$.
\end{lemma}

\begin{proof}
Let $(u_{0,n})$ be a sequence in $C^\infty_c(\R)$ such that $u_{0,n} \to u_0$ strongly in $H^2(I)$ for any bounded interval $I \subset \R$ and \eqref{u0n-snd} holds (see \S \ref{Ss:loc-ex}). Denote by $u_n$ the unique $L^2$ solution to \eqref{irAC} with $u_0$ replaced by $u_{0,n}$. According to the proof of Theorem \ref{T:loc-L2-ex}, $u_n$ converges to an $L^2_{\rm loc}$ solution to \eqref{irAC} (as well as \eqref{pde-obs0}) on $[0,T]$ (see also a priori estimates for $u_n$ established in \S \ref{Ss:loc-ex}). Moreover, \eqref{we-1} and \eqref{we-2} are satisfied for $u = u_n$. Set $\zeta = \zeta_R$ defined as in the proof of Lemma \ref{L:cp-obs+}. Multiply both sides of \eqref{we-1} by $t$. Then
\begin{align*}
\lefteqn{
 \dfrac 1 2 \dfrac \d {\d t} \left( t\|\partial_t u_n\|_{L^2_\zeta}^2 \right) + t \|\partial_x \partial_t u_n\|^2_{L^2_\zeta} + t \int_\R \beta'(u_n) |\partial_t u_n|^2 \zeta_R^2 \, \d x \quad
}\\
&\leq \lam t \|\partial_t u_n\|^2_{L^2_\zeta} + \frac4R t\|\partial_x \partial_t u_n\|_{L^2_\zeta} \|\partial_t u_n\|_{L^2(I_{2R})} + \frac 1 2 \|\partial_t u_n\|_{L^2_\zeta}^2.
\end{align*}
Integrate both sides over $(0,t)$. We find that
\begin{align*}
\lefteqn{
 \dfrac t 2 \|\partial_t u_n(t)\|_{L^2_\zeta}^2 + \int^t_0 \tau\|\partial_x \partial_t u_n\|^2_{L^2_\zeta} \, \d \tau 
}\\
&\leq \lam \int^t_0 \tau \|\partial_t u_n\|^2_{L^2_\zeta} \, \d \tau 
+ \frac 1 2 \int^t_0 \|\partial_t u_n\|_{L^2_\zeta}^2 \, \d \tau\\
&\quad + \frac4R \int^t_0 \tau \|\partial_x \partial_t u_n\|_{L^2_\zeta} \|\partial_t u_n\|_{L^2(I_{2R})} \, \d \tau,
\end{align*}
which  along with \eqref{AA}  implies
$$
\sup_{t \in [0,T]} t \|\partial_t u_n(t)\|_{L^2(I_R)}^2 + \int^T_0 \tau \|\partial_x \partial_t u_n\|^2_{L^2(I_R)} \, \d \tau \leq C_{T,R}.
$$
Here and henceforth, we write $I_R := (-R,R)$ for any $R > 0$.  For any $t_0 \in (0,T)$, one observes that
\begin{equation}\label{BB}
\sup_{t \in [t_0,T]}\|\partial_t u_n(t)\|_{L^2(I_R)}^2 + \int^T_{t_0}\|\partial_x \partial_t u_n(\tau)\|^2_{L^2(I_R)} \, \d \tau \leq \frac{C_{T,R}}{t_0}.
\end{equation}

We next recall \eqref{we-2} for $u = u_n$, multiply it by $(t-t_0)$ and integrate it over $(t_0,t)$. It then follows that
\begin{align*}
\lefteqn{
\frac 1 2 \int^t_{t_0} (\tau-t_0) \left\| \partial_t^2 u_n \right\|_{L^2_\zeta}^2 \, \d \tau + \frac 1 2 (t-t_0) \left\| \partial_x \partial_t u_n \right(t)\|_{L^2_\zeta}^2 
}\\
&\leq \frac{\lam}2 (t-t_0)\left\| \partial_t u_n(t) \right\|_{L^2_\zeta}^2
+ \frac 1 2 \int^t_{t_0} \left\| \partial_x \partial_t u_n \right\|_{L^2_\zeta}^2 \, \d \tau\\
&\quad + \frac 1 2 \|\beta'(u_n)\|_{L^\infty}^2 \int^t_{t_0} (\tau-t_0)\|\partial_t u_n\|_{L^2_\zeta}^2 \, \d \tau\\ 
&\quad +\dfrac4R\int^t_{t_0} (\tau-t_0) \|\partial_x \partial_t u_n\|_{L^2(I_{2R})} \|\partial_t^2 u_n\|_{L^2_\zeta} \, \d \tau,
\end{align*}
which  along with \eqref{AA} and \eqref{BB}  yields
\begin{equation}\label{CC}
\int^T_{t_0} (\tau-t_0) \|\partial_t^2 u_n\|^2_{L^2(I_R)} \, \d \tau + \sup_{t \in [0,T]} (t-t_0) \|\partial_x \partial_t u_n(t)\|_{L^2(I_R)}^2 \leq \frac{C_{T,R}}{t_0}.
\end{equation}
Combining all these facts and repeating the argument for compactly supported data, one can derive 
$$
\partial_t u \in W^{1,2}(2t_0,T;L^2(I_R)) \cap L^\infty(2t_0,T;H^1(I_R)) \subset C^\gamma(\overline{I_R}\times[2t_0,T])
$$
for $\gamma \in [0,1/2)$. Moreover, by  \eqref{CC} and  the equation $\partial_x^2 u = \partial_t u + \eta + f(u)$ along with \eqref{eta-est}, we can also verify that
$$
\partial_x u \in W^{1,2}(2t_0,T;L^2(I_R)) \cap L^\infty(2t_0,T;H^1(I_R)) \subset C^\gamma(\overline{I_R}\times[2t_0,T])
$$
for $\gamma \in [0,1/2)$. From the arbitrariness of $t_0$, $T$ and $R$, we complete the proof. 
\end{proof}

\subsection{Corollaries of $C^1$ regularity}\label{Ss:appl-c1}

This subsection concerns a couple of corollaries of Theorem \ref{T:C1}, which may be used in later sections.

\subsubsection{Classical regularity on the non-coincidence set}\label{Sss:CR}
Set
$$
Q_T^+ := \{(x,t) \in Q_T \colon u(x,t) > u_0(x) \},
$$
which is the complement of the coincidence set. Then $\eta \equiv 0$ on $Q_T^+$. Hence by Theorem \ref{T:C1},
\begin{equation}\label{EQonQ+}
\partial_x^2 u = \partial_t u - f(u) \in C(\R \times \R_+) \ \mbox{ in } Q^+_T,
\end{equation}
which ensures $u \in C^{2,1}(Q^+_T)$. Therefore, $u=u(x,t)$ solves \eqref{cac} over $Q_T^+$ in the classical sense.

\subsubsection{Behavior of solutions near the free boundary}\label{Sss:BFB}

In the rest of this section, we address ourselves to the behavior of solutions to \eqref{irAC} in the third phase (see \S \ref{Ss:ph3}), i.e., $u = u(x,t)$ also solves \eqref{pde-alpha} and $u(x,t) > \alpha$ if and only if $x > r(t)$. For simplicity, we set $t_2 = 0$ by translation. Thanks to the $C^1$-regularity of $u(x,t)$ in space, we readily observe that
$$
\partial_x u(r(t),t) = 0 \quad \mbox{ for all } \ t \geq 0,
$$
 where $r(t)$ is the free boundary defined by \eqref{r(t)}. Furthermore, it also follows that
\begin{equation}\label{ut=0}
\partial_t u(r(t),t) = 0 \quad \mbox{ for all } \ t \geq 0.
\end{equation}
Indeed, for any $h < 0$, we find that
$$
\dfrac{u(r(t),t+h)-u(r(t),t)}h = 0
$$
by $u(r(t),t+h)=u(r(t),t) = \alpha$. Hence letting $h \to -0$ and using Theorem \ref{T:C1}, one obtains $\partial_t u(r(t),t) = 0$. We further  find by \eqref{EQonQ+}  that
$$
\partial_x^2 u(r(t)+0,t) = f(\alpha) \quad \mbox{ for all } \ t > 0.
$$
Then we also deduce that
\begin{alignat}{4}
\partial_t u = \partial_x^2 u - f(u), \quad r(t) < x < +\infty, \quad t > 0,\label{pde-u}\\
 u(r(t),t) = \alpha, \quad \partial_x u (r(t),t) = 0, \quad t > 0,\label{bc-u-1}\\
\partial_x^2 u(r(t)+0,t) = f(\alpha), \quad t > 0,\label{bc-u-2}
\end{alignat}
 whence follows that $u$ loses $C^2$ regularity only on the free boundary point $r(t)$ similarly to traveling waves $\phi_\alpha$.

We finally remark that even continuity of $r(t)$ has not yet been ensured.

\subsubsection{Motion equation of the free boundary for regular solutions}\label{Sss:moeq}

 Let us finally derive a motion equation of the free boundary $r(t)$ for \emph{regular} solutions. For each $t \geq 0$, we have obtained
\begin{equation}\label{me:pde}
\partial_x^2 u(\cdot,t) = \partial_t u(\cdot,t) + f(u(\cdot,t)) \ \mbox{ in } \ (r(t),+\infty)
\end{equation}
and the right-hand side is continuous in $\R \times \R_+$. Hence $\partial_x^2 u$ can be regarded as a continuous function on $[r(t),+\infty) \times [0,+\infty)$. In what follows, we further assume that 
\begin{equation}\label{regu-hyp}
r(\cdot) \in C^1([0,+\infty)) \ \mbox{ and } \ u(\cdot, t) \in C^3([r(t),+\infty)) \ \mbox{ for } t > 0,
\end{equation}
which may not be at all trivial and can be however checked for traveling wave solutions. Moreover, \eqref{me:pde} and \eqref{regu-hyp} ensure that $\partial_t u(\cdot,t) \in C^1([r(t),+\infty))$ for $t > 0$. Set
$$
y := x - r(t), \quad v(y,t) := u(x,t) - \uo.
$$
Then we have
$$
\partial_t u(x,t) = \partial_t v(y,t) - \dot r(t) \partial_y v(y,t)
\ \mbox{ for all } \ t > 0,
$$
where $\dot{r} := \d r/\d t$. Thus \eqref{pde-u}--\eqref{bc-u-2} are rewritten as
\begin{alignat}{4}
 \partial_t v = \partial^2_y v - f(v+\uo) + \dot r \partial_y v , \quad & y \in \R_+, \ t > 0,\label{pde-v}\\
 v(0,t) = 0, \quad \partial_y v(0,t) = 0, \quad \partial_y^2 v(0,t) = f(\alpha), \quad & t > 0,\label{bc-v}\\
 v(y,0) = v_0(y) := u(r_0+y,0) - \uo, \quad & y \in \R_+,\label{ic-v}
\end{alignat}
where $r_0 := r(0)$. Differentiating both sides of \eqref{pde-v} in $y$ and setting $z(y,t) := \partial_y v(y,t)$, we get
\begin{alignat}{4}
 \partial_t z = \partial_y^2 z - f'(v+\uo) z + \dot r \partial_y z, \quad &y \in \R_+, \ t > 0,\label{pde-z}\\
 z(0,t) = 0, \quad \partial_y z(0,t) = f(\alpha), \quad &t > 0,\label{bc-z}\\
 z(y,0) = v_0'(y), \quad &y \in \R_+.\label{ic-z}
\end{alignat}
Moreover, \eqref{pde-z} holds at $y = 0$; indeed, by assumption, $\partial_t z (\cdot,t)= \partial_t \partial_x u(\cdot+r(t),t) +\dot{r}(t) \partial^2_y v (\cdot,t)$ and $\partial_y^2 z(\cdot,t) = \partial_x^3 u(\cdot+r(t),t)$ are continuous on $[0,+\infty)$ (in particular, at the origin) for $t > 0$. Since $z(0,t) = 0$ for $t > 0$, we observe that 
$$
\partial_t z(+0,t) = 0.
$$
Hence substituting \eqref{bc-z} to \eqref{pde-z} at $y = +0$ and recalling $\partial^2_y z(+0,t) = \partial^3_x u(r(t)+0,t)$, one can derive a motion equation of the free boundary $r(t)$ as follows (cf.~Remark \ref{R:StAC}):
$$
\dfrac{\d r}{\d t}(t) = - \dfrac{\partial^3_x u(r(t),t)}{f(\alpha)} \quad \mbox{ for } \ t > 0.
$$

\section{Quasi-convergence to traveling waves}\label{S:Q-conv}

In this section, we shall prove Lemma \ref{L:quasi-conv}. To this end, suppose that the solution $u$ of \eqref{irAC} solves
\begin{align}\label{obac}
\partial_t u - \partial_x^2 u + f(u) + \partial \iIIal(u) \ni 0 \quad \mbox{ in } \R \times (0,+\infty)
\end{align}
and satisfies
\begin{gather*}
u(\cdot,0) \in H^2_{\rm loc}(\R) \cap L^\infty(\R), \quad 
u(x,0) = \alpha \ \mbox{ for } \ x \in (-\infty,r_0], 
\\
\alpha \leq u(x,t) \leq \phi_\alpha(x-c_\alpha t)
\quad \mbox{ for } \ x \in \R, \quad t \geq 0,\\
u(x,t) \geq a_+ - \tfrac{\delta_0}2 \quad \mbox{ for } \ x \in [r_1, + \infty), \quad t \geq 0,
\end{gather*}
 where $\delta_0$ is the constant appeared in Lemma \ref{L:comfunc},  for some $r_0, r_1 \in \R$ satisfying $r_0 < r_1$. Indeed, the setting above does not lose any generality by virtue of the first and second phases discussed in \S \ref{S:OL} and by suitable translation in space and time. It then follows that $u(\cdot,t) \equiv \alpha$ on $(-\infty,c_\alpha t]$ for any $t \geq 0$.

Set
$$
v(y,t) := u(y + c_\alpha t,t) - a_+ \quad \mbox{ for } \ y, t \in [0,+\infty).
$$
Then noting that $u(x,t) = v(x-c_\alpha t,t) + a_+$ and $\partial_t u = \partial_t v - c_\alpha \partial_y v$, we find that $v$ satisfies
\begin{align}
 \partial_t v - \partial_y^2 v + f(v+a_+) + \partial \iIIal(v+a_+) \ni c_\alpha \partial_y v \ & \mbox{ in } \R_+ \times \R_+,\label{v-obs}\\
 v(0,t) = \alpha-a_+, \quad \partial_y v(0,t) = 0 \ &\mbox{ for } t \geq 0,\label{v-obs-bc}\\
 v(y,0) = v_0(y) := u(y,0) - a_+ \ &\mbox{ for } y \geq 0,\label{v-obs-ic}\\
\alpha - a_+\leq v(y,t) \leq \phi_\alpha(y)-a_+\ &\mbox{ for } y \geq 0.\label{v-obs-snd}
\end{align}
Indeed, we used the fact that $\partial_y v(0,t) = \partial_x u(c_\alpha t,t)=0$, since $u$ is of class $C^1$ in $\R \times \R_+$ (see Theorem \ref{T:C1}) and $u(\cdot,t) = \alpha$ in $(-\infty, c_\alpha t]$. Hence $v(y,t) \in [\alpha-a_+,0]$. In what follows, we shall rewrite \eqref{v-obs} into a gradient flow with an exponential weight, and then, $v$ is obviously (square-)integrable over $\R_+$ with the weight and the integral is uniformly bounded for $t \geq 0$. Furthermore, a weighted energy turns out to be bounded in time.

\subsection{Weighted gradient flow}

Equation \eqref{v-obs} can be written as
\begin{equation}\label{wgf}
\e^{c_\alpha y}\partial_t v - \partial_y \left( \e^{c_\alpha y} \partial_y v \right) + \e^{c_\alpha y} f(v+a_+) + \partial \iIIal(v+a_+) \ni 0.
\end{equation}
Here we used the fact that  the set  $\e^{c_\alpha y}\partial \iIIal(v+a_+)$  coincides with  $\partial \iIIal(v+a_+)$. Let us first multiply both sides by $v \zeta_R^2$  with a smooth cut-off function  $\zeta_R \in C^1_c([0,+\infty))$ satisfying 
$$
\zeta_R \equiv 1 \ \mbox{ on } [0,R], \quad \zeta_R \equiv 0 \ \mbox{ on } [2R,+\infty), \quad \|\zeta_R'\|_{L^\infty(\R_+)} \leq \dfrac2R
$$
for $R>0$ and integrate it over $\R_+$. It then follows that
\begin{align*}
\lefteqn{
 \frac 1 2 \frac \d {\d t} \int^{+\infty}_0 \e^{c_\alpha y} v^2 \zeta_R^2 \, \d y
 - \underbrace{\left[\e^{c_\alpha y} \partial_y v \zeta_R^2 v \right]^{+\infty}_0}_{=0}
 + \int^{+\infty}_0 \e^{c_\alpha y} |\partial_y v|^2 \zeta_R^2 \, \d y
}\\
 &&+ \int^{+\infty}_0 \e^{c_\alpha y} f(v+a_+)v \zeta_R^2 \, \d y
 + \int^{+\infty}_0 \e^{c_\alpha y} \eta v \zeta_R^2 \, \d y
\\
&&\quad = -  2  \int^{+\infty}_0 \e^{c_\alpha y} \partial_y v  \zeta_R\zeta_R'  v\, \d y,
\end{align*}
where $\eta$ is a section of $\partial \iIIal(v+a_+)$. Here we note  by \eqref{f-1} and \eqref{f-2}  that
$$
f(v+a_+)v = f(a_+)v + f'(\theta_v v + a_+)v^2 \geq -\lam v^2 \quad \mbox{ for any } \ v \in [\alpha-a_+,0]
$$
for some $\theta_v \in (0,1)$ and
$$
v \leq 0 \quad \mbox{ and } \quad \eta \leq 0 \quad \mbox{ for all } \ \eta \in \partial \iIIal(v+a_+).
$$
Hence
\begin{align*}
 \frac 1 2 \frac \d {\d t} \int^{+\infty}_0 \e^{c_\alpha y} v^2  \zeta_R^2  \, \d y
 + \int^{+\infty}_0 \e^{c_\alpha y} |\partial_y v|^2  \zeta_R^2  \, \d y
 -\lam \int^{+\infty}_0 \e^{c_\alpha y} v^2  \zeta_R^2  \, \d y
\\
\leq -  2  \int^{+\infty}_0 \e^{c_\alpha y} \partial_y v  \zeta_R \zeta_R'  v \, \d y,
\end{align*}
which also implies
\begin{align*}
 \frac 1 2 \frac \d {\d t} \left(
 \e^{-2\lam t} \int^{+\infty}_0 \e^{c_\alpha y} v^2 \zeta_R^2 \, \d y
\right)
 + \e^{-2\lam t} \int^{+\infty}_0 \e^{c_\alpha y} |\partial_y v|^2 \zeta_R^2 \, \d y
\\
\leq - 2 \e^{-2\lam t} \int^{+\infty}_0 \e^{c_\alpha y} \partial_y v  \zeta_R\zeta_R' v \, \d y.
\end{align*}
Integrating both sides over $(0,t)$, one has
\begin{align*}
\lefteqn{
\frac 1 2 \e^{-2\lam t} \int^{+\infty}_0 \e^{c_\alpha y} v(\cdot,t)^2 \zeta_R^2 \, \d y
 + \int^t_0 \e^{-2\lam\tau} \left( \int^{+\infty}_0 \e^{c_\alpha y} |\partial_y v|^2 \zeta_R^2 \, \d y \right) \d \tau
}\\
&\leq \frac 1 2 \int^{+\infty}_0 \e^{c_\alpha y} v(\cdot,0)^2 \zeta_R^2 \, \d y
- 2\int^t_0 \e^{-2\lam\tau} \left( \int^{+\infty}_0 \e^{c_\alpha y} \partial_y v \zeta_R\zeta_R' v \, \d y \right) \d \tau\\
&\leq \frac{1}{2|c_\alpha|} \|v(\cdot,0)\|_{L^\infty(\R_+)}^2 +  \frac12 \int^t_0 \e^{-2\lam\tau} \left( \int_0^{+\infty} \e^{c_\alpha y} |\partial_y v|^2  \zeta_R^2  \, \d y \right) \d \tau\\
&\quad + \frac{4}{\lam|c_\alpha|R^2} \|v\|_{L^\infty(Q_T)}^2.
\end{align*}
Therefore it follows that
\begin{align*}
\frac 1 2 \e^{-2\lam t} \int^{+\infty}_0 \e^{c_\alpha y} v(\cdot,t)^2 \zeta_R^2 \, \d y + \frac12\int^t_0 \e^{-2\lam\tau} \left( \int_0^{+\infty} \e^{c_\alpha y} |\partial_y v|^2 \zeta_R^2\, \d y \right) \d \tau
\\
 \leq \frac{1}{2|c_\alpha|} \|v(\cdot,0)\|_{L^\infty(\R_+)}^2
+ \frac{4}{\lam|c_\alpha|R^2} \|v\|_{L^\infty(Q_T)}^2.
\end{align*}
Moreover, passing to the limit as $R \to +\infty$, we obtain 
\begin{equation}\label{we:vx}
\int^t_0 \e^{-2\lam\tau} \left( \int^{+\infty}_0 \e^{c_\alpha y} |\partial_y v|^2 \, \d y \right) \d \tau
\leq C \|v(\cdot,0)\|_{L^\infty(\R_+)}^2,
\end{equation}
whence follows that there exists $\tau_0 \in (0,1)$ such that
\begin{equation}\label{ainotane}
\int^{+\infty}_0 \e^{c_\alpha y} |\partial_y v(y,\tau_0)|^2 \, \d y \leq C \e^{2 \lam \tau_0} \|v(\cdot,0)\|_{L^\infty(\R_+)}^2.
\end{equation}

Test \eqref{wgf} by $\partial_t v$  (see Remark \ref{R:v-rigo} below).  We then observe that
\begin{align}\label{wgf:xdv}
 \int^{+\infty}_0 \e^{c_\alpha y} |\partial_t v|^2 \, \d y + \dfrac \d {\d t} \left( \dfrac 1 2 \int^{+\infty}_0 \e^{c_\alpha y} |\partial_y v|^2 \, \d y + \int^{+\infty}_0 \e^{c_\alpha y} h(v) \, \d y \right) = 0.
\end{align}
where $h(v)$ is a non-negative primitive function of $v \mapsto f(v+a_+)$ given by
$$
h(v) := \hat f(v+a_+) \geq 0
$$
(see \eqref{f-2}). Here we also used the fact that
$$
\int^{+\infty}_0 \eta \partial_t v \, \d y = \dfrac{\d}{\d t} \iIIal(v+a_+) \equiv 0
\quad \mbox{ for any } \ \eta \in \partial \iIIal(v+a_+).
$$
Therefore integrating both sides of \eqref{wgf:xdv} over $(\tau_0,t)$, we obtain
\begin{equation}\label{v-est-1}
 \int^t_{\tau_0} \int^{+\infty}_0 \e^{c_\alpha y} |\partial_t v(y,\tau)|^2 \, \d y \d \tau
+ E(v(\cdot,t)) \leq E(v(\cdot,\tau_0))
\ \mbox{ for all } \ t \geq \tau_0,
\end{equation}
where $E$ is given by
$$
E(w) := \dfrac 1 2 \int^{+\infty}_0 \e^{c_\alpha y} |\partial_y w(y)|^2 \, \d y + \int^{+\infty}_0 \e^{c_\alpha y} h(w(y)) \, \d y \geq 0.
$$
Here we also note by \eqref{ainotane} that
$$
\dfrac 1 2 \int^{+\infty}_0 \e^{c_\alpha y} |\partial_y v(y,\tau_0)|^2 \, \d y < +\infty,
$$
and moreover,
$$
\int^{+\infty}_0 \e^{c_\alpha y} h(v(y,\tau_0)) \, \d y
\leq \dfrac 1 {|c_\alpha|}  \sup_{s \in [\alpha-a_+,0]} h(s) < +\infty.
$$
Consequently, we conclude that
$$
\int^{+\infty}_{\tau_0}\int^{+\infty}_0 \e^{c_\alpha y} |\partial_t v(y,t)|^2 \, \d y \d t + \sup_{t \geq \tau_0} \left( \int^{+\infty}_0 \e^{c_\alpha y} |\partial_y v(y,t)|^2 \, \d y \right) < +\infty.
$$

\begin{remark}\label{R:v-rigo}
{\rm
The second estimate was derived by a formal argument. To be precise, one should test \eqref{wgf} by $\partial_t v \zeta_R^2$, where $\zeta_R$ is the cut-off function defined at the beginning of this subsection, and repeat the same argument as in the first estimate, i.e., integration by parts and passage to the limit as $R \to +\infty$, with the aid of the fact $\eta\zeta_R^2 \in \partial \iIIal(v+a_+)$ and estimate \eqref{we:vx}.
}
\end{remark}

\subsection{Quasi-convergence}

Therefore we can take a sequence $t_n \to +\infty$ such that
\begin{alignat*}{2}
\partial_t v(\cdot,t_n) \to 0 \quad \mbox{ strongly in } L^2(\R_+;\e^{c_\alpha y}\,\d y),
\end{alignat*}
which also implies that
$$
\partial_t v(\cdot,t_n) \to 0 \quad \mbox{ strongly in } L^2(0,R)
$$
for any $R > 0$. Moreover, we deduce by \eqref{eta-est} that
$$
\eta(\cdot,t_n) \to \eta_\infty \quad \mbox{ weakly in } L^2(0,R)
$$
for some $\eta_\infty \in L^2_{\rm loc}(\R)$. Furthermore, there exists $\psi \in H^2_{\rm loc}(\R_+) \subset C^1(\R_+)$ such that, for any $R > 0$, up to a (not relabeled) subsequence,
\begin{alignat}{2}
v(\cdot,t_n) &\to \psi \quad &&\mbox{ weakly in } H^2(0,R),\nonumber\\
& &&\mbox{ strongly in } C^1([0,R]),\label{quasi-conv-0}\\
& &&\mbox{ weakly star in } L^\infty(\R),\nonumber
\end{alignat}
which yields $\psi \geq \alpha - a_+$ in $\R_+$, $\psi(0) = \alpha - a_+$ and $\psi'(0) = 0$. Applying a diagonal argument, one can take a (not relabeled) subsequence of $(n)$ such that $v(\cdot,t_n)$ converges to $\psi$ pointwisely on $\R$ and uniformly on each bounded interval. By the demiclosedness of maximal monotone operators, one can also verify that $\eta_\infty \in \partial \iIIal(\psi+a_+)$ a.e.~in $\R_+$. Therefore the limit $\psi$ solves
$$
- \psi'' + f(\psi + a_+) + \partial \iIIal(\psi+a_+) \ni c_\alpha \psi' \ \mbox{ in } \R_+.
$$
Set $\phi(x) := \psi(x) + a_+$ for $x \geq 0$ and $\phi(x) \equiv \alpha$ for $x < 0$. Then it follows that $\phi \geq \alpha$ in $\R$, $\phi(0) = \alpha$, $\phi'(0) = 0$, $\phi \in H^2_{\rm loc}(\R) \subset C^1(\R)$ and
$$
- \phi'' + f(\phi) + \partial \iIIal(\phi) \ni c_\alpha \phi' \ \mbox{ in } \R.
$$

We next claim that there exists $h_1 \geq 0$ such that
\begin{equation}\label{cl1-qc}
\phi(x) = \phi_\alpha(x-h_1) \ \mbox{ for all } x \in \R.
\end{equation}
Indeed, recall that $u \geq \alpha$ in $\R \times \R_+$ and $u(x,t) \geq a_+ - \delta_0/2$ for $x \geq r_1$ and $t \geq 0$. Hence thanks to Lemma \ref{L:comfunc} one can verify that
$$
v(y,t) = u(y + c_\alpha t, t) - a_+
\geq w^-(y + c_\alpha t,t) - a_+
\ \mbox{ for all } y \in \R_+, \ t \geq 0,
$$
where $w^-(x,t) := \phi_\alpha(x - c_\alpha t - x_2 - \sigma \delta (1-\e^{-\beta t}) ) - \delta \e^{-\beta t}$ with some $x_2 \in \R$ and positive constants $\beta$, $\delta$, $\sigma$ satisfying all the requirements of Lemma \ref{L:comfunc}. It turns out that $\psi \not\equiv \alpha - a_+$ (i.e., $\phi\not\equiv \alpha$) by a passage to the limit as $t = t_n \to +\infty$. Therefore one can take $h_1 \geq 0$ such that $\phi \equiv \alpha$ in $(-\infty,h_1]$ and $\phi > \alpha$ in $(h_1, h_2)$ for some $h_2 \in (h_1,+\infty]$. Since $\phi$ is of class $C^1$ over $\R$, we find that $\phi'(h_1) = 0$. Hence, it follows that
$$
- \phi'' + f(\phi) = c_\alpha \phi' \ \mbox{ in } (h_1,h_2), \quad \phi(h_1) = \alpha, \quad \phi'(h_1) = 0.
$$
By virtue of a standard uniqueness theorem for ODEs, we conclude that $\phi$ coincides with $\phi_\alpha(\cdot-h_1)$ on $[h_1,h_2)$. Since $\phi_\alpha$ is strictly increasing in $\R_+$, one observes that $h_2 = +\infty$. Thus the assertion \eqref{cl1-qc} follows.

Recalling $w^-$ defined above, we can deduce that
$$
0 \leq a_+ - u(x,t) \leq a_+ - w^-(x,t) \ \mbox{ for all } \ x \in \R.
$$
For any $\vep > 0$, one can take $t_\vep > 0$ such that $\delta \e^{-\beta t_\vep} < \vep$. Moreover, by the definition of $\phi_\alpha$, there exists $L_\vep > 0$ such that
$$
a_+ - \phi_\alpha(\xi) < \vep \ \mbox{ for all } \ \xi \geq L_\vep.
$$
Hence, it follows that
\begin{equation}\label{far-control}
0 \leq a_+ - u(x,t) \leq a_+ - w^-(x,t) < 2\vep, 
\end{equation}
provided that $t \geq t_\vep$ and $x - c_\alpha t \geq x_2 + \sigma \delta + L_\vep$. We derive from \eqref{far-control} that
$$
|u(y+c_\alpha t, t)-\phi_\alpha(y-h_1)|
\leq |u(y+c_\alpha t, t)-a_+| + |a_+-\phi_\alpha(y-h_1)| < 3\vep
$$
whenever $t \geq t_\vep$ and $y \geq R_\vep := \max\{x_2+\sigma\delta, h_1\} + L_\vep$. By virtue of \eqref{quasi-conv-0} and the fact that $u(\cdot+c_\alpha t_n,t_n) = v(\cdot,t_n) + a_+ = \alpha = \phi_\alpha(\cdot-h_1)$ on $(-\infty,0]$, one finds that
$$
\|u(\cdot+c_\alpha t_n,t_n)-\phi_\alpha(\cdot - h_1)\|_{L^\infty(-\infty,R_\vep)} < 3 \vep
$$
for $n \in \mathbb N$ large enough. Hence combing all these facts, we conclude that, there exists $N_\vep \in \mathbb N$ such that
$$
\sup_{x \in \R}|u(x,t_n) - \phi_\alpha(x-c_\alpha t_n-h_1)| 
= \sup_{y \in \R}|u(y+c_\alpha t_n,t_n) - \phi_\alpha(y-h_1)| 
< 3\vep,
$$
provided that $n \geq N_\vep$. Here $(t_n)$ denotes the subsequence which we have extracted so far. This completes the proof of Lemma \ref{L:quasi-conv}. \qed

\section{Exponential convergence to traveling waves}\label{S:exp-conv}

In this section, we shall prove Lemma \ref{L:exp-stbl}.

\subsection{Enclosing lemma}

In what follows, we fix $\beta \in (0,\beta_0)$ and $\sigma \in (0,\sigma_\beta)$ as in Lemma \ref{L:comfunc}. Here we recall that $\phi_\alpha$ always satisfies \eqref{phi-base} (i.e., the origin is the interfacial point of $\phi_\alpha$).

\begin{lemma}\label{L:enclo}
Let $\delta_0 > 0$ be the number appeared in Lemma \ref{L:comfunc} and choose $\beta$ and $\sigma$ as in Lemma \ref{L:comfunc}. Let $\rho_0 \in (0,1/4)$ be small enough that
\begin{equation}\label{rho0}
\sup_{r \leq 4\rho_0} \phi_\alpha'(r) 
< \frac 1 {2\sigma}.
\end{equation}
Let $u = u(x,t)$ be a non-decreasing {\rm (}in time{\rm )} $L^2_{\rm loc}$ solution to \eqref{obs} {\rm(}equivalently, \eqref{pde-alpha}{\rm)} in $\R \times \R_+$ such that there exists $r_0 \in (-\infty,\rho_0]$ satisfying 
\begin{equation}\label{hypo-r0}
u(\cdot,0) = \alpha \ \mbox{ in } \ (-\infty,r_0], \quad u(\cdot,0) > \alpha \ \mbox{ in } \ (r_0,+\infty).
\end{equation}
Let $h_0 > 0$ be fixed. Then there exist constants $\tau_0 > 0$ {\rm (}large enough{\rm )} and $\vep_0 \in (0,\sigma^{-1})$  {\rm (}small enough{\rm )} such that the following holds\/{\rm :} if there exist $\delta \in [0,\frac{\delta_0}2)$ and $h \in [0,h_0)$ such that 
\begin{equation}\label{hyp-ini}
\phi_\alpha(\cdot-h) - \delta \leq u(\cdot,0) \leq \phi_\alpha(\cdot) + \delta \ \mbox{ in } \R,
\end{equation}
then, for each $t \geq \tau_0$,
\begin{equation}\label{u-xt}
\phi_\alpha(\cdot-c_\alpha t + \xi_t-h_t)-\delta_t
\leq u(\cdot,t) \leq \phi_\alpha(\cdot-c_\alpha t + \xi_t) + \delta_t \ \mbox{ in } \R
\end{equation}
for some constants $\xi_t$, $\delta_t$ and $h_t$ satisfying
\begin{gather*}
\mbox{either } \ \xi_t = \sigma \delta - \sigma \vep_0 \bar h \ \mbox{ or } \ \xi_t = \sigma \delta,\\
\mbox{and } \ 0 \leq \delta_t \leq (\delta + \vep_0 \bar h) \e^{-\beta (t - \tau_0)},\quad
0 \leq h_t \leq h - \sigma (\vep_0 \bar h - 2 \delta), 
\end{gather*}
where $\bar h := h \wedge 1$. Here we note that $\delta_t < \delta$ and $h_t < h$, provided that $t$ is large enough and $\delta$ is small enough, respectively, unless either $\delta$ or $h$ is zero.
\end{lemma}

\begin{remark}
{\rm 
\begin{enumerate}
 \item[(i)] Even though either $\delta$ or $h$ is zero, both $\delta_t$ and $h_t$ can be positive. Indeed, upper bounds for $\delta_t$ and $h_t$ are linear combinations of $\delta$ and $h$.
 \item[(ii)] One can assume $\delta < 1$ without any loss of generality by choosing $\delta_0$ less than $2$ in Lemma \ref{L:comfunc}.
\end{enumerate}
}
\end{remark}

\begin{proof}
By Lemma \ref{L:comfunc} and comparison principle with \eqref{hyp-ini} (see Lemmas \ref{L:cp-obs+} and \ref{L:cp-sub} for $w^+$ and $w^-$, respectively, and see also Remarks \ref{R:equivalence} and \ref{R:comparison}), it holds that
\begin{equation}\label{w-uw+}
 w^-(x,t) \leq u(x,t) \leq w^+(x,t) \quad \mbox{ for } \ (x,t) \in \R \times (0,+\infty),
\end{equation}
where
\begin{align*}
 w^+(x,t) &:= \phi_\alpha(x - c_\alpha t + \sigma \delta (1-\e^{-\beta t})) + \delta \e^{-\beta t},\\
 w^-(x,t) &:= \phi_\alpha(x - c_\alpha t - \sigma \delta (1-\e^{-\beta t}) - h) - \delta \e^{-\beta t}
\end{align*}
for $\delta \in [0,\delta_0/2)$. Set
$$
W^\pm(x,t) := \pm \left[ w^\pm(x,t) - u(x,t) \right].
$$
Then it follows that $W^\pm \geq 0$ by \eqref{w-uw+}. Moreover, put $\bar h = h \wedge 1$ and note that
$$
\int^{3+h_0}_{2+h_0} \left[ \phi_\alpha(y) - \phi_\alpha(y-\bar h) \right] \, \d y
= \bar h \int^{3+h_0}_{2+h_0} \int^1_0 \phi_\alpha'(y-\theta \bar h) \, \d \theta \d y
\geq 2k \bar h,
$$
where $\theta = \theta(y) \in (0,1)$ and $k = k(\phi_\alpha,h_0)$ is given by
$$
k := \frac 1 2 \inf_{s \in [1+h_0,3+h_0]} \phi_\alpha'(s) > 0.
$$
Then one of the following is satisfied:
\begin{equation}\label{a2}
\int^{3+h_0}_{2+h_0} \left[ \phi_\alpha(y) - u(y,0) \right] \, \d y \geq k \bar h
\end{equation}
or
\begin{equation}\label{a1}
\int^{3+h_0}_{2+h_0} \left[ u(y,0) - \phi_\alpha(y-\bar h) \right] \, \d y \geq k \bar h.
\end{equation}
Let us first treat the case \eqref{a2}. To prove the assertion of the lemma, we set
$$
\xi_t^+ := c_\alpha t - \sigma \delta (1 - \e^{-\beta t}) \leq 0 \quad \mbox{ for } \ t > 0
$$ 
and divide $\R$ into three intervals $(-\infty, r_0+\rho_0]$, $(r_0+\rho_0,r)$ and $[r,+\infty)$ for some $r$ large enough.

\subsubsection{\rm Behavior on the right-interval for the case \eqref{a2}}

Since $W^+ \geq 0$, we find that
$$
u(x,t) \leq \phi_\alpha(x -\xi_t^+) + \delta \e^{-\beta t}
$$
for any $(x,t) \in \R \times \R^+$. We also note that
\begin{align*}
 \phi_\alpha(x-\xi_t^+) - \phi_\alpha(x-\xi_t^+ - 2\sigma \vep \bar h)
\leq 2 \sigma \vep \bar h \sup_{\theta \in (0,1)} \phi_\alpha'(x-\xi_t^+ -2\sigma\vep\theta\bar h).
\end{align*}
Here choose $r = r(\phi_\alpha,\sigma) > 0$ large enough that
\begin{equation}\label{r:+}
\sup_{y \geq r - 1} \phi_\alpha'(y) < \dfrac 1 {2\sigma}.
\end{equation}
If $\vep \in (0,\frac 1 {2\sigma}]$, noting that $x - \xi_t^+ - 2\sigma \vep \theta \bar h \geq x - 2 \sigma \vep \geq x - 1$, we then infer that
$$
\sup_{x \geq r}\sup_{\theta \in (0,1)} \phi_\alpha'(x-\xi_t^+ - 2\sigma\vep\theta\bar h) < \dfrac 1 {2\sigma},
$$
which implies
$$
\phi_\alpha(x-\xi_t^+) - \phi_\alpha(x-\xi_t^+- 2\sigma \vep \bar h) < \vep \bar h
\quad \mbox{ if } \ x \geq r.
$$
Thus
$$
u(x,t) \leq \phi_\alpha(x-\xi_t^+-2\sigma\vep\bar h) + \delta \e^{-\beta t} + \vep \bar h
$$
for all $x \geq r$, $t > 0$ and $\vep \in (0,\frac 1 {2\sigma}]$. 

\subsubsection{Behavior on the left-interval for the case \eqref{a2}}

Here we shall estimate the behavior of $u(x,t)$ for $x \leq r_0+\rho_0$. We have already known that
$$
u(x,t) \leq \phi_\alpha(x -\xi_t^+) + \delta \e^{-\beta t}.
$$
Note that
\begin{align*}
 \phi_\alpha(x-\xi_t^+) - \phi_\alpha(x-\xi_t^+ - 2\sigma \vep \bar h)
\leq 2 \sigma \vep \bar h \sup_{\theta \in (0,1)} \phi_\alpha'(x-\xi_t^+-2\sigma\vep\theta\bar h)
\end{align*}
for any $\vep > 0$. Choose $\tau_+ = \tau_+(\alpha,\sigma,\delta_0,\beta,\rho_0)> 0$ small enough that 
\begin{equation}\label{tau:+}
0 \leq - \xi_{\tau_+}^+ \leq \rho_0 \quad \mbox{ for all } \ \delta \in (0,\delta_0)
\end{equation}
(i.e., $x-\xi_{\tau_+}^+-2\sigma\vep\theta\bar h \leq r_0 + 2\rho_0 \leq 3 \rho_0$ if $x \leq r_0 + \rho_0$), and hence,
$$
\sup_{\theta \in (0,1)} \phi_\alpha'(x-\xi_{\tau_+}^+ - 2\sigma \vep \theta \bar h) < \dfrac 1 {2\sigma}
$$
by assumption \eqref{rho0}. Thus we obtain
$$
\phi_\alpha(x-\xi_{\tau_+}^+) - \phi_\alpha(x-\xi_{\tau_+}^+ - 2\sigma \vep \bar h) < \vep \bar h
$$
for any $x \leq r_0 + \rho_0$ and $\vep > 0$. Thus
$$
u(x,{\tau_+}) \leq \phi_\alpha(x-\xi_{\tau_+}^+-2\sigma\vep\bar h) + \delta \e^{-\beta {\tau_+}} + \vep \bar h
$$
for all $x \leq r_0+\rho_0$ and $\vep > 0$.

\subsubsection{Behavior in the middle-interval for the case \eqref{a2}}\label{Sss:+:mid}

We shall estimate the behavior of $u(x,t)$ for $x \in (r_0+\rho_0,r)$. Since $u$ solves $\partial_t u = \partial_x^2 u - f(u)$ for all $x > r_0$ and $t > 0$, it follows that
\begin{align*}
 \partial_t W^+ &\geq 
\partial_x^2 W^+ - f(w^+) + f(u)\\
&\geq \partial_x^2 W^+ - M W^+ \ \mbox{ for all } \ x > r_0, \ t > 0,
\end{align*}
where $M := \sup_{s\in[a_--1,a_++1]} |f'(s)|$. Hence $W^+$ is a supersolution to the Cauchy-Dirichlet problem,
\begin{gather}
 \partial_t w = \partial_x^2 w - Mw \ \mbox{ on } (r_0,+\infty) \times \R_+,\label{w1}\\
w(r_0,\cdot) = 0 \ \mbox{ in } \R_+, \quad w(\cdot,0) = W^+(\cdot,0) \ \mbox{ in } (r_0,+\infty),\label{w2}
\end{gather}
whose solution is represented by
\begin{align}
 w(x,t) &= \dfrac{\e^{-Mt}}{\sqrt{4\pi t}} \int^{+\infty}_0 \left[ \e^{- \frac{(x - r_0 - y)^2}{4t}}-\e^{- \frac{(x - r_0 + y)^2}{4t}} \right] W^+(r_0+y,0) \, \d y\nonumber\\
&= \dfrac{\e^{-Mt}}{\sqrt{4\pi t}} \int^{+\infty}_0 \e^{- \frac{(x - r_0 - y)^2}{4t}} \left[ 1 -\e^{- \frac{(x - r_0)y}{t}} \right] W^+(r_0+y,0) \, \d y\nonumber\\
&\geq \dfrac{\e^{-Mt}}{\sqrt{4\pi t}} \int^{3+h_0-r_0}_{2+h_0-r_0} \e^{- \frac{(x - r_0 - y)^2}{4t}} \left[ 1 -\e^{- \frac{(x - r_0)y}{t}} \right] W^+(r_0+y,0) \, \d y\nonumber\\
&\geq \dfrac{\e^{-Mt}}{\sqrt{4\pi t}}  \left[ 1 -\e^{- \frac{(x - r_0)(2+h_0-r_0)}{t}} \right] \int^{3+h_0}_{2+h_0} \e^{- \frac{(x - y)^2}{4t}} W^+(y,0) \, \d y\nonumber\\
&\geq \dfrac{\e^{-Mt}}{\sqrt{4\pi t}} \left[ 1 -\e^{- \frac{(x - r_0)(2+h_0-r_0)}{t}} \right] \e^{- \frac{x^2 + (3+h_0)^2 }{2t}} \int^{3+h_0}_{2+h_0} W^+(y,0) \, \d y\label{w+-SMP}
\end{align}
for $(x,t) \in (r_0,+\infty) \times (0,+\infty)$. Put $t = {\tau_+}$ defined by \eqref{tau:+}. We then see that, for any $x \in (r_0+\rho_0,r)$,
\begin{align*}
 W^+(x,{\tau_+}) &\geq \dfrac{\e^{-M{\tau_+}}}{\sqrt{4\pi {\tau_+}}} \left[ 1 -\e^{- \frac{\rho_0(2+h_0-\rho_0)}{{\tau_+}}} \right] \e^{- \frac{r^2 + (3+h_0)^2 }{2{\tau_+}}} \int^{3+h_0}_{2+h_0}  W^+(y,0) \, \d y\\
&=: C_+(M,{\tau_+},r,\rho_0,h_0) \int^{3+h_0}_{2+h_0} W^+(y,0) \, \d y.
\end{align*}
Hence
\begin{align*}
 W^+(x,{\tau_+}) &\geq C_+(M,{\tau_+},r,\rho_0,h_0) \int^{3+h_0}_{2+h_0} \left[ \phi_\alpha(y) + \delta - u(y,0) \right] \, \d y\\
&\stackrel{\eqref{a2}}> C_+(M,{\tau_+},r,\rho_0,h_0) k \bar h > 0
\quad \mbox{ for all } \ x \in [r_0+\rho_0,r].
\end{align*}
On the other hand, we have
\begin{align*}
 W^+(x,{\tau_+}) &= \phi_\alpha(x-\xi_{\tau_+}^+) + \delta \e^{-\beta {\tau_+}} - u(x,{\tau_+})\\
&\leq \phi_\alpha(x - \xi_{\tau_+}^+ - 2 \sigma \vep \bar h) + 2 \sigma \vep \bar h \sup_{\theta \in (0,1)} \phi_\alpha'(x - \xi_{\tau_+}^+ - 2 \sigma \vep \theta \bar h) \\
&\quad + \delta \e^{-\beta {\tau_+}} - u(x,{\tau_+})
\end{align*}
for $\vep > 0$. Set $\vep_+ = \vep_+(\phi_\alpha,\alpha,\sigma,\delta_0,\beta,\rho_0,M,h_0)$ by
\begin{equation}\label{vep:+}
 \vep_+ := \min \left\{ \frac 1 {2\sigma}, \frac{C_+(M,{\tau_+},r,\rho_0,h_0) k}{2\sigma \sup_{s\in \R} \phi_\alpha'(s)},\frac{\delta_0}2
\right\} > 0
\end{equation}
(it is enough to choose $\vep_+$ not greater than the minimum above). Then
$$
u(x,{\tau_+}) \leq \phi_\alpha(x - \xi_{\tau_+}^+ - 2 \sigma \vep_+ \bar h) + \delta \e^{-\beta \tau_+}
$$
for $x \in (r_0+\rho_0,r)$. 

\subsubsection{Conclusion for the case \eqref{a2}}\label{sss:conc+}

Combining all these facts, we have
\begin{equation}\label{case+}
u(x,\tau_+) \leq \phi_\alpha(x-\xi_{\tau_+}^+-2\sigma \vep_+\bar h) + \delta \e^{-\beta \tau_+} + \vep_+ \bar h 
\end{equation}
for all $x \in \R$. Recall that $\tau_+$ and $\vep_+$ are chosen by \eqref{tau:+} and \eqref{vep:+}, respectively, and they depend only on $\phi_\alpha,\alpha,\sigma,\delta_0,\beta,\rho_0,M,h_0$. Moreover, for any $s > 0$, we deduce that
\begin{align*}
\lefteqn{
u(x,{\tau_+}+s)
}\\
 &\leq \phi_\alpha(x - c_\alpha ({\tau_+} + s) + \sigma (\delta \e^{-\beta {\tau_+}} + \vep_+ \bar h ) (1- \e^{-\beta s}) + \sigma\delta(1-\e^{-\beta {\tau_+}})-2\sigma \vep_+\bar h)\\
&\quad  + ( \delta \e^{-\beta {\tau_+}} + \vep_+ \bar h ) \e^{-\beta s} 
\end{align*}
by Lemma \ref{L:comfunc}. Here we used the fact that $\delta \e^{-\beta {\tau_+}} + \vep_+ \bar h < \frac{\delta_0}2 + \frac{\delta_0}2 < \delta_0$ for any $\delta \in [0,\delta_0/2)$ by \eqref{vep:+} and $\bar h \leq 1$. Set $t = {\tau_+}+s$. Then, for $t \geq {\tau_+}$,
\begin{align*}
u(x,t) &\leq \phi_\alpha(x - c_\alpha t + \sigma (\delta \e^{-\beta {\tau_+}} + \vep_+ \bar h ) (1- \e^{-\beta s}) +\sigma \delta (1-\e^{-\beta {\tau_+}}) -2\sigma \vep_+\bar h)\\
&\quad  + ( \delta \e^{-\beta {\tau_+}} + \vep_+ \bar h ) \e^{-\beta s}. 
\end{align*}
Here we set
\begin{align*}
 \delta(t) &:= ( \delta \e^{-\beta {\tau_+}} + \vep_+ \bar h ) \e^{-\beta s}\\
&= \delta \e^{-\beta t} + \vep_+ \bar h \e^{-\beta(t-{\tau_+})}
\leq \left( \delta + \vep_+ \min\{h,1\} \right) \e^{-\beta(t-{\tau_+})}
\end{align*}
(then $\delta (t) \geq \delta \e^{-\beta t}$). Moreover, one finds that
\begin{align*}
\lefteqn{
\sigma (\delta \e^{-\beta {\tau_+}} + \vep_+ \bar h ) (1- \e^{-\beta s}) + \sigma \delta (1-\e^{-\beta {\tau_+}}) -2\sigma \vep_+\bar h
}\\
&\leq \sigma (\delta \e^{-\beta {\tau_+}} + \vep_+ \bar h ) \cdot 1 + \sigma \delta (1-\e^{-\beta {\tau_+}}) -2\sigma \vep_+\bar h\\
&=\sigma (\delta + \vep_+ \bar h ) - 2\sigma \vep_+\bar h = \sigma \delta - \sigma \vep_+\bar h =: \xi(t)
\end{align*}
and $\xi(t) \geq - \bar h/2$. Thus
$$
u(x,t) \leq \phi_\alpha(x - c_\alpha t + \xi(t)) + \delta(t)
\quad \mbox{ for all } \ x \in \R, \ t > \tau_+.
$$
On the other hand, recalling that
\begin{align*}
u(x,t) \geq w^-(x,t) = \phi_\alpha(x - c_\alpha t - \sigma \delta (1-\e^{-\beta t}) - h) - \delta \e^{-\beta t}, 
\end{align*}
we deduce that
$$
u(x,t) \geq \phi_\alpha(x - c_\alpha t + \xi(t) - h(t)) - \delta(t)
\quad \mbox{ for all } \ x \in \R, \ t > \tau_+,
$$
where $h(t)$ is given by
\begin{align*}
h(t) &:= \sigma \delta (1-\e^{-\beta t}) + h + \xi(t)\\
&= \sigma \delta (2-\e^{-\beta t}) + h - \sigma \vep_+\bar h\\
&\leq h - \sigma(\vep_+ \bar h - 2 \delta). 
\end{align*}
We also note that $h(t) \geq h - \bar h/2 \geq h/2 \geq 0$. Thus we have obtained the assertion of the lemma.

We next treat the case \eqref{a1}. We set
$$
\xi_t^- := c_\alpha t + \sigma \delta (1 - \e^{-\beta t}) \quad \mbox{ for } \ t > 0.
$$
Here we remark that the sign of $\xi_t^-$ may be indefinite, although $\xi_t^+$ is non-positive. We shall divide $\R$ into three intervals $(-\infty,h+2\rho_0]$, $(h+2\rho_0,r)$ and $[r,+\infty)$ for some $r > 0$ large enough.

\subsubsection{Behavior on the right-interval for the case \eqref{a1}}

Recall that
$$
u(x,t) \geq \phi_\alpha(x -\xi_t^- -h) - \delta \e^{-\beta t} \quad \mbox{ for any } (x,t) \in \R \times \R_+
$$
and note that
\begin{align*}
\lefteqn{
 \phi_\alpha(x-\xi_t^- -h) - \phi_\alpha(x-\xi_t^- + 2\sigma \vep \bar h-h)
}\\&
\geq - 2 \sigma \vep \bar h \sup_{\theta \in (0,1)} \phi_\alpha'(x-\xi_t^-+2\sigma\vep\theta\bar h-h).
\end{align*}
Here choose $r = r(\phi_\alpha,\delta_0,\sigma,h_0) > 0$ large enough that
\begin{equation}\label{r:-}
\sup_{y \geq r-\sigma\delta_0-h_0} \phi_\alpha'(y) < \frac 1 {2\sigma}. 
\end{equation}
Noting that $x-\xi_t^-+2\sigma\vep\theta\bar h-h \geq x - \sigma \delta_0 - h_0$, we infer that
$$
\sup_{x \geq r}\sup_{\theta \in (0,1)} \phi_\alpha'(x-\xi_t^-+2\sigma\vep\theta\bar h-h) < \dfrac 1 {2\sigma},
$$
which yields
$$
\phi_\alpha(x-\xi_t^- -h) - \phi_\alpha(x-\xi_t^- + 2\sigma \vep \bar h -h) > - \vep \bar h.
$$
Thus
$$
u(x,t) \geq \phi_\alpha(x-\xi_t^-+2\sigma\vep\bar h-h) - \delta \e^{-\beta t} - \vep \bar h
$$
for all $x \geq r$, $t > 0$ and $\vep > 0$.

\subsubsection{Behavior on the left-interval for the case \eqref{a1}} 

We shall estimate $u(x,t)$ on $(-\infty, h+2\rho_0]$. We see that
\begin{align*}
 u(x,{t}) &\geq \phi_\alpha(x-\xi_{t}^--h)-\delta\e^{-\beta {t}}\\
&\geq \phi_\alpha(x-\xi_{t}^-+2\sigma \vep\bar h-h)  - \delta \e^{-\beta {t}}
\\
&\quad - 2\sigma \vep\bar h \sup_{\theta \in (0,1)} \phi_\alpha'(x-\xi_{t}^-+2\sigma\vep\bar h \theta - h)
\end{align*}
for any $x \in \R$, $t \in \R^+$ and $\vep > 0$. Choose $\tau_-^0 = \tau_-^0(\rho_0,c_\alpha)$ by
$$
\tau_-^0 = \frac{\rho_0}{-2c_\alpha} > 0.
$$
Then we deduce by \eqref{rho0} that
\begin{align*}
 \sup_{\theta \in (0,1)} \phi_\alpha'(x-\xi_t^-+2\sigma\vep\bar h \theta - h) < \dfrac 1 {2\sigma}
\end{align*}
by noting that
$$
x-\xi_t^-+2\sigma\vep\bar h \theta - h 
\leq 2\rho_0-c_\alpha t + \frac{\rho_0}2
< 3\rho_0
$$
for all $x \leq h + 2 \rho_0$, $t \in (0,\tau_-^0)$ and $\vep \in (0,\frac{\rho_0}{4\sigma})$. Consequently, one obtains
\begin{align*}
u(x,t) \geq \phi_\alpha(x-\xi_t^-+2\sigma \vep\bar h-h) 
- \delta \e^{-\beta t} - \vep \bar h
\end{align*}
for all $x \leq h + 2 \rho_0$, $t \in (0,\tau_-^0)$ and $\vep \in (0,\tfrac{\rho_0}{4\sigma})$.

\subsubsection{Behavior in the middle-interval for the case \eqref{a1}}

We recall that $w^-$ is a subsolution to \eqref{ac-w} in $Q_-^h := \{(x,t) \in \R \times \R_+ \colon x > \xi_t^- + h \}$. By Lemma \ref{L:comfunc} and (ii) of Remark \ref{R:comparison},
\begin{align*}
\partial_t W^- &\geq \partial_x^2 u - f(u) - \partial_x^2 w^- + f(w^-)\\
&= \partial_x^2 W^- - f(u) + f(w^-)\\
&\geq \partial_x^2 W^- - M W^-,
\end{align*}
where $M$ is given as in \S \ref{Sss:+:mid}, for all $(x,t) \in Q_-^h$. Hence $W^-$ becomes a supersolution to \eqref{w1}, \eqref{w2} with $W^+(\cdot,0)$ replaced by $W^-(\cdot,0)$ in $Q_-^h$. Now, one can take a number $\tau_- = \tau_-(\phi_\alpha, \alpha,\sigma,\delta_0,\beta,\rho_0)$ such that
\begin{equation}\label{tau:-}
0 < \tau_- < \tau_-^0 \quad \mbox{ and } \quad \xi_t^- < \rho_0 \ \mbox{ for } \ t \in [0,\tau_-].
\end{equation}
Hence, $\partial_t W^- \geq \partial_x^2 W^- - M W^-$ in $(\rho_0+h,+\infty) \times (0,\tau_-)$. Thus replacing $r_0$ by $\rho_0 + h$ in \eqref{w+-SMP} with $t = \tau_-$, we can deduce that, for any $x \in (2\rho_0+h,r)$,
\begin{align*}
 W^-(x,\tau_-) &\geq \dfrac{\e^{-M\tau_-}}{\sqrt{4\pi \tau_-}} \left[ 1 -\e^{- \frac{\rho_0(2-\rho_0)}{\tau_-}} \right] \e^{- \frac{r^2 + (3+h_0)^2 }{2\tau_-}} \int^{3+h_0}_{2+h_0} W^-(y,0) \, \d y\\
&=: C_-(M,\tau_-,r,\rho_0,h_0) \int^{3+h_0}_{2+h_0} W^-(y,0) \, \d y\\
&> C_-(M,\tau_-,r,\rho_0,h_0) \int^{3+h_0}_{2+h_0} [u(y,0)-\phi_\alpha(y-\bar h)] \, \d y\\
&\stackrel{\eqref{a1}}\geq C_-(M,\tau_-,r,\rho_0,h_0) k \bar h.
\end{align*}
On the other hand, one deduces that
\begin{align*}
 W^-(x,{\tau_-}) &= u(x,{\tau_-}) - \phi_\alpha(x-\xi_{\tau_-}^--h) + \delta \e^{-\beta {\tau_-}}\\
&\leq u(x,{\tau_-}) - \phi_\alpha(x - \xi_{\tau_-}^- + 2 \sigma \vep \bar h-h)\\
&\quad + 2 \sigma \vep \bar h \sup_{\theta \in (0,1)} \phi_\alpha'(x - \xi_{\tau_-}^- + 2 \sigma \vep \theta \bar h-h) + \delta \e^{-\beta {\tau_-}} 
\end{align*}
for any $\vep > 0$. Choose $\vep_- = \vep_-(\phi_\alpha,\alpha,\sigma,\delta_0,\beta,\rho_0,M,h_0)$ by
\begin{equation}\label{vep:-}
\vep_- := \min \left\{ \frac{\rho_0}{4\sigma}, \dfrac{\delta_0}2, \dfrac{C_-(M,\tau_-,r,\rho_0,h_0) k}{2 \sigma \sup_{s\in \R} \phi_\alpha'(s)} \right\}>0
\end{equation}
(it is enough to choose $\vep_-$ not greater than the minimum above). Thus
$$
u(x,{\tau_-}) \geq \phi_\alpha(x - \xi_{\tau_-}^- + 2 \sigma \vep_- \bar h-h) - \delta \e^{-\beta {\tau_-}}
$$
for $x \in (2\rho_0+h,r)$.

\subsubsection{Conclusion for the case \eqref{a1}}

Combining all these facts, we have
\begin{equation}\label{case-}
u(x,\tau_-) \geq \phi_\alpha(x-\xi_{\tau_-}^-+2\sigma \vep_-\bar h -h) - \delta \e^{-\beta \tau_-} - \vep_- \bar h 
\end{equation}
for all $x \in \R$. Recall that $\tau_-$ and $\vep_-$ are chosen as in \eqref{tau:-} and \eqref{vep:-}, respectively, and they depend only on $\phi_\alpha,\alpha,\sigma,\delta_0,\beta,\rho_0,M,h_0$. Repeating the same argument as in \S \ref{sss:conc+}, we can verify \eqref{u-xt} by setting $\xi(t) = \sigma \delta \geq \sigma \delta(1-\e^{-\beta t})$, $0 \leq \delta(t) \leq (\delta+\vep_- \min \{h,1\}) \e^{-\beta(t-\tau_-)}$ and $0\leq h(t) \leq \sigma \delta - \sigma \vep_- \bar h + h + \xi(t) = h - \sigma(\vep_-\bar h - 2\delta)$.

Consequently, we have proved Lemma \ref{L:enclo} by setting $\tau_0 := \max \{\tau_+,\tau_-\} > 0$, $\vep_0 := \min \{\vep_+,\vep_-\} > 0$, $\xi_t := \xi(t)$, $\delta_t := \delta(t)$ and $h_t := h(t)$ defined above with $\vep_\pm$ replaced by $\vep_0$.
\end{proof}

\subsection{Exponential convergence}

In this subsection, we shall prove the following lemma by modifying an iterative argument used in~\cite{XChen} to take into account of the additional requirement appeared in Lemma \ref{L:enclo}, i.e., there exists $r_0 \in (-\infty,\rho_0]$ satisfying \eqref{hypo-r0} (see also \eqref{rho0} for $\rho_0$), which roughly means that the free boundary $r_0$ of the region $\{x \in \R \colon u(x,0) = \alpha\}$ is close enough to that of a barrier function.
\begin{lemma}\label{L:asym-stbl}
There exist constants $x_0 \in \R$ and $K, \kappa > 0$ such that
\begin{equation}\label{exp-conv}
\|u(\cdot,t)-\phi_\alpha(\cdot-c_\alpha t+x_0)\|_{L^\infty(\R)} \leq K \e^{-\kappa t}
\end{equation}
for all $t \geq 0$.
\end{lemma}

\begin{proof}
Let $\rho_0 > 0$ be given by \eqref{rho0}. By Lemma \ref{L:quasi-conv} with translation in space and time, one can ensure that \eqref{hyp-ini} holds with
\begin{equation}\label{h_and_delta}
h = 0 \quad \mbox{ and } \quad \delta = \delta_* := \min\{\tfrac{\delta_0}2,\tfrac{\rho_0}{4\sigma},\tfrac{\vep_0}{4}\cdot\tfrac{\rho_0}{2}, \phi_\alpha(\tfrac{\rho_0}2)-\alpha\} > 0.
\end{equation}
Set
$$
\kappa_* := 1 - \frac{\sigma\vep_0}2 \in (0,1).
$$
Put $r_0 := \sup\{r \in \R \colon u(x,0) = \alpha \ \mbox{ for all } \ x \leq r\}$. Then \eqref{hypo-r0} holds with $r_0 \leq \frac{\rho_0}2$; indeed, it follows that $r_0 \leq \phi_\alpha^{-1}(\alpha+\delta) \leq \frac{\rho_0}2$ from the fact that $\phi_\alpha(\cdot) - \delta \leq u(\cdot,0)$ and \eqref{h_and_delta}. Therefore applying Lemma \ref{L:enclo}, we assure that \eqref{u-xt} holds with $\xi_t$, $\delta_t$ and $h_t$ satisfying
\begin{equation}\label{as:1}
\xi_t = \sigma \delta,\ 
0 \leq \delta_t \leq \delta \e^{-\beta (t-\tau_0)} \leq \delta_*, \
0 \leq h_t \leq 2 \sigma \delta \leq h_*:=\dfrac{\rho_0}2<1. 
\end{equation}
We further note by the first inequality of \eqref{u-xt} at $x = r(t) := \sup \{r \in \R \colon u(x,t) = \alpha \ \mbox{ for all } \ x \leq r\}$ that 
$$
\phi_\alpha(r(t)-c_\alpha t + \xi_t - h_t) - \delta_t \leq u(r(t),t) = \alpha,
$$ 
which implies
$$
r(t) - c_\alpha t + \xi_t \leq h_t + \phi_\alpha^{-1} ( \alpha + \delta_t)
\leq \frac{\rho_0}2 + \phi_\alpha^{-1} ( \alpha + \delta_t).
$$
We note by $\delta_t < \delta$ and \eqref{h_and_delta} that
\begin{equation*}
\phi_\alpha^{-1} ( \alpha + \delta_t) < \phi_\alpha^{-1} ( \alpha + \delta ) \leq \frac{\rho_0}2,
\end{equation*}
which gives
\begin{equation}\label{r0}
r(t)-c_\alpha t+\xi_t < \rho_0 \quad \mbox{ for } \ t > 0.
\end{equation}
Now, let us take $t_* > \tau_0$ large enough that
\begin{equation}\label{kappa*}
(1+\vep_0 h_*/\delta_*) \e^{-\beta(t_*-\tau_0)} \leq \kappa_*,
\end{equation}
which will be used later. We set $T_1 = t_*$, $\delta_1 = \delta_*$, $h_1 = h_*$ and $\xi_1 = \xi_{t_*} = \sigma \delta$.

We next claim that, for any $k \in \mathbb{N}$, 
\begin{enumerate}
 \item[(a)] it holds that
\begin{equation}\label{es-cl-a}
\phi_\alpha(x-c_\alpha T_k+\xi_k-h_k) - \delta_k \leq u(x,T_k) \leq \phi_\alpha(x-c_\alpha T_k+\xi_k) +\delta_k
\end{equation}
with
$$
T_k := k t_*, \quad \delta = \delta_k := \kappa_*^{k-1} \delta_*, \quad h = h_k := \kappa_*^{k-1} h_*, \quad \xi_k \in \R\/{\rm ;} 
$$
 \item[(b)] the interfacial point $r(T_k)$ of $u(\cdot,T_k)$ satisfies
$$
r(T_k) - c_\alpha T_k + \xi_k < \rho_0,
$$
where the left-hand side corresponds to $r_0$ of $\eqref{hypo-r0}$ for $u(\cdot,T_k)$ shifted in space by $c_\alpha T_k - \xi_k$.
\end{enumerate}
Indeed, both (a) and (b) have already been proved for $k = 1$. Suppose that (a) and (b) hold for $k$. Then by Lemma \ref{L:enclo} along with both (a) and (b) for $k$, we have 
\begin{align}\label{order-T_k+1}
\phi_\alpha(x-c_\alpha T_{k+1}+\xi_k+\xi_{k,t_*}-h_{k,t_*}) - \delta_{k,t_*} 
\leq u(x,T_{k+1}) \qquad
\\
\leq \phi_\alpha(x-c_\alpha T_{k+1}+\xi_k+\xi_{k,t_*}) +\delta_{k,t_*}
\nonumber
\end{align}
for some constants $\xi_{k,t_*}$, $h_{k,t_*}$ and $\delta_{k,t_*}$ satisfying
\begin{gather*}
 \xi_{k,t_*} \in \{\sigma \delta_k - \sigma \vep_0 h_k, \sigma \delta_k\},\\
0 \leq \delta_{k,t_*} \leq (\delta_k + \vep_0 h_k) \e^{-\beta(t_*-\tau_0)}
= \kappa_*^{k-1} (\delta_* + \vep_0 h_*) \e^{-\beta(t_*-\tau_0)}
\stackrel{\eqref{kappa*}}\leq \kappa_*^k \delta_* = \delta_{k+1}
\end{gather*}
and
\begin{align*}
0 \leq h_{k,t_*} &\leq h_k - \sigma (\vep_0 h_k - 2 \delta_k)\\
&= \kappa_*^{k-1} \left[ (1-\sigma \vep_0) h_* + 2\sigma \delta_*\right]
 \leq \kappa_*^{k-1} \left(1 - \frac{\sigma\vep_0}2\right) h_*
= \kappa_*^k h_* = h_{k+1}.
\end{align*}
Here we used the fact that $\delta_* \leq  \frac{\vep_0}4 \cdot \frac{\rho_0}2 = \frac{\vep_0}4 h_*$. Hence (a) is satisfied for $k + 1$, where $\xi_{k+1} := \xi_k + \xi_{k,t_*} \in \{\xi_k +\sigma\delta_k - \sigma \vep_0 h_k, \xi_k+ \sigma \delta_k\}$. Moreover, substituting $x = r(T_{k+1})$ in the first inequality of \eqref{order-T_k+1} and using $h_{k,t_*} \leq \kappa_*^k h_* < \rho_0/2$ and $\delta_{k,t_*} < \delta_* = \delta$, we can also deduce as in \eqref{r0} that
$$
r(T_{k+1}) - c_\alpha T_{k+1} + \xi_{k+1}
\leq h_{k,t_*} + \phi_\alpha^{-1}(\alpha + \delta_{k,t_*}) < \rho_0,
$$
which ensures (b) for $k+1$. Thus the claim is proved.

The rest of the proof for Lemma \ref{L:asym-stbl} runs as in~\cite{XChen}. However, for the convenience of the reader, we shall complete it. By virtue of Lemma \ref{L:comfunc}, we deduce that
\begin{equation}\label{u-ext}
\phi_\alpha(x-c_\alpha t+\xi-h) - \delta \leq u(x,t) \leq \phi_\alpha(x-c_\alpha t+\xi) +\delta, 
\end{equation}
where $\xi = \xi_k +\sigma \delta_k$, $\delta = \delta_k$ and $h = h_k + 2 \sigma \delta_k$, for any $t \geq T_k$ and $x \in \R$. Now, define piecewise-constant interpolants,
$$
\delta(t)=\delta_k, \quad \xi(t)=\xi_k+\sigma \delta_k, \quad h(t)=h_k+2\sigma\delta_k \quad \mbox{ for } \ t \in [T_k, T_{k+1})
$$
for $k \geq 1$. Then it follows from \eqref{u-ext} that
\begin{equation}\label{AAA}
\phi_\alpha(x-c_\alpha t+\xi(t)-h(t)) - \delta(t) \leq u(x,t) \leq \phi_\alpha(x-c_\alpha t +\xi(t)) + \delta(t)
\end{equation}
for $x \in \R$ and $t \geq T_1$. 

For each $t \geq T_1$, denote by $k_t$ the integer satisfying
$$
k_t \leq \frac{t}{t_*} < k_t + 1, \ \mbox{ i.e., } \ T_{k_t} = k_t t_* \leq t < (k_t+1) t_* = T_{k_t+1}.
$$
Then one observes by $\kappa_* \in (0,1)$ that
\begin{align}
 \delta(t) &= \delta_{k_t} = \kappa_*^{k_t-1} \delta_*
 = \delta_* \exp \left((k_t-1)\log \kappa_*\right)
 \leq \delta_* \exp \left( \frac{\log \kappa_*}{t_*} (t-2t_*) \right),\label{exp-d}\\
 h(t) &= h_{k_t} + 2 \sigma \delta_{k_t} = \kappa_*^{k_t-1} h_* + 2 \sigma \delta_* \kappa_*^{k_t-1} \leq (h_* + 2 \sigma \delta_*) \exp \left( \frac{\log \kappa_*}{t_*} (t-2t_*) \right),\label{exp-h}
\end{align}
which particularly yields that $h(t) \to 0$, $\delta(t) \to 0$  as $t \to +\infty$. Moreover, let $s \in [T_{k_t - 1}, T_{k_t})$ (i.e., $T_{k_s} = T_{k_t-1}$). Then noting that
$$
\xi(t) \in \xi(s) + \sigma \delta(t) + \{- \sigma \vep_0 h(s) + 2 \sigma^2 \vep_0 \delta(s), 0 \},
$$
we derive that
\begin{align}\label{H(t)}
|\xi(t)-\xi(s)| &\leq \sigma \delta(t) + \sigma \vep_0 h(s)+2\sigma^2 \vep_0 \delta(s)\\
&\leq C \exp \left( \frac{\log \kappa_*}{t_*} (s-2t_*) \right) =: H(s).\nonumber
\end{align}
Moreover, for any $t \geq s$, we observe that
\begin{align*}
\lefteqn{
 |\xi(t)-\xi(s)|
}\\
&\leq |\xi(t)-\xi(T_{k_t})| + |\xi(T_{k_t})-\xi(T_{k_t-1})| + \cdots + |\xi(T_{k_s+1}) - \xi(s)|\\
&\leq H(T_{k_t-1}) + \cdots + H(s)
\leq \dfrac{H(T_{k_s})}{1 - \kappa_*}
\end{align*}
from the fact that $H(T_{k+1}) = \kappa_* H(T_k)$, and hence, $\xi(t)$ converges to a limit $x_0 \in \R$ as $t \to +\infty$. Furthermore, it follows that
\begin{equation}\label{xi-exp-conv}
|x_0-\xi(s)| \leq \dfrac{H(T_{k_s})}{1 - \kappa_*}
\leq \dfrac C {1 - \kappa_*} \e^{{\frac{\log \kappa_*}{t_*}(s-3t_*)}}.
\end{equation}

Now, we also observe by \eqref{AAA} that
\begin{align*}
\lefteqn{
 u(x,t)-\phi_\alpha(x-c_\alpha t - x_0)
}\\
&\leq \phi_\alpha(x-c_\alpha t+\xi(t))+\delta(t)-\phi_\alpha(x-c_\alpha t - x_0)\\
&\leq \delta(t) + \sup_{x \in \R} |\phi_\alpha'(x)| |\xi(t)-x_0|
\end{align*}
and
\begin{align*}
 u(x,t)-\phi_\alpha(x-c_\alpha t - x_0)
&\geq - \delta(t) - \sup_{x \in \R} |\phi_\alpha'(x)| |\xi(t)-x_0-h(t)|.
\end{align*}
Thus using \eqref{exp-d}, \eqref{exp-h} and \eqref{xi-exp-conv}, we have proved the lemma.
\end{proof}

\subsection{Exponential convergence of the free boundary}

We finally prove the following:

\begin{lemma}\label{L:r-conv}
It holds that
\begin{equation}\label{r-exp-conv}
 |r(t)-c_\alpha t + x_0| = O(\e^{-\frac{\kappa t}2}) \ \mbox{ as } \ t \to +\infty.
\end{equation}
\end{lemma}

\begin{proof}
We first prove that
\begin{equation}\label{r-conv}
 \lim_{t \to +\infty} |r(t)-c_\alpha t + x_0| = 0.
\end{equation}
To this end, substitute $x = r(t)$ in \eqref{exp-conv} to find that
$$
\phi_\alpha(r(t)-c_\alpha t + x_0) \leq \alpha + K \e^{-\kappa t}.
$$
From the fact that $\phi_\alpha(s) = \alpha + \frac{f(\alpha)}2 s^2 + O(s^3)$ for $s > 0$ small enough, we find that
\begin{equation}\label{r-sup-exp}
\left( r(t) - c_\alpha t + x_0 \right)_+ \leq O(\e^{-\frac{\kappa t}2}),
\end{equation}
which in particular implies
\begin{equation}\label{limsup}
\limsup_{t \to +\infty} \left( r(t) - c_\alpha t + x_0 \right) \leq 0. 
\end{equation}

We further claim that
\begin{equation}\label{liminf}
\limsup_{t \to +\infty} \left( c_\alpha t - x_0 - r(t) \right) \leq 0. 
\end{equation}
To this end, suppose on the contrary that \eqref{liminf} does not hold. Namely, there exist $\delta_0 > 0$ (small enough without any loss of generality) and a sequence $t_n \to +\infty$ such that
\begin{equation}\label{liminf0}
r(t_n) < c_\alpha t_n - x_0 - \delta_0 \quad \mbox{ for all } \ n \in \mathbb{N}.
\end{equation}
On the other hand, \eqref{exp-conv} yields
$$
\alpha \leq \sup_{x \leq c_\alpha t_n - x_0} u(x, t_n) \leq \alpha + o(1) \quad \mbox{ as } \ n \to +\infty.
$$
Here we find by \eqref{liminf0} that $c_\alpha t_n - x_0 > r(t_n) + \delta_0$. Hence it follows that
\begin{equation}\label{u-conv-cont}
\sup_{y \in (-\infty,\delta_0]}u(r(t_n)+y,t_n) \to \alpha. 
\end{equation}
Now, since $\partial_t u \geq 0$ and $u$ solves \eqref{ac-w} in $(r(t_n),+\infty)$ at $t = t_n$, we have
\begin{align*}
\partial_x^2 u(r(t_n)+y,t_n) 
&= \partial_t u (r(t_n)+y,t_n) + f(u(r(t_n)+y,t_n))\\
&\geq f(u(r(t_n)+y,t_n))\\
& \to f(\alpha) > 0 \ \mbox{ uniformly for } y \in (0,\delta_0]
\end{align*}
as $n \to +\infty$. By Taylor's theorem (see \S \ref{Sss:CR} and \S \ref{Sss:BFB}), we can take $\theta_n \in (0,1)$ such that
\begin{align*}
u(r(t_n)+\delta_0,t_n) - \alpha
&=u(r(t_n)+\delta_0,t_n) - u(r(t_n),t_n)\\
&= \dfrac 1 2 \partial_x^2 u(r(t_n)+\theta_n \delta_0,t_n) \delta_0^2\\
&> \dfrac{f(\alpha)} 4 \delta_0^2 > 0 \quad \mbox{ for } n \mbox{ large enough},
\end{align*}
which yields a contradiction to \eqref{u-conv-cont}. Thus we have proved \eqref{liminf}, which along with \eqref{limsup} implies \eqref{r-conv}.

Furthermore, by Taylor's theorem, one can take $\theta_t \in (0,1)$ such that
\begin{equation}\label{Taylor}
u(c_\alpha t - x_0 , t)
= \alpha + \frac 1 2 \partial_x^2 u \big(r(t) + \theta_t (c_\alpha t - x_0 - r(t)),t \big)
 (c_\alpha t - x_0 - r(t))^2,
\end{equation}
whenever $c_\alpha t - x_0 > r(t)$. Hence by \eqref{exp-conv} and $\partial_t u \geq 0$, one has $\theta_t \in (0,1)$ such that
\begin{align*}
\alpha + K \e^{-\kappa t}
&\stackrel{\eqref{exp-conv}}\geq u(c_\alpha t - x_0 , t)\\
&\stackrel{\eqref{Taylor}}= \alpha + \frac 1 2 \partial_x^2 u\big(r(t) + \theta_t (c_\alpha t - x_0 - r(t)),t \big) (c_\alpha t - x_0 - r(t))^2.\\
&\geq \alpha + \frac 1 2 f \left( u\big(r(t) + \theta_t (c_\alpha t - x_0 - r(t)),t \big)\right)
 (c_\alpha t - x_0 - r(t))^2,
\end{align*}
whenever $c_\alpha t - x_0 > r(t)$. Here we note by \eqref{exp-conv} and \eqref{r-conv} that
\begin{align*}
\lefteqn{
u\big(r(t) + \theta_t (c_\alpha t - x_0 - r(t)),t \big)
}\\
&= \phi_\alpha \left( (1-\theta_t)\left(r(t)-c_\alpha t + x_0\right)\right)
+ O(\e^{-\kappa t}) \in I(\alpha),
\end{align*}
where $I(\alpha)$ is a neighbourhood of $\alpha$ such that $\inf_{s \in I(\alpha)} f'(s) > 0$ for $t$ large enough. Thereby
\begin{align*}
 \alpha + K \e^{-\kappa t}
&\geq \alpha + \frac 1 2 \left( \inf_{s \in I(\alpha)}f (s) \right) (c_\alpha t - x_0 - r(t))^2,
\end{align*}
whenever $c_\alpha t - x_0 > r(t)$. Thus
$$
\left( c_\alpha t - x_0 - r(t) \right)_+ \leq \left( \frac{2K \e^{-\kappa t}}{\inf_{s \in I(\alpha)}f (s)}\right)^{1/2}
$$
for $t$ large enough. Combining this fact with \eqref{r-sup-exp}, we conclude that \eqref{r-exp-conv} holds.
\end{proof}

Thus we have proved Lemma \ref{L:exp-stbl}.

\section*{Acknowledgments}
GA is supported by the Alexander von Humboldt Foundation and by the Carl Friedrich von Siemens Foundation and by JSPS KAKENHI Grant Number JP16H03946, JP18K18715, JP20H01812 and JP17H01095. He is also deeply grateful to the Helmholtz Zentrum M\"unchen and the Technishce Universit\"at M\"unchen for their kind hospitality and support during his stay in Munich. The authors would also like to acknowledge the kind hospitality of the Erwin Schr\"odinger International Institute for Mathematics and Physics, where a part of this research was developed under the frame of the Thematic Program {\it Nonlinear Flows} in 2016. CK acknowledges support via a Lichtenberg Professorship of the VolkswagenStiftung.

\appendix

\section{Proof of Lemma \ref{L:cp-sub}}\label{Apdx:S:cp}

Subtract equations for $u^-$ and $u^+$ and set $w := u^- - u^+$. It follows that
\begin{equation}\label{w-ineq}
\partial_t w - \partial_x^2 w + f(u^-)-f(u^+)\leq 0 \ \mbox{ for } \ x \geq d(t) \ \mbox{ and } \ t > 0.
\end{equation}
Test \eqref{w-ineq} by $w_+ \zeta_R^2$  where  $w^+ = \max\{w,0\}$ and  $\zeta_R$  is a smooth non-negative cut-off function such that
$$
 \zeta_R \equiv 1 \ \mbox{ on } \ (-\infty,R], \quad \zeta_R \equiv 0 \ \mbox{ on } \ [2R, +\infty), \quad \|\zeta_R'\|_{L^\infty(\R)} \leq \frac2R
$$
 for $R > 0$.  Note that 
$$
\left( f(u^-)-f(u^+) \right)w_+ = \left(\beta(u^-)-\beta(u^+)\right)w_+ - \lam w_+^2
\geq -\lam w_+^2
$$
by the monotonicity of $\beta$. The integration of both sides over $(d(t), +\infty)$ yields
\begin{align*}
\lefteqn{
\dfrac 1 2 \dfrac \d {\d t} \int^{+\infty}_{d(t)} w_+^2 \zeta_R^2 \, \d x
+ \frac{d'(t)}2 w_+(d(t),t)^2 \zeta_R^2(d(t)) - \left[(\partial_x w) w_+ \zeta_R^2 \right]^{+\infty}_{d(t)}
}\\
&\quad + \int^{+\infty}_{d(t)} |\partial_x w_+|^2 \zeta_R^2 \, \d x\\
&\leq -  2\int^{+\infty}_{d(t)} (\partial_x w) \zeta_R\zeta_R' w_+ \, \d x + \lam \int^{+\infty}_{d(t)} w_+^2 \zeta_R^2 \, \d x\\
&\leq \frac12 \int_{d(t)}^{+\infty} |\partial_x w_+|^2 \zeta_R^2 \, \d x
+ \frac{ 8}{R^2} [2R-d(t)] \|w_+\|_{L^\infty(Q^-_T)}^2 + \lam \int^{+\infty}_{d(t)} w_+^2 \zeta_R^2 \, \d x
\end{align*}
for $R > 0$ large enough. Note by assumption that $w_+(d(t),t) = 0$ and multiply both sides by $\e^{-2\lam t}$. Then
\begin{align*}
\dfrac 1 2 \dfrac \d {\d t} \left( \e^{-2\lam t} \int^{+\infty}_{d(t)} w_+^2 \zeta_R^2 \, \d x \right)
+ \frac12\e^{-2\lam t} \int^{+\infty}_{d(t)} |\partial_x w_+|^2 \zeta_R^2 \, \d x
\\
\leq \frac{ 8}{R^2} [2R-d(t)] \e^{-2\lam t} \|w_+\|_{L^\infty(Q^-_T)}^2.
\end{align*}
Integrating both sides over $(0,t)$, we infer that
\begin{align*}
\lefteqn{
\dfrac 1 2 \e^{-2\lam t} \int^R_{d(t)} w_+(\cdot,t)^2 \, \d x
+ \frac12\int^t_0 \left( \e^{-2\lam \tau} \int^{R}_{d(\tau)} |\partial_x w_+|^2 \, \d x \right) \d \tau
}\\
&\leq \dfrac 1 2 \int^{2R}_{d(0)} w_+(\cdot,0)^2 \, \d x
+ \frac8{R^2} \left( \int^t_0 \e^{-2\lam \tau} [2R-d(\tau)] \, \d \tau \right) \|w_+\|_{L^\infty(Q^-_T)}^2.
\end{align*}
Due to the fact that $w_+(\cdot,0) = 0$  a.e.~in $(d(0),+\infty)$,  one can verify that
$$
\int^t_0 \left(\e^{-2\lam \tau} \int^{R}_{d(\tau)} |\partial_x w_+|^2 \, \d x \right) \d \tau
\leq \frac{C}{R^2} \left( \int^T_0 \e^{-2\lam \tau} [2R-d(\tau)] \, \d \tau \right) \|w_+\|_{L^\infty(Q^-_T)}^2.
$$
Passing to the limit as $R \to +\infty$, we deduce that
$$
\partial_x w_+ \equiv 0 \ \mbox{ a.e.~in } Q^-_T,
$$
and hence, $w_+ \equiv 0$. This completes the proof. \qed

\section{Existence of approximate solutions}\label{S:A0}

In~\cite{ae1}, existence of $L^2$ solutions for \eqref{irAC} is proved for the Cauchy-Dirichlet problem posed on an arbitrary (smooth) \emph{bounded} domain $\Omega$ of $\R^N$, and the boundedness of the domain is only used to apply the Rellich and Aubin-Lions compactness lemma for deriving a strong compactness in $C([0,T];H)$ with $H = L^2(\Omega)$ of approximate solutions $u_\mu$ which solves the Cauchy problem for \eqref{aprx} with $\psi$ defined with $H^1(\R)$ replaced by $H^1_0(\Omega)$. However, even for unbounded domains, e.g., $\Omega = \R^N$, one can derive a strong convergence of $u_\mu$ in $C([0,T];H)$. Indeed, as in~\cite[Proof of Theorem 6.1]{ae1}, one can prove that $u_\mu$ also solves the obstacle problem,
$$
\partial_t u_\mu + \partial \psi_\mu (u_\mu) + \eta_\mu = {\bar\lam} u_\mu, \ \eta_\mu \in \partial I_{[u_0(x),+\infty)}(u_\mu) \ \mbox{ in } H.
$$
Then since $\partial I_{[u_0(x),+\infty)}$ is monotone and acting on $u_\mu$, we can verify that $(u_\mu)$ forms a Cauchy sequence in $C([0,T];H)$ based on a similar argument to~\cite[p.56]{HB1}. Therefore $u_\mu$ converges to a limit $u$ strongly in $C([0,T];H)$ as $\mu \to +0$.

\section{Maximum principle for $L^2$ solutions}\label{S:A1}

Let $u$ be a (still possibly unbounded) $L^2$ solution of \eqref{pde-obs0} in $\R \times (0,T)$ with initial data $u_0 \in C^\infty_c(\R)$ satisfying \eqref{u0n-snd}. Then noting that $u(\cdot,t) \in H^1(\R)$, one can check $(u(\cdot,t)-M_+)_+ \in H^1(\R)$ for a.e.~$t \in (0,T)$. Testing \eqref{pde-obs0} by $(u-M_+)_+$, we have
\begin{align*}
 \dfrac 1 2 \dfrac \d {\d t} \int_{\R} (u-M_+)_+^2 \, \d x
 + \int_{\R} \partial_x u \partial_x (u-M_+)_+ \, \d x
 + \int_{\R} f(u) (u-M_+)_+ \, \d x\\
 = - \int_{\R} \eta (u-M_+)_+ \, \d x.
\end{align*}
Note that
$$
f(u) (u-M_+)_+ \geq \left[f(u)-f(M_+)\right](u-M_+)_+
\geq - \lam (u-M_+)_+^2,
$$
since $f(M_+) \geq 0$ and $\beta(u) = f(u) + \lam u$ is monotone.
Moreover, it follows that
$$
 \int_{\R} \eta (u-M_+)_+ \, \d x
= \int_{\{u=u_0\} \cap \{u > M_+\}} \eta (u-M_+)_+ \, \d x = 0,
$$
since $M_+ \geq u_0$ a.e.~in $\R$. Thus noting that $\partial_x (u-M_+)_+ = \mathrm{sgn}(u-M_+) \partial_x u$, we deduce that
\begin{align*}
 \dfrac 1 2 \dfrac \d {\d t} \int_{\R} (u-M_+)_+^2 \, \d x
 \leq \lam \int_{\R} (u-M_+)_+^2 \, \d x,
\end{align*}
which along with $(u-M_+)_+|_{t=0} = 0$ by \eqref{u0n-snd} yields $(u-M_+)_+ \equiv 0$, that is, $u \leq M_+$ a.e.~in $\R \times (0,T)$. One can similarly prove that $u \geq M_-$ a.e.~in $\R \times (0,T)$. Thus \eqref{loc:u-bdd} follows.

\end{document}